\documentclass[11pt]{article}
  \usepackage{a4wide}
  \usepackage{amsmath}
  \usepackage{amssymb}
  \usepackage{latexsym}
  \usepackage[pdftex]{graphicx}
  \usepackage{subcaption}
  \usepackage{color}
  \usepackage[mathscr]{euscript}
  \usepackage{aliascnt}

  \newenvironment{proof}{\vspace{1ex}\noindent{\bf Proof.}}{\hspace*{\fill}$\blacksquare$\vspace{1ex}}
  \newenvironment{proofof}[1]{\vspace{1ex}\noindent{\bf Proof of #1.}}{\hspace*{\fill}$\blacksquare$\vspace{1ex}}

\newtheorem{theorem}{Theorem} 

\newaliascnt{lemma}{theorem}
\newtheorem{lemma}[lemma]{Lemma}
\aliascntresetthe{lemma}

\newaliascnt{corollary}{theorem}
\newtheorem{corollary}[corollary]{Corollary}
\aliascntresetthe{corollary}

\newaliascnt{claim}{theorem}

\aliascntresetthe{claim}

\newaliascnt{proposition}{theorem}
\newtheorem{proposition}[proposition]{Proposition}
\aliascntresetthe{proposition}

\newaliascnt{definition}{theorem}

\aliascntresetthe{definition}

\newaliascnt{remark}{theorem}

\aliascntresetthe{remark}

\newaliascnt{conjecture}{theorem}
\newtheorem{conjecture}[conjecture]{Conjecture}
\aliascntresetthe{conjecture}

\newaliascnt{problem}{theorem}

\aliascntresetthe{problem}

\newcommand{\Bcal}[0]{\ensuremath{{\mathcal B}}}

\newcommand{\Pcal}[0]{\ensuremath{{\mathcal P}}}
\newcommand{\Qcal}[0]{\ensuremath{{\mathcal Q}}}

\newcommand{\Vcal}[0]{\ensuremath{{\mathcal V}}}

\newcommand{\Xcal}[0]{\ensuremath{{\mathcal X}}}
\newcommand{\Ycal}[0]{\ensuremath{{\mathcal Y}}}
\newcommand{\Zcal}[0]{\ensuremath{{\mathcal Z}}}
\newcommand{\eR}[0]{\ensuremath{ \mathbb R}}

\newcommand{\eN}[0]{\ensuremath{ \mathbb N}}
\newcommand{\Zed}[0]{\ensuremath{ \mathbb Z}}

\newcommand{\norm}[1]{\ensuremath{\left\|#1\right\|}}

\newcommand{\Pee}[0]{\ensuremath{{\mathbb P}}}

\newcommand{\isd}[0]{\hspace{.2ex} \raisebox{-.1ex}{$=$} \hspace{-1.5ex} 
\raisebox{1ex}{{$\scriptstyle d$}} \hspace{.8ex} }

 \newcommand{\eps}{\varepsilon}

\newcommand{\orig}{\underline{0}}

\DeclareMathOperator{\dist}{dist}

\DeclareMathOperator{\vol}{vol}

\DeclareMathOperator{\Po}{Po}

\DeclareMathOperator{\Be}{Be}

\DeclareMathOperator{\dd}{d}

\DeclareMathOperator{\kap}{cap}
\definecolor{orange}{RGB}{255,127,0}
\definecolor{pink}{RGB}{255,150,150}

\usepackage{cleveref}
\usepackage{tikz}
\usepackage{mathtools}
\usepackage{bbm}
\usepackage{enumitem}
\usetikzlibrary{arrows.meta,calc,intersections,decorations.pathreplacing,backgrounds,shadings}   
\crefname{theorem}{Theorem}{Theorems}
\crefname{lemma}{Lemma}{Lemmas}
\crefname{corollary}{Corollary}{Corollaries} 	
\crefname{proposition}{Proposition}{Propositions}
\crefname{definition}{Definition}{Definitions}
\crefname{example}{Example}{Examples}
\crefname{remark}{Remark}{Remarks}
\crefname{equation}{Equation}{Equations}
\crefname{figure}{Figure}{Figures}
\crefname{section}{Section}{Sections}

\newcommand{\RR}{\mathbb{R}}     
\newcommand{\ZZ}{\mathbb{Z}}     
\newcommand{\NN}{\mathbb{N}}

\newcommand{\Pro}{\mathbb{P}}
\newcommand{\Ex}{\mathbb{E}}
\newcommand{\indicator}[1]{\mathbbm{1}_{\{#1\}}}     
\newcommand{\dists}{{\mathrm{dist}}}        
\newcommand{\typ}{\underline{0}}
\DeclarePairedDelimiter{\ceil}{\lceil}{\rceil}   

\newcommand{\nocon}{%
	\mathrel{\tikz[baseline=(a.base)]{
			\node[inner sep=0pt] (a) {$\longleftrightarrow$};
			\draw[line width=0.4pt]
			([xshift=-2.2pt,yshift=-4.5pt]a.center)
			--
			([xshift=1.4pt,yshift=5.5pt]a.center);
	}}%
}
\makeatletter
\newcommand*{\rom}[1]{\romannumeral #1\relax}
\makeatother

\definecolor{lightpurple}{rgb}{0.70,0.55,0.90}

\begin{document}

\title{Thresholds for colouring the random Borsuk graph}

\author{%
\'Alvaro Acitores Montero\thanks{Universitat Polit\`ecnica de Catalunya, Barcelona, Spain. E-mail: 
{\tt alvaroacitores@gmail.com}. 
Part of the work in this paper was done while this author was a visiting student at Groningen University. 
We acknowledge the CFIS Mobility Program for the funding of this visit, particularly 
Fundaci\'o Privada Mir-Puig, CFIS partners (G-Research, Qualcomm and Semidynamics), and donors of the CFIS 
crowdfunding program.} 
\and
Matthias Irlbeck\thanks{Bernoulli Institute, Groningen University, The Netherlands. E-mail: {\tt m.irlbeck@rug.nl}.} 
\and
Tobias M\"uller\thanks{Bernoulli Institute, Groningen University, The Netherlands. E-mail: {\tt tobias.muller@rug.nl}.}
\and 
Mat\v ej Stehl\'ik\thanks{Universit\'e de Paris, CNRS, IRIF, F-75006, Paris, France. E-mail: {\tt matej@irif.fr}.}
}

\date{\today}

\maketitle

\begin{abstract} 
We consider the chromatic number of the random Borsuk graph.
The random Borsuk graph is obtained by sampling $n$ points i.i.d.~uniformly at random on the $d$-dimensional
sphere $S^d$, and joining a pair of points by an edge whenever their geodesic distance is $>\pi-\alpha$ where 
the parameter $\alpha=\alpha(n)$ may depend on $n$.
Kahle and Martinez-Figueroa~\cite{KahleFig} have shown that the switch from being $(d+1)$-colourable to needing $\geq d+2$ colours 
occurs in the regime where the average degree is of logarithmic order.
We show that for each $2\leq k\leq d$, the switch from being $k$-colourable to needing $> k$ colours
occurs in the regime when the average degree is constant.
What is more, we show that for $k=2$ there is a sharp threshold of the form $\alpha(n) = c \cdot n^{-1/d}$, where 
the constant $c$ can be expressed in terms of the critical intensity for continuum AB percolation on $\eR^d$.
For $k=3,\dots,d+1$ we show that there is a sharp threshold for ``almost all $n$''.
\end{abstract}

\section{Introduction and statement of results}

The {\em Borsuk graph} with parameters $d, \alpha$ has as vertices the points of the $d$-dimensional sphere $S^d$ and 
two points $u,v \in S^{d}$ are joined by an edge if and only if $\dist(u,v) > \pi -\alpha$.
Here and in the rest of the paper

$$ \dist(u,v) := \angle u\underline{0}v, $$

\noindent
denotes the geodesic distance of the sphere, where $\typ\coloneqq(0,\dots,0)\in \RR^{d+1}$. That is, $\dist(u,v)$ is the length of the 
shortest curve between $u$ and $v$ that stays on the sphere, which happens to equal
the angle between the vectors $u,v$. 
Typically the parameter $\alpha$ is chosen small, so that points that are joined by an edge are ``nearly antipodal''.
The Borsuk graph was introduced in 1967 by Erd\H{o}s and Hajnal~\cite{ErdosHajnal67} as an explicit example of a graph with 
a large chromatic number and without short odd cycles.
In particular, Erd\H{o}s and Hajnal showed there exists an $\alpha_0=\alpha_0(d)$ such that the 
Borsuk graph has chromatic number exactly $d+2$ for all $0 < \alpha < \alpha_0$. Some thought shows that 
any odd cycle in the Borsuk graph must contain at least $\Omega(1/\alpha)$ vertices.

More recently, Kahle and Martinez-Figueroa~\cite{KahleFig} introduced the {\em random Borsuk graph}, which we
will denote by $G_d(n,\alpha)$. 
It has $n$ vertices $X_1,\dots,X_n \in S^d$, chosen i.i.d.~uniformly at random, and 
an edge $X_iX_j$ if and only if $\dist(X_i,X_j) > \pi -\alpha$.
When the dimension $d$ is clear from the context, we will suppress the 
subscript and just write $G(n,\alpha)$.
Recently the random Borsuk graph featured in a breakthrough result giving an exponential lower bound 
for asymmetric Ramsey numbers by Ma et al.~\cite{MaEtalArXiv}.
(But apparently the authors were not aware that the model has been studied before under the 
name random Borsuk graph.)

\begin{figure}[h!]
    \begin{center}
    \includegraphics[width=0.5\linewidth]{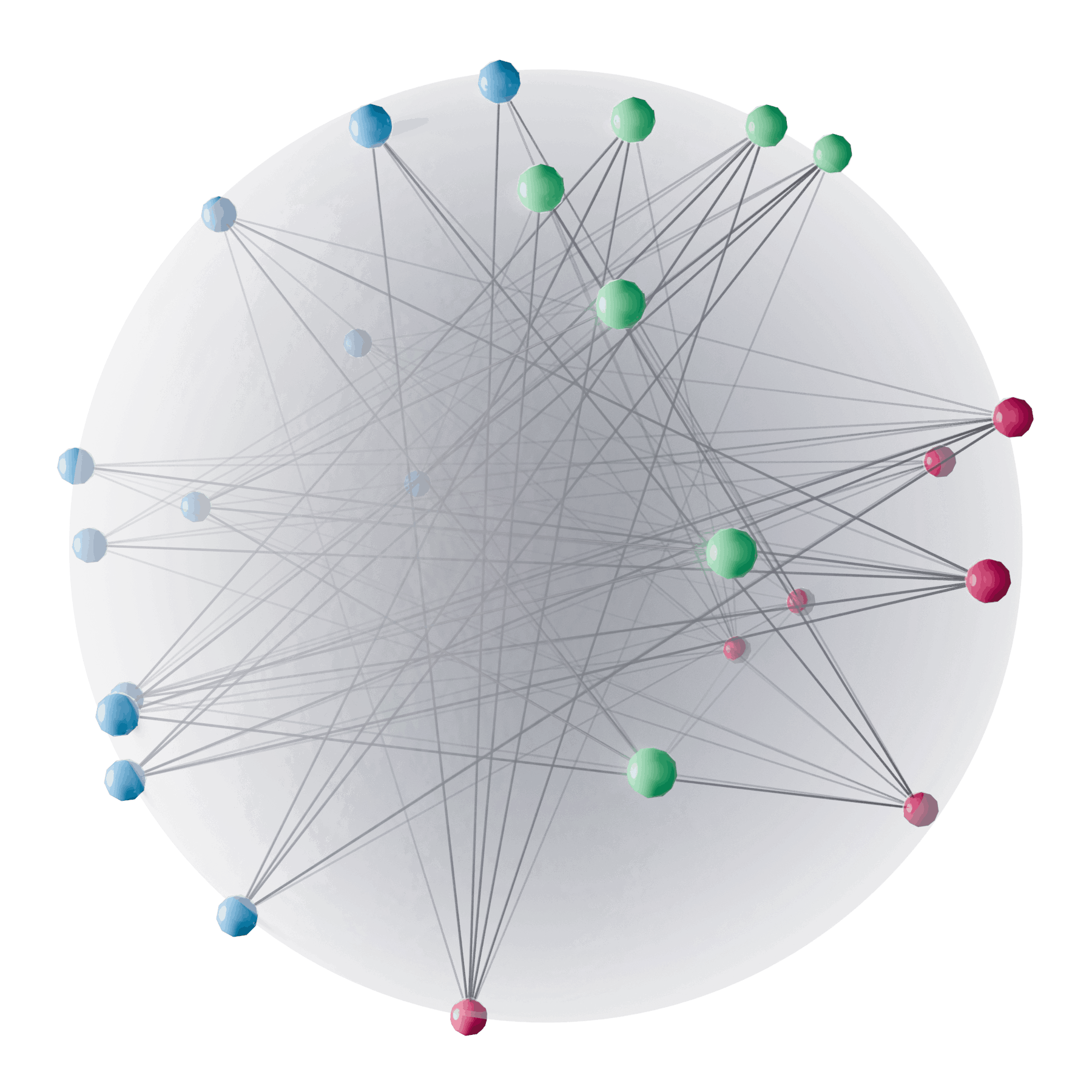}
    \caption{Visualization of the random Borsuk graph with $d=2, n=32, \alpha=1$ and 
    a proper vertex colouring with three colours.\label{fig:placeholder2}}
    \end{center}
\end{figure}

We should perhaps point out that our definition of the (random) Borsuk graph is different from, but equivalent to, the one 
used by Erd\H{o}s and Hajnal~\cite{ErdosHajnal67} and Kahle and Martinez-Figueroa~\cite{KahleFig}.
In those works, instead of the parameter $\alpha$ there is a 
parameter $\eps>0$ and we include an edge between the vertices $u$ and $v$ iff.~$\norm{u-v} > 2-\eps$.
We choose a formulation in terms of the geodesic distance instead, that is also used by other 
authors (e.g.~\cite{AdamsEtal}).
We find this setup more convenient to work with. What is more, the parameter $\alpha$ behaves analogously to the parameter $r$ in the 
standard definition of the random geometric graph (see e.g.~\cite{penroseboek}).
For instance, the average degree scales with $n r^d$ in the random geometric graph and with $n \alpha^d$ in the random Borsuk 
graph.

Kahle and Martinez-Figueroa have shown that (translated into our setup) 
there exist constants $c_L,c_U$ satisfying

\begin{equation}\label{eq:KF} 
\lim_{n\to\infty} \Pee\left( \chi\left( G(n,\alpha) \right) > d+1 \right) =
 \begin{cases}
 1 & \text{ if $\alpha > c_U\cdot(\ln n/n)^{1/d}$, } \\
 0 & \text{ if $\alpha < c_L\cdot(\ln n/n)^{1/d}$. }
 \end{cases}
\end{equation}

\noindent
As pointed out by Kahle and Martinez-Figueroa, this also gives that if $\alpha_0 > 0$ is the constant 
provided by the result of Erd\H{o}s and Hajnal for the (non-random) Borsuk graph, the chromatic number 
of $G(n,\alpha)$ is precisely $d+2$ with probability $1-o_n(1)$ whenever
$c_U (\ln n/n)^{1/d} < \alpha < \alpha_0$.

The results of Kahle and Martinez-Figueroa show that the random Borsuk graph transitions from 
being $(d+1)$-colourable to needing $\geq d+2$ colours in the regime where the average degree is of order $\ln n$.
Our first result shows that the random Borsuk graph achieves chromatic number $\geq d+1$ much ``sooner'', namely when the
average degree is constant.

\begin{theorem}\label{thm:maind+1}
For $d\geq 2$, there is a constant $c=c(d)>0$ such that:
 
 $$ \lim_{n\to\infty} \Pee\left( \chi\left( G(n, c \cdot n^{-1/d}) \right) > d \right) = 1.$$
 
\end{theorem}

\noindent
So in particular the chromatic number of $G(n,\alpha)$ equals $d+1$ with probability $1-o_n(1)$ whenever 
$c \cdot n^{-1/d} \leq \alpha < c_L \cdot (\ln n / n)^{1/d}$.

The regime when the average degree is constant is sometimes called the ``thermodynamic regime'' in the random (geometric) graphs 
literature.
We will show that the switch from being $k$ colourable to needing $>k$ colours 
occurs in the thermodynamic regime, for all $2\leq k \leq d$. 
In fact, for $k=2$ we can be quite precise:

\begin{theorem}\label{bt1}
For $d\geq 2$, there is a constant $c_2=c_2(d)>0 $ such that for all fixed $\eps>0 $ and all sequences $\alpha=\alpha(n)$ 
it holds that
\begin{equation*}
\lim_{n\to\infty}\Pro(\chi(G({n},\alpha)) > 2)=\begin{cases}
1 & \text{if }\alpha\geq (c_2+\eps)\cdot n^{-1/d},\\
0 & \text{if }\alpha\leq (c_2-\eps)\cdot n^{-1/d}.
\end{cases}
\end{equation*}
\end{theorem}

\noindent
The proof provides an expression for $c_2$ in terms of the critical intensity for 
{\em continuum AB percolation}.

Since having chromatic number one is the same as having no edges, the switch from chromatic number one to 
at least two occurs well before the thermodynamic regime.
For completeness we offer the following precise result, that has a relatively straightforward proof 
(which we give in Section~\ref{sec:bt2}).

\begin{theorem}\label{bt2}
We have

$$
\lim_{n\to\infty}\Pro(\chi(G({n},\alpha)) > 1)
=\begin{cases}
    1 & \text{if }n^2\alpha^d\to \infty,\\
    1-e^{- c \cdot \nu} & \text{if }n^2\alpha^d\to \nu \in (0,\infty), \\
    0 & \text{if } n^2\alpha^d\to 0,
\end{cases}
$$

\noindent
where $c := \Gamma((d+3)/2) / \left( 2 \cdot (d+1)\cdot \sqrt{\pi} \cdot \Gamma((d+2)/2) \right)$.
\end{theorem}

The behaviour exhibited by Theorem~\ref{bt1} is a prime example of what is called a {\em sharp threshold} in the 
random graphs community, and Theorem~\ref{bt2} corresponds to what is called a {\em coarse threshold}.
If $\Pcal$ is a monotone graph property (formally a family of graphs closed under isomorphism and under the 
addition of edges) then we say $\beta = \beta(n)$ is a {\em threshold} for $\Pcal$ if
$\lim_{n\to\infty}\Pee( G(n,\alpha) \in \Pcal ) = 1$ for all sequences $\alpha=\alpha(n)$ 
satisfying $\alpha \gg \beta$ and 
$\lim_{n\to\infty}\Pee( G(n,\alpha) \in \Pcal ) = 0$ whenever $\alpha \ll \beta$. It is a 
{\em sharp threshold} if $\Pee( G(n,\alpha) \in \Pcal ) \to 1$ whenever $\alpha \geq (1+\eps) \beta$ and 
$\Pee( G(n,\alpha) \in \Pcal ) \to 0$ whenever $\alpha \leq (1-\eps)\beta$ for all $\eps>0$ 
(but where we keep $\eps$ fixed as we take the limit $n\to\infty$).
A {\em coarse threshold} is simply a threshold which is not sharp.
The situation in~\eqref{eq:KF} is sometimes called a {\em semi-sharp threshold}.
For $k=2,\dots,d$, the threshold for $k$-colourability is also semi-sharp, as a 
consequence of Theorems~\ref{thm:maind+1} and~\ref{bt1}.
The semi-sharpness of these thresholds leaves open the possibility that they are in fact sharp.
We conjecture that this is the case. 
(See Section~\ref{sec:discuss} for more discussion and more open problems.)
We have not been able to prove our conjecture, but we can offer the following weaker result.

\begin{theorem}\label{thm:main}
 For $d\geq 2$, here is a set $N \subseteq \eN$ of density one and 
 sequences $\alpha_k = \alpha_k(n)$ for $k=3,\dots,d+1$, such that for every fixed $\eps>0$, 
 every sequence $\alpha=\alpha(n)$ and every increasing sequence $n_1,n_2,\dots \in N$: 
 
 $$ \lim_{i\to\infty} \Pee\left( \chi\left( G(n_i,\alpha) \right) > k \right) =
 \begin{cases}
 1 & \text{ if $\alpha > (1+\eps)\alpha_k$, } \\
 0 & \text{ if $\alpha < (1-\eps)\alpha_k$. }
 \end{cases} $$
 
\end{theorem}

\noindent
Here $N$ having density one means that 

$$ \lim_{n\to\infty} \frac{|N \cap [n]|}{n} = 1, $$
where for a positive integer $n \in \eN$ we write $[n] := \{1,\dots, n\}$.
One could paraphrase the result as 
``the property of having chromatic number $> k$ has a sharp threshold for {\em almost all} $n$''. 

Kahle and Martinez-Figueroa (\cite{KahleFig}, item 2 in Section 4) asked whether, for each $1\leq k \leq d+1$ there exists a
sequence $\alpha=\alpha(n)$ such $\chi( G(n,\alpha) ) = k$ with probability $1-o_n(1)$.
Our results show that this is true for $k \in \{1,2,d+1\}$ :
for $k=1$ we can take any $\alpha$ satisfying $\alpha \ll n^{-2/d}$;
for $k=2$ we can take any $\alpha$ satisfying $n^{-2/d} \ll \alpha \ll n^{-1/d}$; 
for $k=d+1$ we can take any $\alpha$ satisfying $n^{-1/d} \ll \alpha \ll (\ln n/n)^{1/d}$.
For the remaining values of $k$ our results and proof methods leave us in the dark.
We conjecture that the property of having chromatic number $> k$ has a sharp threshold 
of the form $c_k \cdot n^{-1/d}$ for each $3\leq k \leq d$, where $c_2<c_3<\dots<c_{d}$ (with $c_2$ 
as provided by Theorem~\ref{bt1}). 
This would in particular provide a positive answer to the question of Kahle and Martinez-Figueroa.
As mentioned previously, see Section~\ref{sec:discuss} for more discussion and open problems.

\vspace{1ex}

{\bf Related work.} For context, we give a very brief and very incomplete overview of related results for 
the Erd\H{o}s-R\'enyi random graph and random geometric graph model.
The study of thresholds for monotone properties in general has played a
central role in literature on the Erd\H{o}s-R\'enyi random graph model.
Landmarks include the work of Bollob\'as and Thomason~\cite{BollobasThomason87} who 
observed that every monotone property has a threshold, the work of
Friedgut and Kalai~\cite{FriedgutKalai96} showed that any monotone property will transition from 
holding with probability $o_n(1)$ to probability $1-o_n(1)$ inside an interval for the parameter $p$
of length $O(1/\ln n)$ and Friedgut's seminal paper~\cite{Friedgut99} that essentially gave a characterization of the 
graph properties that have a sharp threshold. 
For more background on thresholds for monotone properties in the Erd\H{o}s-R\'enyi model 
we refer the reader to~\cite{Friedgut_hunting,Perkins_survey} and the references therein.
For random geometric graphs, McColm~\cite{McColm_threshold} established an analogue of the result of Bollob\'as and Thomason, but 
only in dimension $d=1$. For other dimensions it remains an open problem whether every 
monotone graph property has a threshold. Goel et al.~\cite{GoelEtal} proved a result analogous to the one 
of Friedgut and Kalai. 

Perhaps even more than monotone properties, the study of the chromatic number of random graphs has had a huge influence
on the development of the field. 
It was already discussed in the founding work of Erd\H{o}s and R\'enyi~\cite{ErdosRenyi60}.
In 1975 Grimmett and McDiarmid~\cite{GrimmettMcDiarmid75} gave upper and lower bounds for the case of constant $p$ that were a factor two 
apart, and in a seminal 1988 paper, Bollob\'as~\cite{Bollobas88} was able to show the probability mass 
of the chromatic number is concentrated near the lower bound of Grimmett and 
McDiarmid. More detailed questions on the 
distribution of the chromatic number for the Erd\H{o}s-R\'enyi random graph however remained open even for constant $p$, and 
the research line is actively pursued to this day. A recent highlight is for instance~\cite{Heckel21}.

For $p = p(n)$ tending to zero with $n$ faster than $n^{-5/6-o_n(1)}$, {\L}uczak~\cite{Luczak91} showed that the 
distribution of the chromatic number has almost all of its probability mass concentrated on 
two consecutive integers $k,k+1$, where $k=k(n,p)$ depends on both $n$ and $p$. Alon and Krivelevich~\cite{AlonKrivelevich} subsequently 
extended the range of $p$ where {\L}uczak's result holds.
Achlioptas and Friedgut~\cite{AchlioptasFriedgut} showed that for each fixed $k\geq 3$, the property of 
being $k$-colourable has a sharp threshold in the regime when the average degree is constant, and 
subsequently Achlioptas and Naor~\cite{AchlioptasNaor} obtained more detailed information on 
the shape of the threshold and on the value of the aforementioned $k(n,p)$ in the 
case when $p = \text{const} \cdot n^{-1}$.
It is widely believed in the community, but remains unproven in general, that the (sharp) threshold for $k$-colourability 
of the Erd\H{o}s-R\'enyi random graph is of the form $p = \text{const} \cdot n^{-1}$.
In this light, our Conjecture~\ref{conj:conjk} seems (even more) natural.
For random geometric graphs, the chromatic number was studied in~\cite{McDiarmidMuller, Mullertwopoint, penroseboek}.
In particular, the third author of the current paper showed an analogue of {\L}uczak's two-point concentration result : 
if the average degree is not too large, then the probability mass of the chromatic number is concentrated on 
two consecutive integers. However, in the thermodynamic regime (when the average degree is constant) the
value of these two integers tends to infinity with $n$, so that the behaviour is rather different 
from what we observe in the current paper for the random Borsuk graph, and the mentioned results
on the Erd\H{o}s-R\'enyi graph. Intuitively speaking, the reason is that the chromatic number 
for random geometric graphs is ``forced'' by groups of points that are close together in 
the ambient space. In particular, in the thermodynamic regime the random geometric graph even has cliques of
growing size, while the random Borsuk graph does not have cliques of size $>2$.

\subsection{The structure of the paper.}

In Section~\ref{sec:informal} we give an informal description of some of the intuitions and strategies guiding 
the proofs of our main results. 
Next, in Section~\ref{sec:prelim} we list some notations, definitions and results from the literature that we
will be using in the rest of the paper.
Section~\ref{sec:maind+1} contains the proof of Theorem~\ref{thm:maind+1}, Section~\ref{sec:main} contains the proof of
Theorem~\ref{thm:main}, Section~\ref{sec:bt1} contains the proof of Theorem~\ref{bt1}, and Section~\ref{sec:bt2} 
contains the proof of Theorem~\ref{bt2}.
We include some more discussion and suggest some open problems in Section~\ref{sec:discuss}.
We conclude the paper with an appendix containing proofs of some auxiliary results that we include for completeness.

\subsection{Informal description of the proofs.\label{sec:informal}}

In this section, we describe some of the main intuitions and strategies guiding the proofs in an informal, non-rigorous 
manner in order to aid the reader in digesting the paper more easily.

\vspace{1ex}

{\bf Theorem~\ref{thm:maind+1}.} The strategy for the proof of Theorem~\ref{thm:maind+1} is to show that, when the average 
degree is a large constant, we can identify a subgraph of the random Borsuk graph that has chromatic number (at least) $d+1$.
More specifically, we will exhibit a subgraph that behaves like a Borsuk graph but in dimension $d-1$ rather than $d$.
The technique for lower bounding the chromatic number used by Kahle and Martinez-Figueroa involves showing that
the spherical caps (balls with respect to $\dist$, see Section~\ref{sec:geoprel} for a detailed definition)
with opening angle $\alpha/2$ around the antipodes of the random points cover the sphere, and then apply
a version of the Lyusternik-Shnirelmann theorem. See Lemma~\ref{lem:KFchi} below.
A candidate subgraph that we might want to try this technique on is 
the one induced by all points within a small distance (of order a small constant times $\alpha$, say) of the 
equator $S^{d-1} \times \{0\}$. 
Unfortunately that will not work however, because in the thermodynamic regime, even when the 
average degree is a large constant, there will be lots of ``empty patches'' on the sphere 
with diameter of order $\alpha$ and larger that are completely devoid of points. 
In particular the equator will not be covered by the aforementioned spherical caps around antipodes of points.

To get around this, we find a suitable embedding $\varphi : S^{d-1} \to S^d$ that stays close to the equator, avoids 
empty patches, is nearly antipodal (satisfies $\varphi(-u) \approx -\varphi(u)$) and is ``locally nearly flat''. 
To construct this embedding, we find it technically easier to work in $\eR^d$ rather than $S^d$. So we 
apply stereographic projection to obtain a random set of points in $\eR^d$ and 
construct a map $h : S^{d-1} \to (-1,1)$ such that 
for every $u \in S^{d-1}$ one of the random points is within distance $\eps\cdot\alpha$ of $\psi(u) := (1+h(u))\cdot u$, and 
moreover $h(u)+h(-u)=0$, and $h$ is Lipschitz with an appropriate constant (so that it ``looks flat'' at the scale $O(\alpha)$).
The construction of $\psi$ proceeds iteratively. 
Roughly speaking, we order the ``bad patches'' by scale. We start by avoiding the largest patches. 
Each iteration we add a small perturbation in order to also avoid patches on the next smaller scale.
The perturbations are made in such a way that on a local scale the function $h$ changes very gradually.

We now consider the subgraph spanned by all (stereographically projected) points within distance $\eps\alpha$ 
of the image of $\psi$.
Because of the careful choice of $h$, it turns out that this subgraph itself contains an induced
subgraph of the $(d-1)$-dimensional Borsuk graph to which the lower bound technique of Kahle and Martinez-Figueroa 
(stated as Lemma~\ref{lem:KFchi} below) applies.

\vspace{1ex}

{\bf Theorem~\ref{bt1}.}
The main intuition behind the proof of Theorem~\ref{bt1} is that locally the random Borsuk graph in the 
thermodynamic regime behaves like the continuum AB percolation model.
This model, first introduced by Iyer and Yogeshwaran in~\cite{IyerYogesh}, is defined by taking two 
independent, homogeneous Poisson processes $\Zcal_A$ and $\Zcal_B$ on $\eR^d$, the 
first having intensity $\lambda_A$ and the second $\lambda_B$, and joining each point 
of $\Zcal_A$ to any point of $\Zcal_B$ within distance $1$. (There are no edges with 
both endpoints in $\Zcal_A$ or both endpoints in $\Zcal_B$.)
Let us consider a spherical cap $C$ of opening angle some large constant $K$ times $\alpha$. 
In the thermodynamic regime, as $n\to\infty$, up to rescaling by $1/\alpha$, the geometry of $C$ 
behaves like that of a ball of radius $K$ in $\eR^d$. 
And, under the same rescaling, the set $\Xcal_A$ of all random points that fall inside the cap $C$ behaves like a constant 
intensity Poisson process restricted to this ball of radius $K$. 
Now also consider the set of points that fall in $-C$. Their antipodes, which we 
shall denote by $\Xcal_B$, lie in $C$ and, similarly to before, under the appropriate rescaling $\Xcal_B$ behaves like 
a constant intensity Poisson process, independent of the one describing $\Xcal_A$ and of the same intensity. 
What is more, the subgraph of the Borsuk graph spanned by the 
points in $C \cup -C$ is approximately the same as the graph we get if we join a point of the first 
Poisson point process to all points of the second Poisson point process at distance at most one.
That is, the ``local picture'' of the random Borsuk graph is that of 
the continuum AB percolation model with $\lambda_A = \lambda_B$, restricted to some large ball.
Being a bit more careful, one can see that the intensity in fact scales with $\alpha$ (the larger $\alpha$, 
the larger the intensity).
The constant $c_2$ in Theorem~\ref{bt1} corresponds to the {\em critical value} $\lambda_c$ of the continuum AB percolation.
That is, if $\lambda_A=\lambda_B < \lambda_c$ all clusters (connected components) of the model  are finite, and for 
$\lambda_A=\lambda_B > \lambda_c$ there is an infinite cluster -- both statements holding with probability one
and with respect to the model that is not restricted to a ball but lives on all of $\eR^d$.

Recall that a graph is bipartite (two-colourable) if and only if there is no odd cycle, and that 
any odd cycle in the Borsuk graph must have at least $\Omega(1/\alpha)$ (which is $\Omega( n^{1/d} )$ 
if we are in the thermodynamic regime) vertices.
That the random Borsuk graph is two colourable when we choose $\alpha$ in such a way that the corresponding continuum AB percolation 
model is subcritical follows relatively easily once we have established that 
cluster sizes decay exponentially in the subcritical regime of continuum AB percolation.
Even though continuum AB percolation has been considered in a few papers (e.g.~\cite{PenroseAB,DereudrePenrose}) 
since the work of Iyer and Yogeshwaran, the literature does not provide the subcritical exponential decay result we need. 
(We work in the specific situation when $\lambda_A=\lambda_B$, but the focus of the cited works is on 
other choices of the parameters.) We end up adapting various arguments from the percolation literature
to the continuum AB percolation setting.

Showing that if $\alpha$ is such that the corresponding continuum AB percolation model is supercritical 
then the random Borsuk graph contains an odd cycle is (even) more involved.
As explained, locally the random Borsuk graph is well-approximated by the continuum AB percolation model.
The challenge is to patch together many of these local pictures into a global one where we are sure 
that a (necessarily macroscopic) odd cycle exists with probability $1-o_n(1)$.
In order to achieve this we adapt the ``sprinkling technique'' of 
Grimmett and Marstrand~\cite{GrimmettMarstrand1990Supercritical}, developed for their famous 
``percolation on slabs'' result.

\vspace{1ex}

{\bf Theorem~\ref{bt2}.} As mentioned before, Theorem~\ref{bt2} follows via a relatively straightforward argument. 
We consider the range of $\alpha$ where the expected number of edges is 
constant, and use the ``method of moments'' to show that the number of edges converges in distribution to a 
Poisson random variable.

\vspace{1ex}

{\bf Theorem~\ref{thm:main}.}
The proof of Theorem~\ref{thm:main} leverages a ``sharp thresholds for Boolean functions'' result from~\cite{MullerConfetti}, stated as
Proposition~\ref{prop:BourTobias} below.
We use this result as follows. Rather than choosing precisely $n$ points uniformly at random on $S^d$, we 
switch to considering a Poisson point process on $S^d$ with $\mu = n$ points in expectation.
Lemma~\ref{lem:transferPo} below establishes that it is enough to prove the result for this ``Poissonized'' version 
of the model.
We choose $\alpha_k=\alpha_k(n)$ such that the probability that the 
(Poissonized) random Borsuk graph with parameter $\mu=n$ and $\alpha=\alpha_k$ has chromatic number $> k$ is precisely
$1/2$.
We discretize the model, choosing a dissection of $S^d$ into sets $C_1,\dots, C_m$ of small diameter, and for $i=1,\dots,m$ 
we take $Z_i$ to be the indicator variable that at least one of the random points falls inside $C_i$.
Having chosen the sets $C_i$ sufficiently small, the event that the (Poissonized) random Borsuk graph 
has chromatic number $> k$ is well-approximated by an increasing event defined in terms of $Z_1,\dots, Z_m$.
The aforementioned sharp thresholds result then allows us to conclude that (the discretized version of) the 
event that the chromatic number of the Poissonized random Borsuk graph is $> k$ has a sharp threshold 
in the parameter $\mu$ (the expected total number of points of the Poissonized random graph).
In slightly more detail : the sharp thresholds result implies that if there were no 
sharp threshold, then we could (deterministically) include a bounded number of points in specific locations in the 
vertex set and increase $\alpha$ very slightly to  
boost the probability of having chromatic number $> k$ to close to one. 
Increasing the expected number of points from $\mu$ to $(1+\eps)\mu$
corresponds to adding another Poisson process with $\eps\mu$ points in expectation that is independent of the 
original process. Such a 
Poisson process will contain a set of points that ``looks like'' the points in specific locations 
provided by the sharp threshold result.

Finally we use~\eqref{eq:KF}, respectively Theorems~\ref{thm:maind+1} and~\ref{bt1}, to show that 
the sequence $\alpha_k=\alpha_k(n)$ does not behave too ``wildly''.
More specifically, for most $n$ the ratio of the values $\alpha_k(n)$ and $\alpha_k( \lceil(1+\delta)n\rceil)$
is close to one (but we cannot exclude the possibility that occasionally there are ``jumps'').
For those values of $n$ where the ratio is close to one, we are able to transfer the sharp threshold in terms 
of the expected number of points to one in terms of $\alpha$.


\section{Notation and preliminaries\label{sec:prelim}}

Here we list some notations, definitions and results from the literature that we will use in the proofs.

\subsection{Geometric preliminaries\label{sec:geoprel}}

We will use $B(x,r) := \{ y \in \eR^d : \norm{x-y} < r \}$ to denote the open ball of radius $r$ around $x$.
We use $\vol(\cdot)$ to denote the $d$-dimensional volume (Lebesgue measure) and 
we denote by $\kappa_d := \vol( B(\underline{0}, 1 ) )$ the volume of the $d$-dimensional unit ball.
That is,

\begin{equation}\label{eq:volball} 
\kappa_d := \vol( B(\orig,1) ) = \frac{\pi^{d/2}}{\Gamma(d/2+1)},  
\end{equation}

\noindent
where the stated equality is a classical result (see e.g.~\cite{Schilling}, Corollary 15.15, for a proof).
The {\em spherical cap} around $u \in S^d$ of opening angle $\alpha$ is:

$$ \kap(u,\alpha) := \{ x \in S^d : \dist(u,x) < \alpha \}. $$

For all $u, v \in S^d$ the Euclidean distance and spherical distance are related by the 
following identity:

$$ \norm{u-v} = 
2 \cdot \sin\left(\dist(u,v)/2\right). $$

\noindent
(As can be seen by considering the triangle with 
corners $\orig, u, v$.)
For convenient future reference, we point out the following consequence:

\begin{equation}\label{eq:distnorm} 
(2/\pi) \cdot \dist(u,v) \leq \norm{u-v} \leq \dist(u,v),  
\end{equation}

\noindent
for all $u,v \in S^d$. 
(We use that $\alpha \geq \sin\alpha \geq (2/\pi)\cdot\alpha$ for $0 \leq \alpha \leq \pi/2$.)

\vspace{.5ex}

We denote by $\pi : S^d \setminus \{(0,\dots,0,1)\} \to \eR^d$ the $d$-dimensional stereographic projection from the 
``north pole''. 
It is given explicitly by 

$$ \pi(x_1,\dots,x_{d+1}) := \left(\frac{x_1}{1-x_{d+1}},\dots,\frac{x_d}{1-x_{d+1}}\right). $$

\noindent 
As is easily checked, $\pi$ is a diffeomorphism (a smooth bijection whose inverse is also smooth).
In several of our arguments below, we will want to compare the geodesic distance between two points $a, b \in S^d$ with the 
Euclidean metric between their images $\pi(a),\pi(b) \in \eR^d$.
The following bounds suffice for our purposes.

\begin{lemma}\label{lem:metric}
$\text{ }$
\begin{enumerate}
 \item\label{itm:metricgenlb} For all $a, b \in S^d \setminus \{(0,\dots,0,1)\}$ we have 
 
 $$ \norm{\pi(a)-\pi(b)} \geq (1/2) \cdot \dist(a,b). $$
 
 \item\label{itm:metricub12} For every $\eps > 0$ there is a $\delta=\delta(\eps) > 0$ such that 
 
 $$ \norm{\pi(a)-\pi(b)} \leq (1/2+\eps)\cdot \dist(a,b), $$
 
 \noindent 
 for all $a,b \in \kap( (0,\dots,0,-1), \delta)$.
 
 \item\label{itm:metricgenub} For all $a, b \in S^d \cap \{x_{d+1} \leq t\}$ it holds
 
 $$ \norm{\pi(a)-\pi(b)} \leq \frac{2-t}{(1-t)^2} \dist(a,b). $$
 
\end{enumerate}
\end{lemma}

\noindent
We have not found a convenient reference to the literature.
The proof is relatively straightforward but pedestrian and a bit tedious. In order not to disrupt the flow of the paper
we defer the proof to Appendix~\ref{sec:metric}.

The following result is certainly known in the literature. 
We have not found a convenient reference with a stand-alone proof, but it can for instance be deduced 
from Equation 3.7 in \cite{lee2006riemannian}. 
We give a proof in Appendix~\ref{proof:piunif} for completeness.  

\begin{lemma}\label{lem:piunif}
Let $U \isd \text{Unif}(S^d)$ and $Z := \pi(U)$.
Then the pdf of $Z$ is given by

$$ f_Z(z) = \frac{1}{(d+1)\kappa_{d+1}} \cdot \left(\frac{2}{1+\norm{z}^2}\right)^d. $$

\end{lemma}

\noindent
Here and in the rest of the paper, $\text{Unif}(A)$ denotes the uniform distribution on the set $A$.

\subsection{Some more ingredients from probability theory}

For $\overline{p} = (p_1,\dots, p_n) \in (0,1)^n$ the notation $\Pee_{\overline{p}}(.)$ will 
signify the situation where $X_1, \dots, X_n$ are independent random variables with 
$X_i \isd \Be(p_i)$.
Observe that, for every $A \subseteq \{0,1\}^n$, the probability
$\Pee_{\overline{p}}[ (X_1,\dots, X_n ) \in A ]$ can be written as a
polynomial in $p_1,\dots, p_n$.
In particular, this probability is a continuous function of the 
$p_i$-s and the partial derivatives $\frac{\partial}{\partial p_i}\Pee_{\overline{p}}{\big[} (X_1,\dots, X_n) \in A {\big]}$
exist.
Recall that a set $A \subseteq \{0,1\}^n$ is called an up-set if 
$a \in A$ implies that also $(a_1,\dots,a_{i-1},1,a_{i+1},\dots,a_n) \in A$ for all $1\leq i \leq n$.
In other words, if we change a zero coordinate into a one then we do not leave the set $A$.
We remark that if $A$ is an up-set then $\Pee_{\overline{p}}[ (X_1,\dots, X_n ) \in A ]$ is non-decreasing 
in each parameter $p_i$.

The following result from~\cite{MullerConfetti} can be considered as an asymmetric version 
of Bourgain's powerful sharp threshold result (that appeared in the appendix of Friedgut's influential paper~\cite{Friedgut99}).

\begin{proposition}\label{prop:BourTobias}\cite{MullerConfetti}
 For every $C > 0$ and $0 < \eps < 1/2$ there exist $K = K(C,\eps) \in\eN, c = c(C,\eps) > 0$ such that
 the following holds, for every $n \in \eN$ and every up-set $A \subseteq \{0,1\}^n$. \\
 If $\overline{p} \in (0,1)^n$ is such that $\Pee_{\overline{p}}{\big[ }(X_1,\dots, X_n) \in A {\big]} \in (\eps,1-\eps)$
and 
\[  \sum_{i=1}^n p_i(1-p_i) \cdot \frac{\partial}{\partial p_i}\Pee_{\overline{p}}{\big[} (X_1,\dots, X_n) \in A {\big]} \leq C, 
 \]%
then there exist indices $i_1, \dots, i_K \in \{1,\dots, n\}$ such that one of the following holds:
\begin{enumerate}
 \item[{\bf(a)}] $\Pee_{\overline{p}}{\big[} (X_1, \dots, X_n) \in A {\big|} X_{i_1} = \dots = X_{i_K} = 1 {\big]}
 \geq \Pee_{\overline{p}}{\big[} (X_1, \dots, X_n) \in A {\big]} + c$, or
 \item[{\bf(b)}] $\Pee_{\overline{p}}{\big[} (X_1, \dots, X_n) \in A {\big|} X_{i_1} = \dots = X_{i_K} = 0 {\big]}
 \leq \Pee_{\overline{p}}{\big[} (X_1, \dots, X_n) \in A {\big]} - c$.
\end{enumerate}
\end{proposition}

\noindent
For $ m\in \NN $ and $ x\in \RR $ we will write  $ (x)_m:=x(x-1)\dots (x-m+1) $ for the {\em falling factorial}. 
For a proof of the following classical result see e.g. Theorem 1.22 in \cite{Bollobas2001RandomGraphs}.

\begin{lemma}\label{fallingmethod}
	Assume that for some fixed $ \mu>0 $ the sequence of random variables $ X_1,X_2,\dots $ satisfy for all $ m\in \NN $
	\begin{equation*}
		\lim_{n \to \infty} \Ex (X_n)_m = \mu^m.
	\end{equation*}
	Then, $ X_n $ converges to $ \emph{\text{Poi}}(\mu) $ in distribution as $ n\to\infty $.
\end{lemma}

\subsection{Ingredients from percolation theory}

We are going to make use of two percolation results for the 2-colourability threshold. 
One in the supercritical regime and the other subcritical. The following is well known, but,
since we did not find a convenient reference, we provide a proof in Appendix~\ref{sec:depperco} for completeness.

\begin{lemma}\label{dependentperco}
	Let $ C_m(\typ) $ be the cluster of $ \typ $ in bond percolation restricted to $ \Lambda_m\coloneqq \{-m,\dots,m \}^d $ with parameter $ p $. Then, for all $ 0<\eps <1/10$ we can choose $ p=p(\eps,d) $ large enough such that for all large enough $ m\geq m_0(\eps,d,p) $ 
	\begin{equation*}
		\Pro_p(|C_m(\typ)|>(1-\eps)|\Lambda_m|)>1-\eps.
	\end{equation*}
\end{lemma}

\noindent
 A proof for the following subcritical exponential decay result can be found at Theorem 10 in Section 4.5 of \cite{bollobas2006percolation}. 
\begin{lemma}\label{bolloexpdecay}
    Let $G$ be a connected, infinite, transitive, locally-finite graph with $|B_r(x)|\leq r^{(\log r)/100} $ for all $x\in G$, where $B_r(x)$ is the ball with radius $r$ w.r.t. the graph distance around $x$. Then, if we consider site percolation on $G$ with $p<p_c(G)$ there exists a constant $c=c(p,G)>0$ such that for all $n\in \NN$
    \begin{equation*}
        \Pro_p(|C_x| \geq n)\leq e^{-cn},
    \end{equation*}
    where $C_x$ is the connected component of $x$.
\end{lemma}

\subsection{Preliminaries on (random) Borsuk graphs}

For $V \subseteq S^d$, we denote by $G(V,\alpha)$ the subgraph of the Borsuk graph (with parameter $\alpha$) induced
by the vertices $V$.
Throughout the rest of the paper, we let $X_1,X_2,\dots$ be i.i.d.~uniform on $S^d$ and 
let $N \isd \Po(n)$ be a Poisson distributed random 
variable, independent of the $X_i$-s. 
We consider the Poissonized Borsuk graph $G(\Xcal,\alpha)$ where $\Xcal = \{X_1,\dots,X_N\}$.
It suffices to prove Theorems~\ref{thm:maind+1},~\ref{bt1} and~\ref{thm:main} for $G(\Xcal,\alpha)$ instead of 
$G(n,\alpha) = G(\{X_1,\dots,X_n\},\alpha)$ because:

\begin{lemma}\label{lem:transferPo}
For every $k \in \eN$, we have 
\begin{enumerate}
\item If $\Pee_n(\chi(G(\Xcal,\alpha))> k) = 1-o_n(1)$ then also 
$\Pee(\chi(G(n,\alpha))> k) = 1-o_n(1)$, and;
\item If $\Pee_n(\chi(G(\Xcal,\alpha))> k) = o_n(1)$ then also 
$\Pee(\chi(G(n,\alpha))> k) = o_n(1)$.
\end{enumerate}
\end{lemma}

\begin{proof} By obvious monotonicity, we have 

$$ \begin{array}{rcl} 
\Pee_n(\chi(G(\Xcal,\alpha))> k)
& = & \sum_{m} \Pee( \chi(G(\Xcal,\alpha))> k | N = m ) \Pee( N=m )\\
& = & \sum_{m} \Pee( \chi(G(m,\alpha))> k ) \Pee( N=m )\\
& \leq & \Pee( \chi(G(n,\alpha)) > k ) \cdot \Pee( N\leq n) + \Pee( N > n ) \\
& = & \Pee( \chi(G(n,\alpha)) > k ) \cdot (1/2+o_n(1)) + 1/2+o_N(1),
\end{array} $$

\noindent
where we use $\Pee( N \leq n ) = 1/2+o_n(1)$, which follows for instance from the central limit theorem.
The first item follows.
For the second item, we similarly have

$$ \begin{array}{rcl} 
\Pee_n(\chi(G(\Xcal,\alpha))> k)
& \geq & \sum_{m\geq n} \Pee( \chi(G(m,\alpha))> k ) \Pee( N=m ) \\
& \geq & \Pee( \chi(G(n,\alpha)) > k ) \cdot \Pee( N \geq  n ) \\
& = & \Pee( \chi(G(n,\alpha)) > k ) \cdot (1/2+o_n(1)). 
\end{array} $$

\end{proof}

As a side remark we point out that the lemma and its proof actually hold 
if we replace ``chromatic number $>k$'' with any graph property that is ``preserved under taking supergraphs''
(i.e.~if we arbitrarily add vertices and/or edges to a graph having the property then the 
resulting graph also has the property).

\noindent
We shall make use of the following observation due to Kahle and Martinez-Figueroa~\cite{KahleFig}.

\begin{lemma}[\cite{KahleFig}]\label{lem:KFchi}
If $\displaystyle \bigcup_{v\in V} \kap(v,\alpha/2) = S^d$ then 
$\displaystyle \chi( G(V,\alpha) ) \geq d+2$.
\end{lemma}

This result is not stated as a separate lemma or theorem in~\cite{KahleFig}, but is proved 
integral to the proof of Theorem 1.1. The proof in~\cite{KahleFig} is also phrased 
in terms of a different -- but equivalent -- definition
for the Borsuk graph from the one we employ. 
Since the computations translating between the two settings take more mental effort 
than the proof itself and the proof is short and sweet, we include it for the benefit of the reader.

\begin{proof}
Suppose that $G(V,\alpha)$ can be coloured with $d+1$ colours. 
Then we can write $V=V_1\cup\dots\cup V_{d+1}$ such that there are no 
edges with both endpoints in the same $V_i$.
Let us set 

$$ U_i := \bigcup_{v \in V_i} \kap(v,\alpha/2). $$

\noindent
By a version of the Lyusternik-Shnirelmann theorem for open sets (see e.g.~\cite{MatousekBorsukUlam}, Theorem 2.1.1.) 
there exist $1\leq i\leq d+1$ and $u \in S^d$ such that $u,-u \in U_i$.
By definition of $U_i$, there are $v,w \in V_i$ such that $u \in \kap(v,\alpha/2), -u \in \kap(w,\alpha/2)$.
The triangle inequality for the geodesic distance gives:

$$ \dist(v,w) \geq \dist(u,-u) - \dist(u,v) - \dist(-u,w) 
> \pi - \alpha, $$

\noindent
so that $vw$ must be an edge of $G(V,\alpha)$, contradicting 
the choice of $V_1,\dots,V_{d+1}$. It follows there cannot be a colouring of $G(V,\alpha)$ in $d+1$ colours.
\end{proof}

\section{Proof of Theorem~\ref{thm:maind+1}.\label{sec:maind+1}}

Rather than working directly on the $d$-sphere, we find it conceptually easier to 
work in $\eR^d$.
%
%
%
We let $\Xcal = \{X_1,\dots,X_N\}$ be as above and set 

$$   \Zcal := \pi[\Xcal] = \{Z_1,\dots,Z_N\}. $$

We say a function $f : S^{d-1} \to \eR$ is (spherically) {\em odd} if $f(u)+f(-u)=0$ for all $u \in S^{d-1}$.
It is {\em $\eps$-Lipschitz} if $|f(u)-f(v)| \leq \eps \cdot \norm{u-v}$ for all $u,v \in S^{d-1}$.

\begin{lemma}\label{lem:embed}
For every $\eps >0$ there exists a $c=c(\eps)>0$ such that the following holds, 
setting $\alpha = c \cdot n^{-1/d}$.
With probability $1-o_n(1)$, there exists a function 

$$ h : S^{d-1} \to \left(-n^{-0.9/d},n^{-0.9/d}\right),$$ 

\noindent
that is odd, $\eps$-Lipschitz and  

$$ \Zcal \cap B\left( (1+h(u))\cdot u, \eps\cdot\alpha \right) \neq \emptyset,$$ 

\noindent 
for all $u \in S^{d-1}$.
\end{lemma}

The lemma and its proof are inspired by a clever proof technique in the recent paper~\cite{BalisterEtal} (Sections 3, 4 and 5).

\subsection{The proof of Lemma~\ref{lem:embed}.}

We fix constants $\delta=\delta(\eps),K=K(\eps)$ to be chosen more precisely over the 
course of our arguments. The constant $\delta$ will be a small positive real and $K$ will be a
large integer.
We dissect $\eR^d$ into axis-parallel hypercubes of side-length 

$$s_0 := \delta\cdot\alpha, $$ 

\noindent 
in the obvious way. These cubes are called {\em level-0 cubes}.
A level-1 cube is a hypercube of side length $s_1 := K\cdot s_0$ 
that is the union of a group of $K\times\dots\times K$ level-0 cubes.
For each $i \geq 1$, having defined the level $(i-1)$-cubes, a level-$i$ cube
is a hypercube of side length 

\begin{equation}\label{eq:sidef} 
s_i := K^{i} \cdot s_{i-1} = K^{i+(i-1)+\dots+1} \cdot s_0
= K^{i(i+1)/2} \cdot s_0,  
\end{equation}

\noindent
that is the union of a group of $K^{i}\times\dots\times K^{i}$ level-$(i-1)$ cubes.

Specifically, the set of $i$-cubes is 

$$ \Qcal_i := \left\{ p + [0,s_i]^d : p \in s_{i} \Zed^d \right\}. $$

We say a level-0 cube $q \in \Qcal_0$ is {\em good} if $q \cap \Zcal \neq \emptyset$. Otherwise it is {\em bad}.
For each $i \geq 1$, each level-$i$ cube $q \in \Qcal_i$ is the union of $(s_i/s_{i-1})^d = K^{di}$ level-$(i-1)$ cubes.
The cube $q\in\Qcal_i$ is called {\em good} if at most one of these level-$(i-1)$ cubes is bad.
Otherwise, if at least two of these level-$(i-1)$ cubes are bad, then $q$ is bad.

\begin{lemma}\label{lem:boxes}
For every $K, \delta$ there is a $c = c(K,\delta)$ such that, 
setting $\alpha := c \cdot n^{-1/d}, s_0 := \delta \cdot \alpha$, etc., as above, for every $i \geq 2 \log\log n$, 
with probability $1-o_n(1)$, every level-$i$ cube $q \in \Qcal_i$ such that 
$q \subseteq B(\orig,100)$ is good.
\end{lemma}

\begin{proof}
By Lemma~\ref{lem:piunif} and the mapping theorem for Poisson processes 
(see e.g.~\cite{Kingmanboek}, Section 2.3), $\Zcal$ is a Poisson point process on $\eR^d$ with intensity function 

$$ f(x) := n \cdot  \frac{1}{(d+1)\kappa_{d+1}} \cdot \left(\frac{2}{1+\norm{x}^2}\right)^d. $$

\noindent 
It follows that each level-0 cube $q \in \Qcal_0$ that is contained in $B(\orig,100)$
is bad with probability at most

\begin{equation}\label{eq:aap}\begin{array}{rcl} 
\Pee(\text{$q$ is bad} ) 
& \leq & 
\exp\left[ -s_0^d \cdot n  \cdot \frac{1}{(d+1)\kappa_{d+1}} \cdot \left(\frac{2}{10001}\right)^d \right] \\[2ex]
& = & 
\exp\left[ -  \frac{1}{(d+1)\kappa_{d+1}} \cdot \left(\frac{2}{10001}\right)^d \cdot \delta^d \cdot c^d\right] \\
& =: & p_0. 
\end{array} 
\end{equation}

\noindent
We point out that by choosing $c$ large enough, we can get $p_0$ to be as close to zero as we like.
Applying the union bound, we see that for every $q \in \Qcal_1$ that is contained in $B(\orig,100)$, we have 

$$ \Pee(\text{$q$ is bad }) \leq  {(s_1/s_0)^{d} \choose 2} p_0^2 \leq (s_1/s_0)^{2d} p_0^2
= K^{2d} p_0^2 =: p_1. $$

\noindent 
Repeating the argument, we obtain upper bounds $p_2,p_3,\dots$ for the probability that 
a cube $q \in \Qcal_i$ that is contained in $B(\orig,100)$ is bad, satisfying the recursion:

$$ p_i = \left(s_i/s_{i-1}\right)^{2d} \cdot p_{i-1}^2 = K^{2di} p_{i-1}^2. $$

\noindent
Iterating (feeding the recursion back into itself) we find

$$ p_i = \left(K^{2d}\right)^{i + 2(i-1) + 4(i-2) + \dots + 2^{i-1}} \cdot 
p_0^{2^i}. $$

\noindent
The exponent of $K^{2d}$ can be written as:

$$ \sum_{j=0}^i 2^j \cdot (i-j)
= 2^i \sum_{j=0}^i \left(1/2\right)^{i-j} \cdot (i-j)
\leq 2^i \cdot \left( \sum_{k\geq 0} k 2^{-k} \right) = 2^{i+1}. $$

\noindent 
We see that 

$$ p_i \leq \left( K^{4 d} \cdot p_0 \right)^{2^i} 
\leq 2^{-2^i}, $$

\noindent 
where the last inequality holds provided we chose $c = c(K,\delta)$ sufficiently large.
(We use the remark following~\eqref{eq:aap}.) 

Let $N_i := |\{ q \in \Qcal_i : q \subseteq B(\orig,100)\}|$ denote the number of level-$i$-cubes that 
are contained in $B(\orig,100)$.
We have 

$$ N_i \leq \frac{\vol( B(\orig,100) )}{s_i^d} \leq \frac{\vol( B(\orig,100) )}{s_0^d} = O(n). $$

\noindent
When $i \geq 2\log\log n$ then 

$$p_i \leq 2^{-2^{2\log\log n}} = 2^{-\log^2 n} = o(1/n). $$

\noindent 
So, by the union bound, for $i \geq 2\log\log n$ we have 

$$ \Pee\left(\text{There is a bad level-$i$ cube contained in $B(\orig,100)$}\right) 
\leq N_i p_i = o_n(1). $$

\noindent 
This concludes the proof of Lemma~\ref{lem:boxes}.
\end{proof}

The rough idea for the proof of Lemma~\ref{lem:embed} is to 
massage an (odd and appropriately Lipschitz) function that avoids bad $(j+1)$-cubes
into one that avoids bad $j$-cubes without making drastic local changes, and then 
repeat until finally we avoid bad $0$-cubes.
The following lemma is the main tool we shall be using for this purpose.

\begin{lemma}\label{lem:fg}
For every $d\geq 2$ and $k\geq 1$, there exist $t=t(d,k), a_0=a_0(d,k), r_0=r_0(d,k), C=C(d,k) > 0$ such 
that the following holds.
Suppose that $0 < a < a_0, 0 < r < r_0$, the finite set $V \subseteq \eR^d$
and the function $f :S^{d-1} \to (-t,t)$ are such that  
$f$ is odd and $a$-Lipschitz and moreover

$$ \left|B\left( (1+f(u))\cdot u, r \right) \cap V\right| \leq k, $$

\noindent
for all $u \in S^{d-1}$.
Then there also exists a function $g : S^{d-1} \to \eR$ that is odd, $(Ca)$-Lipschitz and satisfies 

$$ |f(u)-g(u)| < Car \quad \text{ and } \quad 
B\left( (1+g(u)) \cdot u, a^{C}\cdot r  \right) \cap V = \emptyset, $$

\noindent 
for all $u \in S^{d-1}$.
\end{lemma}

We find it convenient to first prove a preliminary version of the lemma with a restriction on the set $V$
and with $k=1$.

\begin{lemma}\label{lem:fgorth}
Suppose that $0 < a < 0.01, 0 < r < 0.01/\sqrt{d}$, the finite set $V \subseteq [0,\infty)^d$
and the function $f :S^{d-1} \to (-0.1,0.1)$ are such that  
$f$ is odd and $a$-Lipschitz and moreover

$$ \left|B\left( (1+f(u))\cdot u, r \right) \cap V\right| \leq 1, $$

\noindent
for all $u \in S^{d-1}$.
Then there also exists a function $g : S^{d-1} \to \eR$ that is odd, $(9a)$-Lipschitz and satisfies 

$$ \left|f(u)-g(u)\right| < ar \quad \text{ and } \quad 
B\left( (1+g(u)) \cdot u, (a/2) \cdot r  \right) \cap V = \emptyset, $$

\noindent 
for all $u \in S^{d-1}$
\end{lemma}
%

\begin{proof}
Without loss of generality we can and do assume $\orig \not\in V$. (If we obtain the 
claimed $g$ wrt.~$V \setminus \{\orig\}$ then automatically $\orig \not\in B\left( (1+g(u)) \cdot u, (a/2)\cdot r  \right)$
for all $u$.)

For notational convenience, let us write $\varphi(u) := (1+f(u)) \cdot u$.
For each $p \in V$ we define:

$$ f_{p}(u) := b_p \cdot \left( \max\left(0,r-\norm{\varphi(u)-p}\right)
- \max\left(0, r-\norm{\varphi(-u)-p}\right)\right), $$

\noindent
where 

$$ b_p := \begin{cases}
         a & \text{ if $\norm{p} < 1+f(p/\norm{p})$, }\\
         -a & \text{ otherwise. }
        \end{cases} $$

\noindent
We now set

$$ g := f + \sum_{p \in V} f_p. $$

\noindent
Since $f_p(u)+f_p(-u)=0$ for all $p$ by construction, 
we clearly have $g(u)+g(-u)=f(u)+f(-u)=0$ for all $u$. So $g$ is odd as required.

Next, we point out:

\begin{equation}\label{eq:disjsupp} 
\text{The $f_p$ have disjoint supports.} 
\end{equation}

\begin{quote}\begin{proofof}{\eqref{eq:disjsupp}} 
Aiming for a contradiction, suppose $u \in S^{d-1}$ and $p \neq q \in V$ are such that 
$f_p(u) \neq 0$ and $f_{q}(u) \neq 0$.
We have either {\bf a-1)} $\norm{\varphi(u)-p} < r$ or {\bf a-2)} $\norm{\varphi(-u)-p} < r$, and 
either {\bf b-1)} $\norm{\varphi(u)-q} < r$ or {\bf b-2)} $\norm{\varphi(-u)-q} < r$.
Swapping $u$ for $-u$ if needed, we can and do assume without loss of generality that 
{\bf a-1)} holds.
The condition {\bf b-1)} cannot simultaneously hold, because otherwise 
$|B( \varphi(u), r ) \cap V | \geq 2$. So {\bf b-2)} must hold.

There is a coordinate $1\leq i \leq d$ for which $|\varphi_i(u)| = (1+f(u)) \cdot |u_i| \geq 0.9/\sqrt{d}$. 
As $p \in [0,\infty)^d$, we in fact have $\varphi_i(u) = |\varphi_i(u)| \geq 0.9/\sqrt{d}$, because otherwise

$$ \norm{\varphi(u)-p} \geq |\varphi_i(u)-p_i| \geq 0.9/\sqrt{d} > r. $$

\noindent
This implies that $\varphi_i(-u) = - |\varphi_i(u)| \leq -0.9/\sqrt{d}$ and, since $q \in [0,\infty)^d$, 
we find

$$ \norm{\varphi(-u)-q} \geq |\varphi_i(-u)-q_i| \geq 0.9/\sqrt{d} > r, $$

\noindent
contradicting {\bf b-2)}.
So $u,p,q$ satisfying $f_p(u) \neq 0, f_q(u) \neq 0$ cannot exist, as was to be shown.
\end{proofof}\end{quote}

\noindent 
Having established~\eqref{eq:disjsupp}, it immediately follows that 
$|g(u)-f(u)| \leq a r$ for all $u \in S^{d-1}$.

Now let $u,v \in S^{d-1}$ be arbitrary. By~\eqref{eq:disjsupp}, there are $p,q \in V$ such 
that $g(u) = f(u) + f_p(u)$ and $g(v) = f(v) + f_q(v)$.
It follows that 

\begin{equation}\label{eq:gugv} |g(u)-g(v)| \leq 
|f(u)-f(v)| + |f_p(u)-f_p(v)| + |f_q(u)-f_q(v)|.
\end{equation}

\noindent 
(If $p=q$ it is trivially true and if $p\neq q$ then $f_p(v)=f_q(u)=0$.)
By definition of $f_p$:

\begin{equation}\label{eq:fpufpv}
\begin{array}{rcl} 
|f_p(u)-f_p(v)| 
& \leq & 
a \cdot \left( \left| \norm{\varphi(u)-p} - \norm{\varphi(v)-p} \right| 
+ \left| \norm{\varphi(-u)-p} - \norm{\varphi(-v)-p} \right| \right) \\
& \leq & 
a \cdot \left( \norm{\varphi(u)-\varphi(v)} + \norm{\varphi(-u)-\varphi(-v)} \right), 
\end{array}
\end{equation}

\noindent
where we've used the triangle inequality for the second line.
We also have

$$ \begin{array}{rcl} 
\norm{\varphi(u)-\varphi(v)} 
& \leq & 
\norm{ (1+f(u)) u - (1+f(u)) v} + \norm{(1+f(v))v-(1+f(u))v} \\
& = & 
(1+f(u)) \cdot \norm{u-v} + |f(v)-f(u)| \\
& \leq &  
1.1 \cdot \norm{u-v} + a\cdot \norm{u-v} \\
& \leq & 
2 \cdot \norm{u-v}. 
\end{array} $$

\noindent
Filling this into~\eqref{eq:fpufpv}
gives

$$ |f_p(u)-f_p(v)| \leq 4a \cdot \norm{u-v}. $$

\noindent
Repeating the argument, the same bound applies to $|f_q(u)-f_q(v)|$.
Combining with~\eqref{eq:gugv} and recalling that $f$ is $a$-Lipschitz gives

$$ |g(u)-g(v)| \leq 9a \cdot \norm{u-v}. $$

\noindent
In other words, we've verified that $g$ is $9a$-Lipschitz.

For notational convenience, we write $\psi(u) := (1+g(u))\cdot u$.
It remains to see that $B( \psi(u), a r/2 ) \cap V = \emptyset$ for all $u \in S^{d-1}$.
To this end, let $u \in S^{d-1}$ be arbitrary.
If $u$ does not lie in the support of any $f_p$ then 
$\norm{\psi(u)-p} = \norm{\varphi(u)-p} \geq r$ for any $p \in V$.
We can thus assume $u$ lies in the support of some (unique) $p \in V$.
For $q \neq p$ we have 

$$ \begin{array}{rcl} 
\norm{ \psi(u) - q} 
& \geq &  
\norm{\varphi(u)-q} - \norm{\psi(u)-\varphi(u)}  
= \norm{\varphi(u)-q} - |g(u)-f(u)| \\
& > & 
r - a r = (1-a)r > ar/2. 
\end{array} $$

It thus remains to consider the distance between $\psi(u)$ and $p$.
If $\norm{\varphi(-u)-p} < r$ then 

$$ \norm{\varphi(u)-p} \geq \norm{\varphi(u)-\varphi(-u)} - \norm{\varphi(-u)-p}
= 2 - \norm{\varphi(-u)-p} > 2 - r. $$

\noindent
Hence 

$$ \begin{array}{rcl}
\norm{\psi(u)-p} 
& \geq & \norm{\varphi(u)-p} - \norm{\psi(u)-\varphi(u)}
=  \norm{\varphi(u)-p} - |f(u)-g(u)| \\
& > & 2-r-ar > ar/2.  
\end{array} $$

\noindent
So we can assume that $\norm{\varphi(u)-p} < r$.
We denote $v := p/\norm{p}$ and $\alpha := \angle v\orig u$.
If $\alpha \geq \pi/2$ then 

$$ \norm{\psi(u)-p} \geq \norm{\psi(u)} = (1+g(u)) > 0.9-ar > ar/2. $$

\noindent
We can therefore assume $0 \leq \alpha \leq \pi/2$.
By considering the distance between $\psi(u)$ and the linear hull $\text{span}(\{p\})$ of $p$, we see that

$$ \norm{\psi(u)-p} \geq (1+g(u)) \cdot \sin \alpha 
\geq (0.9-ar) \cdot (2/\pi) \cdot \alpha > \alpha/4, $$

\noindent
(where we use that $\sin\alpha \geq (2/\pi)\cdot\alpha$ for $0 \leq \alpha \leq \pi/2$).
So if $\alpha \geq 2ar$ then $\norm{\psi(u)-p} \geq ar/2$.
We therefore assume $0 \leq \alpha \leq 2ar$.
We have

$$ \begin{array}{rcl} 
\norm{\psi(u)-p} 
& \geq & 
\norm{p-\psi(v)} - \norm{\psi(u)-\psi(v)} 
\geq
\norm{p-\psi(v)} - 9a \cdot \norm{u-v} \\
& = & 
\norm{p-\psi(v)} - 9a \cdot 2\sin(\alpha/2) 
\geq
\norm{p-\psi(v)} - 9a \cdot \alpha \\
& \geq & 
\norm{p-\psi(v)} - 18 a^2 r. 
\end{array} $$

\noindent
(In the fourth line we use that $\sin x \leq x$ for $x \geq 0$.)
Because of the way we have chosen the sign of $b_p$ in the definition of $f_p$, we find that

$$ \begin{array}{rcl} 
\norm{p - \psi(v)} 
& = & |\norm{p}-(1+g(v))| = |\norm{p}-(1+f(v))|+|f(v)-g(v)| \\
& = & \norm{p-\varphi(v)} + a(r-\norm{p-\varphi(v)}) = ar + (1-a)\norm{p-\varphi(v)} \\
& \geq & ar. 
\end{array} $$

\noindent
It follows that 

$$ \norm{p - \psi(u)} \geq ar - 18 a^{2} r > ar/2. $$

\noindent
We have now verified that $B( \psi(u), ar/2 ) \cap V = \emptyset$, which concludes the 
proof of the lemma.
\end{proof}

We separate one more observation to facilitate the derivation of the more 
general Lemma~\ref{lem:fg} from the more restricted Lemma~\ref{lem:fgorth}.
Recall that an {\em orthant} of $\eR^d$ is one of the $2^d$ 
subsets $O_\sigma := \{ x \in \eR^d : x_i\sigma_i \geq 0 \}$ with $\sigma \in \{-1,+1\}^d$.

\begin{lemma}\label{lem:part}
For all $d\geq 2, k\geq 1$ there exists a $\ell = \ell(d,k)$ such that the following holds.
Suppose $r>0$ and $U, V \subseteq \eR^d$ are such that $V$ is finite and 
$|B(u,r) \cap V| \leq k$ for all $u \in U$.
Then we can write $V$ as 

$$ V = V_1 \cup \dots \cup V_\ell, $$

\noindent
where each $V_i$ is contained in some orthant and $|B(u,r/4) \cap V_i| \leq 1$ for all $u \in U$
and $1\leq i\leq\ell$.
\end{lemma}

\begin{proof}
We set 

$$ W := \{ v \in V : \norm{v-u} < r/2 \text{ for some $u \in U$}\}. $$

\noindent
We have $|B( w, r/2 ) \cap W| \leq k$ for all $w \in W$, since 
$B(w,r/2) \subseteq B( u, r)$ for some $u \in U$.

Let $W_1 \subseteq W$ be a subset such that if $w,w'\in W_1$ are distinct then $\norm{w-w'} \geq r/2$
and, subject to this, it has the maximum possible size.
For $2 \leq i \leq k-1$ we similarly pick

$$ W_i \subseteq W \setminus (W_1\cup\dots W_{i-1}), $$

\noindent 
of the maximum possible size subject to the demand that 
$\norm{w-w'} \geq r/2$ for all distinct $w,w'\in W_i$.
In particular

$$|B( u, r/4 ) \cap W_i| \leq 1. $$

\noindent
(Because the diameter of $B(u,r/4)$ is $r/2$.)

For each $w \in W \setminus W_1$ we must have 

$$|B(w,r/2) \cap (W\setminus W_1)| \leq k-1. $$

\noindent
(Otherwise $B(w,r/2)$ does not intersect $W_1$ and $w$ could have been added to $W_1$, contradicting 
maximality of $|W_1|$.) Repeating the argument, we find

\begin{equation}\label{eq:maxW} 
|B(w,r/2)\cap (W\setminus (W_1\cup\dots\cup W_{i-1}))| \leq k-i+1,  
\end{equation}

\noindent
for all $i$ and all $w\in W \setminus (W_1\cup\dots\cup W_{i-1})$.
In particular we have 

$$ W = W_1 \cup \dots \cup W_{k-1}. $$

\noindent
(By~\eqref{eq:maxW} and the maximality of $|W_{k-1}|$, we must have 
$W_{k-1} = W \setminus (W_1\cup\dots\cup W_{k-2})$.)
Writing 

$$ W_k := V \setminus W,$$ 

\noindent 
we thus have
$V = W_1\cup\dots\cup W_k$ and $|B( u, r/4 ) \cap W_i| \leq 1$
for all $1\leq i \leq k$ and all $u \in U$. 

We label the $2^d$ orthants of $\eR^d$ arbitrarily 
as $O_1,\dots,O_{2^d}$ and now define $V_1,\dots,V_{k2^d}$ by:

$$ V_{2^d\cdot(i-1) + j}  := W_i \cap O_j, $$

\noindent
for $1 \leq i \leq k$ and $1\leq j \leq 2^d$. Setting $\ell := k2^d$, we have constructed 
the sought decomposition $V = V_1\cup\dots\cup V_\ell$.
\end{proof}

\begin{proofof}{Lemma~\ref{lem:fg}}
Suppose that $t,a,r,f : S^{d-1} \to (-t,t)$ and $V \subseteq \eR^d$ are such that
$f$ is odd and $a$-Lipschitz and $|B( (1+f(u)) u, r) \cap V| \leq k$ for all $u$. 
By the previous lemma, for some constant $\ell = \ell(d,k)$, we can write $V = V_1 \cup \dots \cup V_\ell$
where each $V_i$ is contained in some orthant and $|B( (1+f(u)) u, r/4 ) \cap V_i| \leq 1$
for all $u \in S^{d-1}, 1\leq i \leq \ell$.

We now ``iterate'' Lemma~\ref{lem:fgorth}. That is, we do the following.
We set $g_0 := f, t_0 := t, a_0 := a, r_0 := r/4$. 
Suppose that, for some $0 \leq i\leq \ell-1$, we have defined $g_i, t_i, a_i, r_i$ such that 
$g_i : S^{d-1} \to (-t_i,t_i)$ is odd, $a_i$-Lipschitz and 

$$ \begin{array}{l} 
|B( (1+g_i(u)) u, r_i) \cap V_j| = 0 \text{ for $j \leq i$, and } \\[2ex]
|B( (1+g_i(u)) u, r_i) \cap V_j| \leq 1 \text{ for $j > i$, }
\end{array} $$

\noindent
for all $u \in S^{d-1}$.
Provided $t_i < 0.1, a_i < 0.01$ and $r_i < 0.01/\sqrt{d}$, we can 
apply Lemma~\ref{lem:fgorth} (and symmetry) to $g_i$ and $V_{i+1}$ to obtain $g_{i+1}$.
Setting 

$$ a_{i+1} = 9a_i, \quad r_{i+1} = (a_i/2)\cdot r_i, 
\quad \Delta_{i+1} := a_ir_i, \quad t_{i+1} = t_i + \Delta_{i+1}, $$

\noindent
we have that $g_{i+1} : S^{d-1} \to (-t_{i+1},t_{i+1})$ is odd, $a_{i+1}$-Lipschitz and 
satisfies 

$$ |g_{i+1}(u)-g_{i}(u)| \leq \Delta_{i+1}, \quad B( (1+g_{i+1}(u)) u, r_{i+1}) \cap V_{i+1} = \emptyset, $$ 

\noindent
for all $u \in S^{d-1}$.
We have $\Delta_{i+1}+r_{i+1} = (3/2)a_ir_i < r_i$, which implies

\begin{equation}\label{eq:Bsubs} 
B( (1+g_{i+1}(u)) u, r_{i+1}) \subseteq B( (1+g_i(u)) u, r_i)
\end{equation}

\noindent 
for all $u \in S^{d-1}$. In particular

$$ \begin{array}{l} 
|B( (1+g_{i+1}(u)) u, r_{i+1}) \cap V_j| = 0 \text{ for $j \leq i+1$, and } \\[2ex]
|B( (1+g_{i+1}(u)) u, r_{i+1}) \cap V_j| \leq 1 \text{ for $j > i+1$, }
\end{array} $$
%

\noindent
The function $g = g_{\ell}$ is the one we seek. It is well defined 
provided the conditions of Lemma~\ref{lem:fgorth} were met each time 
we applied it. That is, provided 
$t, t_1, \dots, t_{\ell} < 0.1$ and $a, a_1, \dots, a_{\ell} < 0.01$ 
and $r, r_1, \dots, r_{\ell} < 0.01/\sqrt{d}$. 
We note that 

$$ \begin{array}{rcl} 
a_{i} & = & 9 a_{i-1} = \dots = 9^{i}\cdot a, \\[2ex]
r_i & = & (a_{i-1}/2) r_{i-1} = \dots = (a_{i-1}\cdot\dots\cdot a_1\cdot a)\cdot 2^{-i}\cdot r_0\\[1.5ex]
& = &  
9^{i(i-1)/2} \cdot 2^{-(i+2)} \cdot a^i \cdot r
= 3^{i(i-1)}\cdot 2^{-(i+2)} \cdot a^i \cdot r, \\[2ex]
\Delta_i & = & a_{i-1} r_{i-1} = 3^{i(i-1)} \cdot 2^{-(i+3)} \cdot a^{i} \cdot r, \\[2ex]
\end{array} $$ 

\noindent
This gives 

$$ \begin{array}{l} 
\Delta_1 + \dots + \Delta_i \leq i \cdot 3^{i(i-1)}\cdot a \cdot r \leq 3^{i^2} \cdot a \cdot r,
\\[2ex]
t_i = t + \Delta_1 + \dots + \Delta_i \leq t +  3^{i^2} \cdot a \cdot r. 
\end{array} $$

\noindent
Setting $C := 3^{\ell^2}$, the function $g$ is clearly well-defined if 
$a< 0.01/C, r < 0.01/(C\sqrt{d}), t < 0.05$, with plenty of room to spare.

To see that our choice of $g$ and $C$ is as required, we note that 
$|f(u)-g(u)| \leq \Delta_1+\dots+\Delta_\ell \leq C a r$ for all $u\in S^{d-1}$, that 
$a_{\ell} < C a$ so that $a_{\ell}$-Lipschitz implies $Ca$-Lipschitz, and 
similarly $r_{\ell} >a^C r$ so that
$B( (1+g(u)) u, r_{\ell} ) \supseteq B( (1+g(u))u, a^C r)$ for all $u\in S^{d-1}$.
%
%
%
%
%
%
\end{proofof}

\begin{proofof}{Lemma~\ref{lem:embed}}
We apply Lemma~\ref{lem:boxes} with parameters $K=K(\eps), \delta=\delta(\eps)$ to be determined 
during the course of the proof. 
For the remainder of the proof, we assume that $\Zcal$ is such that 
the conclusion of Lemma~\ref{lem:boxes} holds. (Which happens with probability $1-o_n(1)$.)
Let $C=C(d,2^d), t=t(d,2^d), a_0=a_0(d,2^d), r_0=r_0(d,2^d)$ be as provided by Lemma~\ref{lem:fg}
applied in dimension $d$ with parameter $k=2^d$.

For $0 \leq j \leq i := \lceil 2\log\log n \rceil$, we set

$$ \Bcal_j := 
\left\{  p \in s_j\Zed^d : \text{$p+[0,s_j]^d$ is bad and contained  in $B(\orig,100)$}\right\}. $$

We construct the sought function $h$ by reverse induction.
Specifically, for $0 \leq j \leq i$, the function $h_j : S^{d-1} \to \eR$ 
will be odd, $(\eps \cdot C^{-j})$-Lipschitz, and satisfy: 

\begin{equation}\label{eq:dem1} 
|h_j(u)| \leq s_{j+1} + s_{j+2} + \dots + s_i, 
\end{equation}

\noindent
and

\begin{equation}\label{eq:dem2} 
B\left( (1+h_j(u))\cdot u, 2\sqrt{d}\cdot s_j \right) \cap \Bcal_j = \emptyset, 
\end{equation}

\noindent
for all $u \in S^{d-1}$.

\noindent
Let us point out that

\begin{equation}\label{eq:sums} 
\begin{array}{rcl}
s_0 + s_1 + \dots + s_i 
& \leq & (i+1)s_i = (i+1) K^{i(i+1)/2} s_0 = e^{O\left( (\log\log n)^2 \right)} \cdot n^{-1/d} \\
& = & 
n^{-1/d+o(1)}. 
\end{array} 
\end{equation}

\noindent
In particular demand~\eqref{eq:dem1} implies that 

$$ h_j(u) \in \left(-n^{-0.9/d},n^{-0.9/d}\right), $$

\noindent
for $n$ sufficiently large. 
Demand~\eqref{eq:dem2} implies that $(1+h_0(u)) \cdot u$ lies in a 
good level-0 cube. In particular 

$$ B( (1+h_0(u))\cdot u, \sqrt{d} \cdot s_0 ) \cap \Zcal \neq \emptyset. $$

\noindent 
Provided $\delta$ was sufficiently small, $\sqrt{d} \cdot s_0 = \sqrt{d} \cdot 
\delta \cdot \alpha < \eps \cdot \alpha$.
So the function $h_0$ can serve as the function $h$ we seek.

As mentioned, we shall establish the existence of $h_0$ by reverse induction. 
We start by defining $h_i$. Clearly $h_i \equiv 0$ does the trick, as $\Bcal_i=\emptyset$.
Now suppose that we've managed to construct $h_{j+1}$ for some $j < i$.

Before proceeding, let us point out that if $h_{j+1}$ satisfies the demands, we have:

\begin{equation}\label{eq:ballsiplus1}
\left| B\left( (1+h_{j+1}(u))\cdot u, \frac{s_{j+1}}{2}\right) \cap \Bcal_j \right| \leq 2^d \quad 
\text{ for all $u \in S^{d-1}$.} 
\end{equation}

\noindent
(Because $B\left( (1+h_{j+1}(u))\cdot u, s_{j+1}/2\right)$ intersects at most $2^d$ level $j+1$ cubes.
None of these can be bad, because that would contradict~\eqref{eq:dem2}.)
We can thus obtain $h_j$ by invoking Lemma~\ref{lem:fg} with 

$$ k=2^d, \quad f=h_{j+1}, \quad V = \Bcal_j, \quad a = \eps\cdot C^{-(j+1)}, \quad 
r = s_{j+1}/2. $$

\noindent
(For $n$ sufficiently large, we are allowed to apply Lemma~\ref{lem:fg} because
$|h_{j+1}| < t$ and $r < r_0$ by~\eqref{eq:sums}, and 
$a < a_0$ assuming without loss of generality $\eps < a_0$.)
We note that 

$$ Ca = \eps \cdot C^{-j}, $$

\noindent
so that $h_j$ is $(\eps C^{-j})$-Lipschitz as required. We also have 

$$ \begin{array}{rcl} 
a^C r 
& = &\left(\eps \cdot C^{-(j+1)}\right)^C \cdot (s_{j+1}/2) \\
& = & \left(\eps \cdot C^{-(j+1)}\right)^C \cdot \frac12 \cdot K^{j+1} \cdot s_{j} \\
& = & \left(\eps^C/2\right) \cdot \left(K/C^C\right)^{j+1} \cdot s_{j} \\
& > & 2\sqrt{d} \cdot s_j,
\end{array} $$

\noindent
where the last line follows because we chose $K$ sufficiently large in the beginning.
So 

$$B\left( (1+h_{j+1}(u))\cdot u, 2\sqrt{d} \cdot s_j\right) \cap \Bcal_j 
\subseteq B\left( (1+h_{j+1}(u))\cdot u, a^C r \right) \cap \Bcal_j 
= \emptyset, $$

\noindent
for all $u \in S^{d-1}$, as
required.
Also note that $\Delta := Car = \eps C^{-j} s_{j+1}/2 < s_{j+1}$, giving

$$ \begin{array}{rcl} 
|h_j(u)| 
& \leq & 
|h_{j+1}(u)| + |h_j(u)-h_{j+1}(u)| \leq 
(s_{j+2} + \dots + s_i) + \Delta \\ 
& < & 
s_{j+1} + s_{j+2} + \dots + s_i. 
\end{array} $$

\noindent
We see that $h_j$ satisfies the demands provided that $h_{j+1}$ does. This concludes the proof.
\end{proofof}

\subsection{Finishing the proof of Theorem~\ref{thm:maind+1}.}

We fix $\eps > 0$, small and to be determined more precisely in what follows.
For the remainder of the section, we assume that $\Xcal$ is such that 
the conclusion of Lemma~\ref{lem:embed} holds (which happens with probability $1-o_n(1)$) and we 
let $\alpha$ and $h$ be as provided by that lemma.
Let us write

$$ A := \bigcup_{u \in S^{d-1}} B( (1+h(u)) \cdot u, \eps \cdot \alpha ). $$

\noindent
We arbitrarily label $\Xcal \cap \pi^{-1}[A]$ as 

$$ \Xcal \cap \pi^{-1}[A] = \{X_1,\dots,X_M\}, $$ 

\noindent 
and set 

$$ Z_i := \pi(X_i), \quad U_i := Z_i / \norm{Z_i}. $$

\noindent
Perhaps superfluously, we note that the $U_i$ lie on the $(d-1)$-sphere $S^{d-1}$.
Our construction of $h,\alpha$ has the following consequence:

\begin{corollary}\label{cor:metriccc}
Provided the constant $\eps>0$ was chosen sufficiently small and $n$ is sufficiently large, we have:
\begin{enumerate} 
\item\label{itm:metriccc1} $\displaystyle \bigcup_{i=1}^M \kap_{S^{d-1}}(U_i,\alpha/10) = S^{d-1}$.
\item\label{itm:metriccc2} If $\dist_{S^{d-1}}(U_i,U_j) > \pi-\alpha/5$ then  $\dist_{S^d}(X_i,X_j) > \pi-\alpha$.
\end{enumerate}
\end{corollary}

\begin{proofof}{Theorem~\ref{thm:maind+1} assuming Corollary~\ref{cor:metriccc}}
Part~\ref{itm:metriccc2} of Corollary~\ref{cor:metriccc}
shows that $G_{d-1}(\{U_1,\dots,U_M\},\alpha/5)$ is a subgraph of $G_d(\Xcal,\alpha)$.
We thus have

$$ \chi( G_d(\Xcal,\alpha) ) \geq \chi(G_{d-1}(\{U_1,\dots,U_M\},\alpha/5)) \geq d+1, $$

\noindent 
where the last inequality follows by part~\ref{itm:metriccc1} of Corollary~\ref{cor:metriccc}
together with Lemma~\ref{lem:KFchi}.
\end{proofof}

It remains to verify Corollary~\ref{cor:metriccc}.
We find it convenient to separate out a few preliminary 
calculations before starting its proof in earnest.

\begin{lemma}\label{lem:eenplushu}
Provided the constant $\eps>0$ was chosen sufficiently small and $n$ is sufficiently large, we have:

$$\norm{ (1+h(u))\cdot u - (1+h(v))\cdot v} \leq 2 \cdot \norm{u-v}, $$
 
\noindent 
for all $u, v \in S^{d-1}$. 
\end{lemma}

\begin{proof}
We have 

$$ \begin{array}{rcl} 
\norm{ (1+h(u))u - (1+h(v))v } 
& \leq & (1+h(u)) \cdot \norm{u - v} + |h(v)-h(u)| \\
& \leq & (1+h(u)) \cdot \norm{u - v} + \eps \cdot \norm{u-v} \\
& = & (1+h(u)+\eps) \cdot \norm{u-v} \\
& \leq & 2 \cdot \norm{u-v},
\end{array} $$

\noindent
where we've used that $h$ is $\eps$-Lipschitz and the last line follows assuming 
$\eps$ is sufficiently small and $n$ sufficiently large (recalling $|h| \leq n^{-0.9/d}$).
\end{proof}

For $z \in \eR^d \setminus \{\orig\}$, we write:

$$ \varphi(z) := \pi(-\pi^{-1}(z)). $$

\noindent
That is, if $z = \pi(x)$ with $x \in S^d$, then $\varphi(z)$ corresponds to the 
image of the antipode of $x$ under the stereographic projection $\pi$.
We remark that $\varphi$ can be written alternatively as:

\begin{equation}\label{eq:phialt} \varphi(z) = -z/\norm{z}^2. \end{equation}

\noindent
(An easy way to see this is to note that 
$z, \varphi(z)$  must lie on a line through the origin, and on opposite sides of the origin, and 
if $z = \pi(x)$ then 

$$\begin{array}{rcl} 
\norm{z} 
& = & 
\sqrt{(x_1^2+\dots+x_d^2)/(1-x_{d+1})^2} = \sqrt{(1-x_{d+1}^2)/(1-x_{d+1})^2} \\
& = & 
\sqrt{(1+x_{d+1})/(1-x_{d+1})}, 
\end{array} $$

\noindent
and similarly $\varphi(z) = \pi(-x)$ has norm $\norm{\varphi(z)}=\sqrt{(1-x_{d+1})/(1+x_{d+1})} = 1/\norm{z}$.)

\begin{lemma}\label{lem:phieenplushu}
Provided the constant $\eps>0$ was chosen sufficiently small and $n$ is sufficiently large, we have:

$$ \norm{\varphi( (1+h(u))\cdot u ) + (1+h(-u)) u} \leq \eps\cdot\alpha, $$

\noindent
for all $u \in S^{d-1}$.
\end{lemma}

\begin{proof}
By~\eqref{eq:phialt}, we have 

$$ \begin{array}{rcl} 
\varphi( (1+h(u))\cdot u ) & = & -(1/(1+h(u))) \cdot u \\[2ex]
& = & - \left(1 - h(u) + h(u)^2/(1+h(u))\right) \cdot u \\[2ex]
& = & - \left(1+h(-u)+h(u)^2/(1+h(u))\right) \cdot u,
\end{array} $$

\noindent
using that $h$ is odd in the last step.
So

$$\begin{array}{rcl} 
\norm{\varphi( (1+h(u))\cdot u ) + (1+h(-u)) u} 
& = & 
\norm{(h(u)^2/(1+h(u))) \cdot u} \\[2ex]
& = & 
h(u)^2 / (1+h(u)) \\[2ex]
& < & 
n^{-1.8/d} / (1-n^{-0.9/d}) \\[2ex]
& < & \eps \cdot \alpha,
\end{array} $$

\noindent
for $n$ sufficiently large, using that $\alpha = c \cdot n^{-1/d}$ in the last step.
\end{proof}

\begin{lemma}\label{lem:phizw} 
If $w, z \in \eR^d$ satisfy $0.9999 \leq \norm{w},\norm{z} \leq 1.0001$ then 

$$ \norm{\varphi(z)-\varphi(w)} \leq 4 \cdot \norm{z-w}. $$

\end{lemma}

\begin{proof}
We have

$$ \begin{array}{rcl}
\norm{\varphi(z)-\varphi(w)} 
& = & 
\norm{(1/\norm{z})^2\cdot z- (1/\norm{w})^2 \cdot w} \\[2ex]
& \leq & 
\norm{(1/\norm{z})^2\cdot z- (1/\norm{z})^2 \cdot w}
+ \norm{\left((1/\norm{z})^2- (1/\norm{w})^2\right) \cdot w} \\[2ex]
& = & 
(1/\norm{z})^2 \cdot \norm{z- w} 
+ \left|(\norm{w}^2-\norm{z}^2)/(\norm{z}\cdot\norm{w})^2\right| \cdot \norm{w}.
\end{array} $$

\noindent
We now note that 

$$ \left|\norm{w}^2-\norm{z}^2\right| = \left|(\norm{w}-\norm{z})\cdot(\norm{w}+\norm{z})\right|
\leq \norm{z-w} \cdot (\norm{w}+\norm{z}). $$

\noindent
Feeding this into the previous inequality gives

$$ \begin{array}{rcl}
\norm{\varphi(z)-\varphi(w)} 
& \leq & 
\left((1/\norm{z})^2 + \norm{w}\cdot (\norm{z}+\norm{w})/(\norm{z}\cdot\norm{w})^2 \right) \cdot \norm{z-w}  \\
& \leq & 
4 \cdot \norm{z-w},
\end{array} $$

\noindent 
where we've used the bounds $0.9999 < \norm{z},\norm{w} < 1.0001$ in the last line.
\end{proof}

\begin{proofof}{Corollary~\ref{cor:metriccc}}
Fix $u \in S^{d-1}$. By construction of $h$, we have $B( (1+h(u)) u, \eps \alpha ) \cap \Zcal \neq \emptyset$.
In other words, there is an $1\leq i \leq M$ such that $\norm{ (1+h(u))\cdot u - Z_i } < \eps \alpha$.
We have

\begin{equation}\label{eq:Uui} 
\begin{array}{rcl} 
\norm{u - U_i} 
& = & 
\left(\frac{1}{1+h(u)}\right) \cdot \norm{ (1+h(u))\cdot u - (1+h(u))\cdot U_i} \\[2ex]
& \leq & 
2 \cdot \left( \norm{(1+h(u))u - Z_i} + \norm{(1+h(u))U_i - Z_i} \right) \\[2ex]
& = & 
2 \cdot \left( \norm{(1+h(u))u - Z_i} + \left|1+h(u)-\norm{Z_i}\right| \right) \\[2ex]
& < & 
4 \cdot \eps \cdot \alpha. 
\end{array} \end{equation}

\noindent 
where we used $h>-1/2$ (true for $n$ sufficiently large) in the second line. 
Appealing to~\eqref{eq:distnorm}, this gives

$$ \dist_{S^{d-1}}(u,U_i) \leq \pi \cdot \norm{u - U_i} <  4\cdot\pi\cdot\eps\cdot\alpha < \alpha/10, $$

\noindent
(provided $\eps$ was chosen sufficiently small) so that part~\ref{itm:metriccc1} follows.

\vspace{.5ex}

Suppose $1\leq i<j \leq M$ are such that $\dist(U_i,U_j) > \pi-\alpha/5$.
Let us pick $u \in S^{d-1}$ with $\norm{Z_i - (1+h(u))u} < \eps\cdot\alpha$.
By Lemma~\ref{lem:eenplushu} and~\eqref{eq:Uui}, we have:

$$ \norm{ (1+h(U_i))U_i - (1+h(u))u } \leq  2 \cdot \norm{u-U_i} \leq 8 \cdot \eps \cdot \alpha, $$

\noindent 
It follows that 

$$ \begin{array}{rcl} 
\norm{Z_i-(1+h(U_i))U_i} 
& \leq & \norm{Z_i - (1+h(u))u} + \norm{ (1+h(U_i))U_i - (1+h(u))u } \\
& < & 9 \cdot \eps \cdot \alpha. 
\end{array} $$

\noindent 
Of course the same bound applies to $\norm{Z_j-(1+h(U_j))U_j}$. 
Lemma~\ref{lem:phizw} now gives:

$$ \norm{\varphi(Z_j) - \varphi((1+h(U_j))U_j)} \leq 36 \cdot \eps \cdot \alpha. $$

\noindent
(We are allowed to invoke Lemma~\ref{lem:phizw} as $|\norm{Z_j}-1| < n^{-0.9/d} + \eps\alpha < 0.0001$ 
for $n$ sufficiently large, and similarly for $(1+h(U_j))U_j$.)
It follows that 

\begin{equation}\label{eq:ZivarphiZj} \begin{array}{rcl} 
\norm{Z_i - \varphi(Z_j)} 
& \leq &  
\norm{(1+h(U_i))U_i - \varphi((1+h(U_j))U_j)} + 45 \cdot \eps \cdot \alpha \\
& \leq & 
\norm{(1+h(U_i))U_i + (1+h(-U_j))U_j} + 46 \cdot \eps \cdot \alpha, 
\end{array} \end{equation}

\noindent 
where we've used Lemma~\ref{lem:phieenplushu} in the second inequality.
Applying Lemma~\ref{lem:eenplushu} once again, we have 

\begin{equation}\label{eq:UiUj} 
\norm{(1+h(U_i))U_i + (1+h(-U_j))U_j} \leq 2 \cdot \norm{U_i+U_j} 
\end{equation}
   
\noindent
Combining~\eqref{eq:ZivarphiZj} and~\eqref{eq:UiUj} we find:

$$ \norm{Z_i - \varphi(Z_j)} \leq 2 \norm{U_i+U_j} + 46 \cdot \eps \cdot \alpha. $$

\noindent
By part~\ref{itm:metricgenlb} of Lemma~\ref{lem:metric}, we have 

$$ \begin{array}{rcl} 
\dist_{S^d}(X_i,-X_j) 
& \leq & 
2\cdot \norm{\pi(X_i)-\pi(-X_j)} \\
& = & 
2 \cdot \norm{Z_i - \varphi(Z_j)} \\
& \leq & 
4 \cdot \norm{U_i+U_j} + 92 \cdot \eps \cdot \alpha.
\end{array} $$

\noindent 
To conclude the argument, we observe that $\dist(U_i,U_j) > \pi-\alpha/5$ implies that  

$$ \norm{U_i+U_j} \leq \dist(U_i,-U_j) < \alpha/5. $$

\noindent 
So

$$ \dist_{S^d}(X_i,-X_j) < (4/5 + 92\cdot\eps)\cdot \alpha < \alpha. $$

\noindent
(Provided $\eps$ was chosen sufficiently small.)
Equivalently, $\dist_{S^d}(X_i,X_j) > \pi-\alpha$.
\end{proofof}
\section{Proof of \cref{thm:main}.\label{sec:main}}

\vspace{1ex}

For $\beta > 0$ chosen appropriately shortly, we let $y_1,\dots,y_M \in S^d$ be
a $\beta$-net. That is, for every $x \in S^d$ there is 
some $i$ such that $\norm{x-y_i} < \beta$.
We let $C_i$ denote the (spherical) Voronoi cell of $y_i$, that is 

$$ C_i := \{ x\in S^d : \norm{x-y_i} = \min_{j=1,\dots,M} \norm{x-y_j} \}, $$

\noindent 
and we let 

$$ \Ycal := \{ y_i : C_i \cap \{X_1,\dots,X_N\} \neq \emptyset \}, $$

\noindent
be the set of those $y_i$ for which at least one of the Poisson points falls inside its cell.

\begin{lemma}\label{lem:aprevlem}
For every $\eps, \mu, \alpha > 0$ there exists a
$\beta=\beta(\eps,\mu,\alpha)$ such that:
 
 $$ \Pee_\lambda\left( G(\Xcal,\alpha) \cong G(\Ycal,\alpha) \right) > 1-\eps, $$
 
 \noindent
 for all $\lambda \leq \mu$, where the error term can be chosen uniform in $\lambda \leq \mu$.
\end{lemma}

\begin{proof}
Let us write 

$$ \begin{array}{rcl} 
E_\beta & := & \left\{\text{there exist $x,y\in\Xcal$ with $\dist(x,y) \in (\pi-\alpha-\pi\beta,\pi-\alpha+\pi\beta)$}\right\}, \\
F_\beta & := & \left\{\text{there exist distinct $x,y\in\Xcal$ with $\dist(x,y) < \pi\beta$}\right\}, \\
\end{array} $$

\noindent
We note that if none of these two events hold then $G(\Xcal,\alpha) \cong G(\Ycal,\alpha)$.
(We use that $\dist(x,y) \leq (\pi/2)\norm{x-y}$ for all $x,y\in S^d$ by~\eqref{eq:distnorm}.)
Now note that 

$$ 
\lim_{\beta\searrow 0} \Pee( E_{\beta} ) 
= \Pee\left( \bigcap_{\beta>0} E_{\beta} \right)
= \Pee\left(\text{there exist $x,y\in \Xcal$ such that $\dist(x,y) = \pi-\alpha$}\right) = 0. $$

Analogously, $\lim_{\beta\searrow 0} \Pee( F_{\beta} ) = 0$ as well.
Uniformity in $\lambda \leq \mu$ follows from the fact that we can 
couple the Poisson processes of intensities $\lambda, \mu$ such that 
$\Xcal_\lambda \subseteq \Xcal_\mu$.
%
\end{proof}

We fix a $\beta$-net $y_1,\dots,y_M \in S^d$ and let $\Ycal$ be as above.
Let us set 

$$ Z_i := 1_{\{y_i \in \Ycal\}} = 1_{\{C_i \cap \Xcal \neq \emptyset\}}. $$

Then $Z_1,\dots,Z_M$ are independent and $Z_i \isd \Be(p_i)$, with

$$ p_i = \Pee( C_i \cap \Xcal \neq \emptyset ) = 1 - e^{-\mu\cdot\nu(C_i)}, $$

\noindent 
where $\nu(.)$ denotes the uniform measure on the sphere.
For clarity, we emphasize that the random set $\Ycal$ is completely determined 
by $Z_1,\dots, Z_M$.

Let $A$ be the event

$$ A := \left\{ \chi( G(\Ycal,\alpha) ) > k \right\}. $$

Note that $A$ can be identified with a subset of $\{0,1\}^M$ 
(i.e.~$A = \{ (Z_1,\dots,Z_M) \in B \}$ for some $B \subseteq \{0,1\}^M$)
that is an up-set. So 
Proposition~\ref{prop:BourTobias} applies.
We now consider the derivative of $\Pee_\mu(A)$ with respect to $\mu$.
It satisfies

\begin{equation}\label{eq:muderiv} 
\begin{array}{rcl}
    \frac{\dd}{\dd\mu} \Pee_\mu ( A ) 
    & = & \displaystyle
    \sum_{i=1}^M \frac{\partial}{\partial p_i} \Pee_{\overline{p}}( A )  \cdot \frac{\dd p_i}{\dd\mu} \\
    & = & \displaystyle
    \sum_{i=1}^M \frac{\partial}{\partial p_i} \Pee_{\overline{p}}( A )  \cdot \nu(C_i) \cdot e^{-\mu\cdot\nu(C_i)} \\
    & = & \displaystyle
    \sum_{i=1}^M \frac{\partial}{\partial p_i} \Pee_{\overline{p}}( A )  \cdot \nu(C_i) \cdot (1-p_i) \\
    & \geq & \displaystyle 
    \frac{1}{3\mu} \cdot \sum_{i=1}^M \frac{\partial}{\partial p_i} \Pee_{\overline{p}}( A )  \cdot p_i(1-p_i),
\end{array} 
\end{equation}

   \noindent
where the last line holds assuming $\beta$ is chosen sufficiently small
(so that $\max_i \nu(C_i) \leq 1/\mu$ say), 
using that $1-e^{-x} \geq x/3$ for all $0\leq x \leq 1$.

\begin{corollary}
Fix $\eps,\delta > 0$. There exists a $L=L(\eps,\delta) \in \eN$ such that, for every $\mu,\alpha>0$ and $k \in \eN$ 
and all sufficiently small $\beta>0$, one of
the following holds:

\begin{enumerate}
 \item $\Pee_{\mu}(A) < \eps$, or;
 \item\label{itm:z1} There exist indices $i_1,\dots,i_L$ such that
 $\Pee_{(1+\delta)\mu}( A | Z_{i_1}=\dots=Z_{i_L}=1) > 1-\eps$.
 %
 %
 %
\end{enumerate}
\end{corollary}

\begin{proof}
We assume $\Pee_\mu(A) \geq \eps$, otherwise there is nothing to prove.

We let $c, K$ be as provided by Proposition~\ref{prop:BourTobias}
with inputs $\eps$ and $C := 1000/\delta$. 
For notational convenience, we write $x := c\delta/100$.

We first claim that there are indices $i_1,\dots,i_K$ such that 

\begin{equation}\label{eq:aap2} 
\Pee_{(1+x)\mu}(A|Z_{i_1}=\dots=Z_{i_K}=1 ) \geq \min\left( \Pee_{\mu}(A)+c, 1-\eps \right).
\end{equation}

By obvious monotonicity, $\Pee_{(1+x)\mu}(A|Z_{i_1}=\dots=Z_{i_K}=1 ) \geq \Pee_{(1+x)\mu}(A)$.
So if $\Pee_{(1+x)\mu}(A) > 1-\eps$ or $\Pee_{(1+x)\mu}(A) > \Pee_{\mu}(A)+c$ then~\eqref{eq:aap2} is trivially true
for any choice of indices $i_1,\dots,i_K$. 
We therefore assume this is not the case.
By the intermediate value theorem, there is some $\mu \leq \lambda \leq (1+x)\mu$ such that 

$$\frac{\dd}{\dd\lambda} \Pee_\lambda ( A ) < c/(x\mu). $$

\noindent
Hence, by~\eqref{eq:muderiv}, we have 

$$ \sum_{i=1}^n \frac{\partial}{\partial p_i} \Pee_{\overline{p}}( A )  \cdot p_i(1-p_i) < 
3\lambda \cdot \frac{c}{x\mu} \leq \frac{3(1+x)\mu c}{x\mu}
= \frac{300(1+\frac{c\delta}{100})}{\delta} < C, $$

\noindent
at that particular value of $\lambda$, assuming without loss of generality $\delta < 1$ in the last step.
Case {\bf(a)} or {\bf(b)} of Proposition~\ref{prop:BourTobias} must hold. 
We now note that 

$$ \begin{array}{rcl} \Pee_{\lambda}( A ) 
& \leq & \Pee_{\lambda}( A | Z_{i_1}=\dots=Z_{i_K}=0) + 
\Pee( (Z_{i_1},\dots,Z_{i_K})\neq (0,\dots, 0) ) \\
& = & \Pee_{\lambda}( A | Z_{i_1}=\dots=Z_{i_K}=0) + o_\lambda(1), 
\end{array} $$

\noindent 
where the last equality follows provided $\beta$ is small 
(so that $\Pee( Z_i = 1 ) \leq \lambda \nu(C_i) = o(1)$, uniformly in $1\leq i\leq M$).
We see that, indeed, we cannot be in case {\bf(b)}, 
because that would imply $\Pee_{\lambda}( A )  \leq \Pee_{\lambda}( A )  - c + o_\lambda(1)$.
So we are in case {\bf(a)}.
As $\Pee_{(1+x)\mu}( A | Z_{i_1}=\dots=Z_{i_K}=1) \geq 
\Pee_{\lambda}( A | Z_{i_1}=\dots=Z_{i_K}=1)$ by obvious monotonicity, this establishes the claim~\eqref{eq:aap2}.

Now let $A^{(1)}$ denote the event that $A$ holds if we force the bits with indices $i_1,\dots,i_K$ 
to be one. (I.e.~we randomly sample all bits, and then we set the values of the
selected bits to one.) 
In particular 

$$ \Pee_\lambda(A^{(1)}) = \Pee_\lambda(A|Z_{i_1}=\dots=Z_{i_K}=1), $$

\noindent
for all $\lambda$. To aid the reader, let us point out that $A^{(1)}$ is an event that depends on (is determined by)
the $M-K$ independent Bernoulli 
random variables $Z_j : j \in [M] \setminus \{i_1,\dots,i_K\}$. 
Completely analogous to the previous, we find $i_{K+1},\dots,i_{2K}$ such that

$$ \Pee_{(1+x)^2\mu}(A^{(1)}|Z_{i_{K+1}}=\dots=Z_{i_{2K}}=1 ) \geq \min\left( \Pee_{(1+x)\mu}(A^{(1)})+c, 1-\eps \right). $$

Repeating our reasoning several more times, we define events $A^{(2)}, \dots, A^{(R)}$
and indices $i_{2K+1}, \dots, i_{L}$, where $R := \lceil 1/c\rceil$ and $L := R \cdot K$.
That is, $A^{(j)}$ is the event that $A$ holds if we set $Z_{i_1}=\dots=Z_{jK}=1$, and
by our arguments:

$$ \begin{array}{rcl} \Pee_{(1+x)^{j+1}\mu}(A^{(j+1)}) 
& = & \Pee_{(1+x)^{j+1}\mu}(A^{(j)}|Z_{i_{jK+1}}=\dots=Z_{i_{(j+1)K}}=1 ) \\
& \geq & \min\left( \Pee_{(1+x)\mu}(A^{(j)})+c, 1-\eps \right), 
\end{array} $$

\noindent 
for $j=1,\dots,R-1$.
It follows that 

$$ \begin{array}{rcl} 
\Pee_{(1+x)^R\mu}( A | Z_{i_1}=\dots=Z_{i_L}=1 )
& = & \Pee_{(1+x)^R\mu}( A^{(R)} ) \\
& \geq & \min\left( \Pee_\mu(A) + R\cdot c, 1-\eps\right) \\
& = & 1-\eps, 
\end{array} $$

\noindent
where the last equality holds by choice of $R = \lceil 1/c\rceil$.

To conclude, we remark that 

$$ (1+x)^R \leq e^{xR} = e^{\lceil 1/c\rceil \cdot \frac{c\delta}{100}}
\leq e^{\frac{(c+1)\delta}{100}}
\leq e^{\delta/50} < 1+\delta, $$

\noindent
where the last inequality holds assuming without loss of generality that $\delta$ was chosen 
sufficiently small.
\end{proof}

\begin{corollary}\label{cor:contcor}
Fix $\eps,\delta > 0$. There exist
$L=L(\eps,\delta) \in \eN$ such that, for every $\mu,\alpha>0$ and $k \in \eN$, one of
the following holds:

\begin{enumerate}
 \item $\Pee_{\mu}( \chi( G(\Xcal,\alpha) ) > k ) < \eps$, or;
 
 \item There exist $x_1,\dots,x_L \in S^d$ such that
 $\Pee_{(1+\delta)\mu}( \chi(G(\Xcal \cup \{x_1,\dots,x_L\},\alpha)) > k ) > 1-\eps$.
\end{enumerate}
\end{corollary}

\begin{proof}
We notice that $\left(\Ycal|Z_{i_1}=\dots=Z_{i_K}=1\right) 
\isd \Ycal \cup \{y_{i_1},\dots,y_{i_K}\}$.
Now notice that $\Ycal \cup \{y_{i_1},\dots,y_{i_K}\}$ is
what we get if we follow the procedure constructing $\Ycal$ from $\Xcal$, but 
starting with $\Xcal \cup \{y_{i_1},\dots,y_{i_K}\}$ in place of $\Xcal$.
The result follows by (a modification of) Lemma~\ref{lem:aprevlem}.
\end{proof}

We say that distinct $y_1,\dots,y_K \in S^d$ form an $\eta$-near configuration with respect to $x_1,\dots,x_K \in S^d$
if there is an orthogonal transformation $T$ 
such that $\dist(T(y_i),x_i) < \eta$ for $i=1,\dots,K$.
(That is, we can apply an isometry of the sphere and 
map $y_i$ to a point close to $x_i$, for each $i$.)

\begin{lemma}
 For any fixed $K \in \eN, \eps > 0$ and $x_1,\dots,x_K \in S^d$, we have 
 
 $$ \Pee_\mu( \text{$\Xcal$ contains an $\eps \mu^{-1/d}$-near configuration} ) = 1 - o_\mu(1). $$
 
\end{lemma}

\begin{proof}
We (greedily) pick orthogonal transformations $T_1,\dots,T_L$ such that 
$\dist\left(T_i(x_j),T_{i'}(x_{j'})\right) > 2 \eps \mu^{-1/d}$ for all
$(i,j)\neq(i',j')$.
This can be done for $L = \Omega( \mu )$, which can be seen as follows.
Suppose $T_1,\dots,T_L$ are such that we cannot find a suitable orthogonal transformation $T_{L+1}$.
If we pick an orthogonal transformation $T \in O(d+1)$ uniformly at random (this makes sense 
as the set of all orthogonal $(d+1)\times(d+1)$-matrices $O(d+1)$ can be seen as a smooth and compact manifold), then 
with probability one, at least one of the random points $Y_1 = T(x_1),\dots,Y_K = T(x_K)$ falls within 
the set  

$$ A := \bigcup_{1\leq i \leq L, \atop 1\leq j\leq K} \kap( T_i(x_j), 2\eps\mu^{-1/d} ). $$

By symmetry considerations, each of $Y_1,\dots,Y_K$ has the uniform distribution on $S^d$ (but they are not independent).
We see that 

\begin{equation}\label{eq:eendje} 
\begin{array}{rcl} 
1 
& = & \Pee( \{Y_1,\dots,Y_K\} \cap A \neq \emptyset )
\leq \sum_{j=1}^K \Pee( Y_j \in A ) \\
& = & K \nu(A) \leq K \sum_{i=1}^L \sum_{j=1}^K \nu( \kap( T_i(x_j), 2 \eps \mu^{-1/d} ) ) \\
& \leq & L \cdot K^2 \cdot \text{const} \cdot \eps^d \cdot (1/\mu), 
\end{array}
\end{equation}

\noindent
where $\text{const}$ is a constant depending only on $d$ (the measure of a cap with opening radius $x$ scales
with $x^d$).
Rephrasing,~\eqref{eq:eendje} states that $L = \Omega(\mu)$, as claimed.

For $1\leq i \leq L$, let us define the events

$$
E_i := \bigcap_{j=1}^K E_{i,j}, \quad 
E_{i,j} :=  \{ \Xcal \cap \kap(T_i(x_j),\eps\mu^{-1/d}) \neq \emptyset \}. $$

Then by choice of $T_1,\dots, T_L$, the events $E_1,\dots,E_L$ are independent and, 
by symmetry considerations, they all have the same probability.
This gives

\begin{equation}\label{eq:kikker1} 
\Pee\left( \bigcup_{j=1}^L E_j \right) = 1 - \left(1 - \Pee(E_1)\right)^L.  
\end{equation}

\noindent
The events $E_{1,1},\dots,E_{1,K}$ are positively dependent, giving

\begin{equation}\label{eq:kikker2} 
\Pee( E_1 ) \geq \Pee(E_{1,1})\cdot\dots\cdot\Pee(E_{1,K}) =
\left(1 - \exp[ - \mu \cdot \nu( \kap(x,\eps\mu^{-1/d})) ] \right)^K = \Omega(1), 
\end{equation}

\noindent
Combining~\eqref{eq:kikker1} and~\eqref{eq:kikker2} with $L = \Omega(\mu)$, we see that 
$\Pee\left( \bigcup_{j=1}^L E_j \right) = 1 - o_\mu(1)$, which proves the lemma.
\end{proof}

\begin{corollary}\label{cor:etacor}
For every $\eps,\delta,\eta > 0$ and $3\leq k \leq d+2$, one of the following holds for all sufficiently 
large $\mu$ and all $\alpha>0$:
\begin{enumerate}
 \item $\Pee_\mu( \chi( G(\Xcal,\alpha)) > k ) < \eps$, or;
 \item $\Pee_{(1+\delta)\mu}( \chi( G(\Xcal,\alpha+\eta\mu^{-1/d})) > k ) \geq 1-\eps$.
\end{enumerate}
\end{corollary}

\noindent
(Just to be clear, this holds uniformly over all $\alpha$. In particular 
$\alpha = \alpha(n)$ is allowed to vary with $n$.)

\begin{proof}
Assume $\Pee_\mu( \chi( G(\Xcal,\alpha)) > k ) \geq \eps$.
By Corollary~\ref{cor:contcor}, 
there are $x_1,\dots,x_K$ such that 
$\Pee_{(1+\delta)\mu}( \chi( G(\Xcal \cup \{x_1,\dots,x_K\},\alpha)) > k ) \geq 1-\eps$.

Now consider the model with parameter $(1+2\delta)\mu$.
The point set $\Xcal$ is the union of $\Xcal_1 \cup \Xcal_2$
where $\Xcal_1$ is a PPP with parameter $(1+\delta)\mu$ and $\Xcal_2$ 
an independent PPP with parameter $\delta\mu$.

By the previous lemma, with probability $1-o(1)$, $\Xcal_2$ contains 
at least one $(\eta/2)\cdot\mu^{-1/d}$-near configuration $Y_1,\dots,Y_K$.
Let $T$ be a corresponding (random) orthogonal transformation. That is, 
$\dist( T(Y_1), x_1 ), \dots, \dist( T(Y_K), x_K ) < (\eta/2)\cdot\mu^{-1/d}$. 
(For definiteness, we can say that we pick $Y_1,\dots,Y_K$ 
uniformly at random from all eligible $K$-tuples and, having chosen 
$Y_1,\dots,Y_K$, we pick $T$ 
uniformly at random from all eligible transformations. Again
this last step makes sense because we are picking from an
open subset of a compact, smooth manifold.)

We now note that 

$$ \begin{array}{rcl} G( \Xcal, \alpha+\eta\mu^{-1/d} ) & \supseteq & 
G( \Xcal_1 \cup \{Y_1,\dots,Y_K\}, \alpha+\eta\mu^{-1/d} ) \\
& \cong & G( T\left[\Xcal_1 \cup \{Y_1,\dots,Y_K\}\right], \alpha+\eta\mu^{-1/d} ) \\
& \supseteq &  G( T\left[\Xcal_1\right] \cup \{x_1,\dots,x_K\}, \alpha ), 
\end{array} $$

\noindent
where $\subseteq$ denotes the subgraph relation and in the last 
step we use the triangle inequality for the spherical distance $\dist(.,.)$.
 
The random orthogonal transformation $T$ is determined entirely by $\Xcal_2$ and 
in particular is independent of $\Xcal_1$.
It follows that $T[\Xcal_1]\isd \Xcal_1$ by the symmetry properties 
of a constant intensity measure $\lambda\cdot \nu(.)$ on the $d$-sphere.
This gives

$$ \begin{array}{rcl} 
\Pee_{(1+2\delta)\mu}\left( \chi\left(G( \Xcal, \alpha+\eta\mu^{-1/d} )\right) > k \right)
& \geq & 
\Pee_{(1+2\delta)\mu}\left( \chi\left(G( T\left[\Xcal_1\right] \cup \{x_1,\dots,x_K\}, \alpha )\right) > k \right)-o(1)\\
& = & \Pee_{(1+2\delta)\mu}\left( \chi\left(G( \Xcal_1 \cup \{x_1,\dots,x_K\}, \alpha )\right) > k \right) -o(1) \\
& = & \Pee_{(1+\delta)\mu}\left( \chi\left(G( \Xcal \cup \{x_1,\dots,x_K\}, \alpha )\right) > k \right) - o(1)\\
& \geq & 1-\eps-o(1). \end{array} $$

\noindent
(The $o(1)$ term accounts for the event that $\Xcal_2$ does not have the required near configuration.)
Adjusting the values of $\eps,\delta$, the result follows.
\end{proof}

For $n\in\eN$ and $3\leq k \leq d+1$ we let $\alpha_k(n)$ be such that

$$ \Pee_n( \chi(G(\Xcal,\alpha_k)) > k ) = 1/2, $$

\noindent
(This $\alpha_k$ exists by a straightforward continuity argument.)
We remark that $\alpha_k$ is non-increasing.
By~\eqref{eq:KF} and Theorems~\ref{bt1} and~\ref{thm:maind+1}, there exists a constant $c > 0$ such that 

\begin{equation}\label{eq:muis} 
\alpha_k(n) \geq c n^{-1/d},  
\end{equation}

\noindent
and also 

\begin{equation}\label{eq:walnoot}
\alpha_k(2n) > c \cdot \alpha_k(n), 
\end{equation}

\noindent 
for all $3\leq k \leq d+1$ and all $n \in \eN$.

\vspace{1ex}

We fix a $3\leq k \leq d+1$ and $\eps,\delta > 0$ and define

\begin{equation}\label{eq:UVdef} 
U := \Big\{ n \in \eN : \alpha_k(n) \cdot (1-\eps/2) > \alpha_k\left( \left\lceil (1+\delta) n\right\rceil\right) \Big\}, 
\quad V := \left\{ n \in \eN : \left\lfloor\frac{n}{1+\delta}\right\rfloor \in U \right\}. 
\end{equation}

\begin{lemma}
We have $\displaystyle \frac{|U \cap [n]|}{n} \leq 100 \cdot \ln(1/c) \cdot (\delta/\eps)$, for all $n \in \eN$.
\end{lemma}

\begin{proof}
For some $\ell \in \eN$, let $2^\ell \leq n_1 < n_2 < \dots n_K < 2^{\ell+1}$ be a 
sequence with the properties that $n_{i+1} \geq (1+\delta)n_i$ and 
$\alpha_k( n_{i+1} ) < (1-\eps/2) \alpha_k(n_i)$ and, subject to these constraints, $K$ is 
as large as possible.
We must have 

$$ \alpha_k( 2^{\ell+1} ) \leq (1-\eps/2)^K \alpha_k(2^\ell). $$

\noindent
So, by~\eqref{eq:walnoot}, we have $c < (1-\eps/2)^K$. Rewriting, we have

$$ K < \frac{\ln c}{\ln(1-\eps/2)} = \frac{\ln(1/c)}{-\ln(1-\eps/2)} \leq \frac{2\ln(1/c)}{\eps}. $$

\noindent
By maximality of $K$, every $n \in U \cap \{2^\ell, \dots, 2^{\ell+1}-1\}$ must satisfy 
$n_i/(1+\delta) < n < (1+\delta)n_i$ for some $1\leq i \leq K$.
It follows that 

$$ \left| U \cap \{2^\ell, \dots, 2^{\ell+1}-1\} \right| \leq 
\sum_{i=1}^K 2\delta\cdot n_i \leq K \cdot 2\delta \cdot 2^{\ell+1}
\leq 8\ln(1/c) \cdot (\delta/\eps) \cdot 2^\ell. $$

Now let $n$ be arbitrary and let $\ell \in \eN$ be the value such that 
$2^\ell \leq n < 2^{\ell+1}$. 
Then 

$$ |U \cap [n]| \leq \sum_{i=0}^\ell |U \cap \{2^i, \dots, 2^{i+1}-1\}|
\leq 8\ln(1/c) \cdot (\delta/\eps) \cdot \sum_{i=0}^\ell 2^i 
\leq 16\ln(1/c) \cdot (\delta/\eps) \cdot n. $$

\noindent
This proves the result.
\end{proof}

\begin{corollary}
We have $\displaystyle \frac{|V \cap [n]|}{n} \leq 1000 \cdot \ln(1/c) \cdot (\delta/\eps)$.
\end{corollary}

\begin{proof}
This follows from the previous lemma and the fact that if $u = \left\lfloor\frac{n}{1+\delta}\right\rfloor$ then 
$n \in \left\{ \left\lceil u(1+\delta)\right\rceil, \left\lceil u(1+\delta)\right\rceil + 1 \right\}$
(assuming without loss of generality $\delta<1$).
\end{proof}

\begin{lemma}\label{lem:elf}
There exists $n_0 = n_0(\eps,\delta,k)$ such that for all $n \in \{n_0,n_0+1,\dots\} \setminus (U\cup V)$ we have 
\begin{enumerate}
\item\label{itm:x1} $\Pee_n\left( \chi( G(\Xcal,(1+\eps)\alpha_k(n)))  > k \right) > 1-\eps$, and;
\item\label{itm:x2} $\Pee_n\left( \chi( G(\Xcal,(1-\eps)\alpha_k(n))) > k \right) < \eps$.
\end{enumerate}
\end{lemma}

\begin{proof}
 Pick $n \in \eN \setminus (U\cup V)$.
 Writing $m := \left\lfloor\frac{n}{1+\delta}\right\rfloor$ for convenience, we have 
 
$$ (1+\eps)\alpha_k(n) \geq 
(1+\eps)(1-\eps/2)\alpha_k\left(m\right) >  \alpha_k\left(m\right) + \eta\cdot m^{-1/d},     
$$
 
\noindent
where $\eta>0$ is a sufficiently small constant and the last inequality holds for $n$ sufficiently large, 
and we use~\eqref{eq:muis}.
Corollary~\ref{cor:etacor} now shows that~\ref{itm:x1} holds for $n$ sufficiently large.

It remains to establish~\ref{itm:x2}.
Aiming for a contradiction, suppose that

$$ \Pee_n( \chi(G(\Xcal, (1-\eps)\alpha_k(n))) > k ) \geq \eps. $$

Writing $\ell := \lceil(1+\delta)n\rceil$, Corollary~\ref{cor:etacor} tells us that 

$$ \Pee_\ell( \chi(G(\Xcal, (1-\eps)\alpha_k(n)+\eta n^{-1/d})) > k ) > 1-\eps. $$

However, $\alpha_k(\ell) \geq (1-\eps/2)\alpha_k(n)$ as $n \not\in U$ and hence

$$ (1-\eps)\alpha_k(n)+\eta n^{-1/d} \leq \frac{1-\eps}{1-\eps/2} \cdot \alpha_k(\ell) +\eta n^{-1/d}
< \alpha_k(\ell), $$

\noindent
for $n$ sufficiently large and we use~\eqref{eq:muis} again, and that $\eta$ is an appropriately chosen small constant.
But now it follows that, for $n$ sufficiently large: 

$$ \Pee_\ell( \chi(G(\Xcal, \alpha_k(\ell))) > k ) > 1-\eps, $$ 

\noindent
contradicting the definition of $\alpha_k$.
It follows that, for $n \in \eN \setminus (U\cup V)$ sufficiently large, we have 

$$ \Pee_n( \chi(G(\Xcal, (1-\eps)\alpha_k(n))) > k ) < \eps, $$

\noindent 
proving~\ref{itm:x2}.
\end{proof}

\begin{proofof}{Theorem~\ref{thm:main}}
Fix $3\leq k \leq d+1$.
We choose sequences $\eps_1,\eps_2,\dots$ and $\delta_1,\delta_2, \dots$ 
satisfying $\eps_i \to 0$ and $\delta_i/\eps_i \to 0$ as $i \to \infty$.
Let $U_i = U(k,\eps_i,\delta_i), V_i = V(k,\eps_i,\delta_i)$ be as defined in~\eqref{eq:UVdef}
and let $n_i = n_0(k,\eps_i,\delta_i)$ be as provided by Lemma~\ref{lem:elf}.
Without loss of generality (switching to even larger numbers if needed), we can assume that 
$n_{i+1}/n_i \to \infty$ as $i\to\infty$.

We define 

$$ N_k := \bigcup_{i=1}^\infty \{n_i, \dots, n_{i+1}-1\} \setminus (U_i\cup V_i). $$

Note that if $n_i \leq n < n_{i+1}$ then 

$$ \begin{array}{rcl} \frac{|[n]\setminus N_k|}{n} & \leq & \frac{|U_i \cap [n]|}{n} + \frac{|V_i \cap [n]|}{n}
+ \frac{|U_{i-1} \cap [n_i]|}{n_i} + \frac{|V_{i-1} \cap [n_i]|}{n_i}
+ \frac{n_{i-1}}{n_i} \\
& \leq & O(\delta_i/\eps_i) + O(\delta_{i-1}/\eps_{i-1}) + o(1) \\ & = & o(1). \end{array} $$

\noindent
In other words, $N_k$ has density one. It follows that

$$ N := N_4\cap\dots\cap N_{d+2}, $$

\noindent
also has density one.
Theorem~\ref{thm:main} now follows by the previous lemma (using $\eps_i \to 0$).
\end{proofof}

\section{Proof of Theorem~\ref{bt1}.\label{sec:bt1}}

For this part, it will be the easiest to compare the projected Borsuk graph locally to the continuum $AB$ percolation model.
In this model, we let $ \mathcal{V}_A $ and  $ \mathcal{V}_B $ be independent Poisson point processes in $ \RR^d $ where each has intensity $ \lambda>0 $. We create a random graph $(\mathcal{V},E) $ with vertices $ \mathcal{V}=\mathcal{V}(\lambda)\coloneqq\mathcal{V}_A\cup\mathcal{V}_B $ 
and edges
	\begin{equation*}
		E\coloneqq\left\{ xy \in {\mathcal{V}\choose 2} \,:\, \norm{x-y}\leq 1, x\in 
		\mathcal{V}_A, y\in\mathcal{V}_B\right\}.
	\end{equation*} 
	With a slight abuse of notation we will also write $\mathcal{V}=\mathcal{V}(\lambda)$ for the graph and  define percolation to be
	\begin{equation*}
	\{\text{percolation}\}\coloneqq\{\text{there exists an infinite connected component in }\mathcal{V}\}.
	\end{equation*}
	Since the law of $ (\mathcal{V}_A,\mathcal{V}_B) $ is ergodic (see e.g. Section 8.4 of \cite{Last_Penrose_2017}) 
	and the occurrence of an infinite component is translation invariant, 
	it follows that 
	
	$$\Pro_\lambda(\text{percolation})\in\{0,1\},$$ 
	
	\noindent
	for all $\lambda$. We define 
\begin{equation*}
	\lambda_c\coloneqq\inf\{\lambda>0\,:\Pro_\lambda(\text{percolation})>0\,\}.
\end{equation*}

\begin{lemma}\label{nontriv}
	$\lambda_c\in (0,\infty) $.
\end{lemma}
\begin{proof}
	In the classical Boolean percolation model, we connect any two points of a PPP 
	with constant intensity $ \lambda>0 $ if they are closer than one. As for continuum AB percolation we can define
	\begin{equation*}
	 \lambda_c^{\text{Bool}}\coloneqq\inf\{\lambda>0\,:\Pro_{\lambda}(\text{percolation occurs in the Boolean model\
	 })>0\,\}.
	\end{equation*} 
	It is well known (see e.g. Chapter 3 of \cite{meester1996continuum}) that $ \lambda_c^{\text{Bool}} \in (0,\infty)$.
	Note that $\Vcal_A \cup \Vcal_B$ forms a PPP of intensity $2\lambda$. So if $\lambda < \lambda_c^{\text{Bool}}/2$ then 
	the continuum AB model certainly does not percolate.
	This gives  
	
	$$\lambda_c \geq \lambda_c^{\text{Bool}}/2 > 0. $$
	
	For the upper bound we let $L\coloneqq 1/(2\sqrt{d})$ and tile $\RR^d$ by disjoint 
	cubes $Q_z\coloneqq Lz+[0,L)^d$ for $z\in \ZZ^d$. 
	We couple the continuum AB percolation model with the standard site percolation model on $\ZZ^d$ by declaring $z$ 
	open iff.~$Q_z$ contains at least one 
	point of $\mathcal{V}_A$ and at least one point of $\mathcal{V}_B$. We have 
	
	$$\Pro(z\text{ open})=(1-e^{-\lambda L^d})^2,$$ 
	
	\noindent 
	and 
	the status of a site is independent of the other sites. 
	We can choose $\lambda$ large enough such that $\Pro(z\text{ open})>p_c^\text{site}(\ZZ^d)$. 
	In particular, for $\lambda$ sufficiently large, the site percolation model has an infinite component 
	(with probability one).
	By choice of $L$, for $x\in Q_z$ and $y\in Q_{z'}$ with $\norm{z-z'}=1$ we have $\norm{x-y}^2\leq 1$.
        So if there exists an infinite component in the coupled standard site percolation model on $\Zed^d$, then the 
        $AB$ model has an infinite connected component too. We can conclude that $\lambda_c<\infty$.
\end{proof}

\noindent
Having defined continuum $AB$ percolation, we can now go back to the Borsuk graph. For $ -1<t<1 $ and $ y\in S^d $ we define 
\begin{equation*}
	C_t(y)\coloneqq \text{cap}(y,\arccos{t})=\{ x \in S^d : \dist(y,x) < \arccos{t} \}=\{x\in S^d\,:\, \langle x,y\rangle >t\}
\end{equation*}
and $ C_t\coloneqq C_t((0,\dots,0,-1)) =\{x\in S^d\,:\,x_{d+1}<-t\}$. We write $ \pi_{y}:S^d\setminus\{-y\}\to P_y $ for the stereographic projection  from $ S^{d} $ through $ -y\in S^{d} $ onto the $ d$-dimensional plane $P_y\coloneqq \{x\in \RR^{d+1} \,:\, \langle x,y\rangle=0\} $. The choice of defining it with respect to $ -y $ and not $ y $ was made, as we will use the map to approximate $ S^d $ around $ y $ with $ \RR^d $. Thus, the projection we have used so far is $\pi(x)=\pi_{(0,\dots,0,-1)}(x)$. For $ x,y\in S^d $ we also define the map
\begin{equation*}
	g_y(x)\coloneqq  \frac{2\pi_y(x)}{\alpha(n)}
\end{equation*}
that first projects and then scales points of the sphere. We use the short form $g\coloneqq g_{(0,\dots,0,-1)}$. We also let 
\begin{equation*}
	\rho(t,n)\coloneqq \frac{2\sqrt{1-t^2}}{(1+t)\alpha(n)}.
\end{equation*}
Note that in the upcoming lemma,  $ B(\typ,\rho(t,n)) $ is meant to be a $ d $-dimensional ball that lives in the $ \{x_{d+1}=0\}$ plane. 
The value that will arise in our connection threshold is now set to be 
\begin{equation*}
	c_2=c_2(d)\coloneqq {((d+1)\kappa_{d+1}\lambda_c)^{1/d}},
\end{equation*}
and because of \cref{nontriv} this is a well defined positive constant (dependent only on $d$).

\begin{lemma}\label{ppp}
For all $\eps>0$ we can choose $ t_0=t_0(\eps,d) $ sufficiently close to 
$1$ such that for all $t>t_0$ the following holds. 
The points of $ g[\mathcal{X}\cap C_t] $ form an inhomogeneous Poisson point process on 
$ B(\typ,\rho(t,n)) $. If $\alpha=(c_2+\eps)n^{-1/d}$ we have that $ g[\mathcal{X}\cap C_t] $ 
stochastically dominates a homogeneous PPP with intensity $ \lambda=\lambda(\eps,d)>\lambda_c $ and 
if $\alpha=(c_2-\eps)n^{-1/d}$ it is stochastically dominated by a PPP with intensity $ \lambda=\lambda(\eps,d)<\lambda_c $.
\end{lemma}

\begin{proof}
	For $ x\in S^d $ with $ -x_{d+1}>t $ it holds that
	\begin{equation*}
		\norm{\pi(x)}=\frac{\sqrt{x_1^2+\dots+x_d^2}}{1-x_{d+1}}=\frac{\sqrt{1-x^2_{d+1}}}{1-x_{d+1}} \leq \frac{\sqrt{1-t^2}}{1+t}.
	\end{equation*}
	And thus $ \pi(C_t) $ maps into $ B(\typ,\sqrt{1-t^2}/(1+t)) $. 
	The fact that $ \pi[\mathcal{X}\cap C_t] $ is an inhomogeneous Poisson point process with intensity measure
	\begin{equation*}
		\frac{n}{(d+1)\kappa_{d+1}}\bigg(\frac{2}{1+\norm{y}^2}\bigg)^d\indicator{\norm{y}\leq \sqrt{1-t^2}/(1+t)},
	\end{equation*}
	follows from \cref{lem:piunif}. As $ 0\leq\norm{y}^2\leq (1-t^2)/(1+t)^2  $ we can bound the intensity uniformly above and below by
	\begin{equation*}
    \begin{split}
        \frac{n}{(d+1)\kappa_{d+1}}2^d\geq \frac{n}{(d+1)\kappa_{d+1}}\bigg(\frac{2}{1+\norm{y}^2}\bigg)^d&\geq \frac{n}{(d+1)\kappa_{d+1}}\bigg(\frac{2(1+t)^2}{(1+t)^2+1-t^2}\bigg)^d\\&= \frac{n  }{(d+1)\kappa_{d+1}}(1+t)^d.
    \end{split}
	\end{equation*}
	By the mapping theorem (see e.g. Section 2.3 of \cite{Kingmanboek}), our result now follows after scaling each 
	particle with $  2/\alpha(n)$. For the lower bound we can compute in the $\alpha=(c_2+\eps)n^{-1/d}$ case
    \begin{equation*}
         \frac{n (1+t)^d }{(d+1)\kappa_{d+1}} \frac{\alpha(n)^d}{2^d} = 
         \bigg(\frac{1+t}{2}\bigg)^d \frac{(((d+1)\kappa_{d+1}\lambda_c)^{1/d}+\eps)^d}{(d+1)\kappa_{d+1}}
         =\lambda_c \bigg(1+\frac{\eps}{c_2}\bigg)^d\bigg(\frac{1+t}{2}\bigg)^d
    \end{equation*}
    and this quantity is strictly larger than $\lambda_c$, given that $t$ is sufficiently close to 1. In a similar fashion we compute for the upper bound on the intensity, if we have 
    $\alpha=(c_2-\eps)n^{-1/d}$ that
    \begin{equation*}
        \frac{n2^d}{(d+1)\kappa_{d+1}}\frac{\alpha(n)^d}{2^d} 
        =\frac{(((d+1)\kappa_{d+1}\lambda_c)^{1/d}-\eps)^d}{(d+1)\kappa_{d+1}}
        =\lambda_c\bigg(1-\frac{\eps}{c_2}\bigg)^d,
    \end{equation*}
    and this quantity is less than $\lambda_c$ in any case.
\end{proof}

\noindent
It is sometimes easier to view the Borsuk graph as a random geometric graph on the sphere. 
To achieve that, we will mirror every $ X_i $ and define $ Y_{i}\coloneqq -X_{i} $ and 
$ \mathcal{M}\coloneqq\{Y_1,\dots,Y_N\} $. Then, we connect $ X_i $ to $ Y_j $ iff $ \dists(X_i,Y_j)\leq \alpha $.
We write $ G_{geo} $ for the resulting graph and for $ x\in S^d $ we let $ G_t(x) $ be the restriction of $ G_{geo} $ 
to $ C_t(x) $. We sometimes also use the shorter 
notation $ G_t\coloneqq G_t((0,\dots,0,-1)) $. The following two lemmas help us compare the Borsuk graph with 
continuum $AB$ percolation. Depending on $\alpha$, it will be supercritical in the first, and subcritical in the second case.

\begin{lemma}\label{coupleab}
	For $ \eps>0 $ and $\alpha=(c_2+\eps)n^{-1/d}$ we can choose $t\geq t_0(\eps,d)$ sufficiently close to 1 such that there exists a coupling of $ G_t $ with an continuum AB percolation model with intensity $ \lambda=\lambda(\eps,t)>\lambda_c$ on $ B(\typ,\rho(t,n)) $ such that for any two points $ x,y $ that are adjacent in the continuum AB percolation model, it holds that $ g^{-1}(x),g^{-1}(y) $ are adjacent vertices in $ G_t $.
\end{lemma}
\begin{proof}
	Both $ \mathcal{X}\cap C_t $ and $ \mathcal{M}\cap C_t $ are independent with an identical distribution. By \cref{ppp}, if we project these points with $ g(\cdot) $, each set stochastically dominates a PPP with intensity $ \lambda=\lambda(\eps,t)>\lambda_c$ on $ B(\typ,\rho(t,n)) $. By \cref{lem:metric} it holds for any $ x,y\in C_t $
	\begin{equation*}
		\norm{g(x)-g(y)} = \frac{2}{\alpha(n)}\norm{\pi(x)-\pi(y)}\geq \frac{\dists(x,y)}{\alpha(n)}.
	\end{equation*}
	Thus, if $ g(x)\in B(\typ,\rho(t,n)) $ has an edge to $ g(y)\in B(\typ,\rho(t,n)) $ in the continuum AB percolation model it holds $\norm{g(x)-g(y)}\leq 1$ and thus we obtain
	\begin{equation*}
		\dists(x,y)\leq \alpha(n)\norm{g(x)-g(y)}\leq \alpha(n),
	\end{equation*}
	 so $ x $ is a neighbor of $ y $ in $ G_t $.
\end{proof}

\noindent
Note that by rotational invariance we could also formulate the above for any other cap $ C_t(y)$, part of the 
graph $ G_t(y) $ and projection map $ g_y(\cdot) $. The same is true for the following.

\begin{lemma}\label{coupleab2}
	For $ \eps>0 $ and $\alpha=(c_2-\eps)n^{-1/d}$ we can choose $ t\geq t_0(\eps,d) $ sufficiently close to 1 such that there is a constant $c=c(\eps,d)>0$ that satisfies the following. There exists a coupling of $ G_t $ with an continuum AB percolation model with $ \lambda=\lambda(\eps,d) <\lambda_c $ and 
	on $ B(\typ,\rho(t,n)) $ such that for any two points $ x,y $ that are adjacent in $ G_t $, it holds that  $ cg(x)$ and $cg(y) $ are adjacent in the continuum AB percolation model.
\end{lemma}
\begin{proof}
	For all $ \eps_1>0 $ we obtain by \cref{lem:metric} (\rom{2}) a $ t=t(\eps_1) $ such that for all neighbors $ x,y \in G_t$ it holds
	\begin{equation*}
		\alpha(n)=(c_2-\eps)n^{-1/d}\geq \dists(x,y)> \frac{2\norm{\pi(x)-\pi(y)} }{1+\eps_1}.
	\end{equation*}
	Thus by rearranging
	\begin{equation*}
		\begin{split}
			\norm{g(x)-g(y)}&=\frac{2}{\alpha(n)}\norm{\pi(x)-\pi(y)}< {1+\eps_1}.
		\end{split}
	\end{equation*}
    Note that if $c\coloneqq 1/(1+\eps_1)$, we of course have $\norm{cg(x)-cg(y)}<1$. By the mapping theorem (see e.g. Section 2.3 of \cite{Kingmanboek}), if we scale every point of a subcritical $AB$ percolation model by the same scalar that is sufficiently close to 1, we obtain a new $AB$ percolation model that is still subcritical. Therefore, if we now choose $ \eps_1=\eps_1(\eps,d)>0 $ small enough, we get what we wanted to show by \cref{ppp}.
\end{proof}

\subsection{Continuum AB percolation machinery}

To show an appropriate subcritical exponential decay property, we will approximate $ \mathcal{V} $ by site percolation on a discrete version $ G'=G'(\eps,r_1) $, where $ \eps,r_1>0 $ are parameters to be chosen later. The graph $G'$ is embedded in $\ZZ^{d+1}$ and can be viewed as two copies of $\ZZ^d$ hovering above each other. One copy will be coupled 
with points of $\mathcal{V}_A$ and the other with $\mathcal{V}_B$. For each $ i=(i_1,\dots,i_d)\in\ZZ^d $ 
we will add a site at $(i_1,\dots,i_d,0)$ and $(i_1,\dots,i_d,1)$. 
We define cubes that are anchored at $ i $ with diameter $ \eps $ as
\begin{equation*}
	S_{i}\coloneqq\bigg\{(x_1,\dots,x_d)\,:\,i_1\eps/\sqrt{d}\leq x_1< (i_1+1)\eps/\sqrt{d},\dots,i_d\eps/\sqrt{d}\leq x_d< (i_d+1)\eps/\sqrt{d} \bigg\}.
\end{equation*} 
 We will connect $(i_1,\dots,i_d,0)$ to $(j_1,\dots,j_d,1)$ iff
 \begin{equation*}
 	\sup_{u\in S_i,v\in S_j}\norm{u-v}\leq r_1
 \end{equation*}
and write $C_0'$ for the connected component of $\typ\times\{0\}$ after performing site percolation on $G'$.
 In the following we will add $ \typ $ as a deterministic point of $ \mathcal{V}_A $ (or $ \mathcal{V}_B $, since the intensities for the two PPP's are the same, the next results will hold in either case) to $ \mathcal{V} $ and write $ C_0 $ for the connected component of $ \typ $ in $ \mathcal{V} $. For two points $z,w\in \mathcal{V}$ we write $z\longleftrightarrow w$ if there is a path connecting $z$ to $w$ in $\mathcal{V}$ and $z{\nocon} w$ if not. In a similar fashion if $A\subset \RR^d$ we write $z\overset{A}{\longleftrightarrow}w$ if there is a path connecting the two points while staying inside $A$ and $z\longleftrightarrow A$ if $z$ is connected to some point in $A$.
\begin{lemma}\label{abdecay}
	Let $ \lambda<\lambda_c $. Then, there exists a constant $ c=c(\lambda,d)>0 $ such that for all $ R>0 $
	\begin{equation*}
		\Pro(\typ\longleftrightarrow B(\typ,R)^c)\leq e^{-cR}.
	\end{equation*}
\end{lemma}
	\begin{proof}
		We follow the ideas of Theorem 3 in Section 8.1 in \cite{bollobas2006percolation} and aim to use a 
		classical result on exponential decay on $ G' $. 
		We could have also defined $AB$ percolation with an arbitrary radius $r>0$ instead of $1$ and connect 
		two points of opposite label iff the distance is less than $r$. 
		Then, we can define $\lambda_c(r)$ in terms of the radius analogously as done for $\lambda_c$. 
		Note that scaling gives us the relationship 
		
		$$ \lambda_c(r_1)\cdot r_1^d=\lambda_c(r_2)\cdot r_2^d, $$
		
		\noindent 
		for all $r_1,r_2>0$.
        Therefore, we can let $ r_1>1 $ and $ \lambda_1>0 $ be such that $ \lambda<\lambda_1<\lambda_c(r_1)<\lambda_c$. Further, choose $ \eps>0 $ small enough such that $ 1+2\eps<r_1 $. We now let $ \mathcal{V}_A' $ and $ \mathcal{V}_B' $ be independent PPP's in $ \RR^d $ where each has intensity $ \lambda_1>0 $. We consider site percolation on $ G'(\eps,r_1)$ and declare $ S_{i}\times\{0\}$ open if $ \mathcal{V}_A'\cap S_{i}\neq \emptyset $ and $ S_{i}\times\{1\} $ if $ \mathcal{V}_B'\cap S_{i}\neq \emptyset $. Thus, all sites are independently open with probability $ p_1=1-e^{-\lambda_1(\eps/\sqrt{d})^d} $. We note that an infinite component in $ G'(\eps,r_1) $ would imply an infinite cluster in the $AB$ model where we connect points if they are closer than $ r_1 $ to each other instead of $1$ w.r.t. $\mathcal{V}_A',\mathcal{V}_B'$. So it has to hold that $ p_1 $ cannot be supercritical for $ G' $, or in other words any $ p<p_1 $ is subcritical. Thus, if we define the open cubes on $ G'(\eps,r_1) $ w.r.t. $ \mathcal{V}_A $ and $ \mathcal{V}_B $, we are subcritical, as $ \lambda<\lambda_1 $. Since $ 1+2\eps<r_1 $, the corresponding cube for every point of $ C_0 $ with respect to $ \mathcal{V}(\lambda) $ and connection radius 1 will be part of $ C_0' $. Thus if it holds that $\typ\longleftrightarrow B(\typ,R)^c$ we will have $|C'_0|\geq c_1 R$ for some $c_1=c_1(\eps,d,r_1)>0$. This yields by \cref{bolloexpdecay} (we can apply it to our case, since a ball of radius $ n $ in $ G' $ contains no more than $ 2(nr_1/\eps)^d  $ vertices) a constant $c_2=c_2(\lambda,\eps,d)>0$ with
        \begin{equation*}
            \Pro(\typ\longleftrightarrow B(\typ,R)^c)\leq \Pro(|C_0'|\geq c_1R)\leq e^{-c_1c_2R}.
        \end{equation*}
	\end{proof}
	
\begin{corollary}\label{abdecaycor}
	For all  $\lambda<\lambda_c$ there exists a constant $ c=c(d,\lambda)>0 $ such that 
	for all $ k\in\NN $, setting
	\begin{equation*}
		M_{k}\coloneqq|\{a\in B(\typ,k/2)\cap \mathcal{V}_A\,:\, a\longleftrightarrow  B(\typ,k)^c \}|,
	\end{equation*}
	we have:
	\begin{equation*}
		\Pro(M_{k}>0)<e^{-ck}.
	\end{equation*}
\end{corollary}

\begin{proof}
By \cref{abdecay} and the Slivnyak-Mecke formula (see e.g. Corollary 3.2.3 in \cite{schneider2008stochastic}) we obtain a constant $ c_1=c_1(\lambda,d)>0 $ such that
	\begin{equation*}
		\begin{split}
			\Pro(M_{k}>0)\leq \Ex M_k&=\int_{B(\typ,k/2)}\Pro(a\longleftrightarrow  B(\typ,k)^c\text{ w.r.t. }\mathcal{V}_A\cup\{a\})\lambda\,\text{d}a
			\\&\leq\lambda \vol(B(\typ,k/2)) e^{-c_1k}\leq e^{-c_2k},
		\end{split}
	\end{equation*}
where the last inequality holds for some constant $ c_2=c_2(d,\lambda)>0 $.
\end{proof}

\noindent
For $ S\subset \RR^d $ we will say that $ S $ is \textit{covered by} $ \mathcal{V} $ if for every $ y\in S $ we have $ \mathcal{V}_A\cap B(y,1/4)\neq \emptyset $ and $ \mathcal{V}_B\cap B(y,1/4)\neq \emptyset $. For $ k>0,x\in \RR^d $ we say that $ x $ is a $ k$-\textit{seed} if $ B(x,k) $ is covered. For a measurable $ R\subset \RR^d $ and $ \delta>0 $ we are now adapting the intensity of $ \mathcal{V}_A $ (resp. $ \mathcal{V}_B $). Inside the area $ R $, the PPP $ \mathcal{V}_A $ (resp. $ \mathcal{V}_B $) has intensity $ \delta $ and for $ R^c $ we use the original intensity $ \lambda $. We write $ \Pro_{\lambda,\delta,R} $ for the corresponding measure. We also define the event
\begin{equation*}
	E_{\ell,k,R}\coloneqq\bigg\{R\overset{B(\typ,3\ell)}{\longleftrightarrow} x \text{ for some } x\in \mathcal{V}\cap B( \ell e_1,\ell/100) \text{ and } x \text{ is a $ 2k $-seed} \bigg\},
\end{equation*}
where $ e_1=(1,0,\dots,0) $ stands for the first unit vector. We will say that $ R $ is \textit{admissible} if for every $ x\in R $ there is a $ y\in R $ with $ \norm{x-y}<2 $ and $ B(y,1/2)\subseteq R $. We note that if we define for some $ y\in \partial B(\typ,\ell) $
\begin{equation}\label{abconnectionevent}
	E_{k,R}(\typ,y)\coloneqq\bigg\{R\overset{B(\typ,3\ell)}{\longleftrightarrow} x \text{ for some } x\in \mathcal{V}\cap B( y,\ell/100) \text{ and } x \text{ is a $ 2k $-seed} \bigg\},
\end{equation}
 by symmetry the result below also holds for this event.
\begin{figure}[t]
  \centering
  \includegraphics[width=0.5\linewidth]{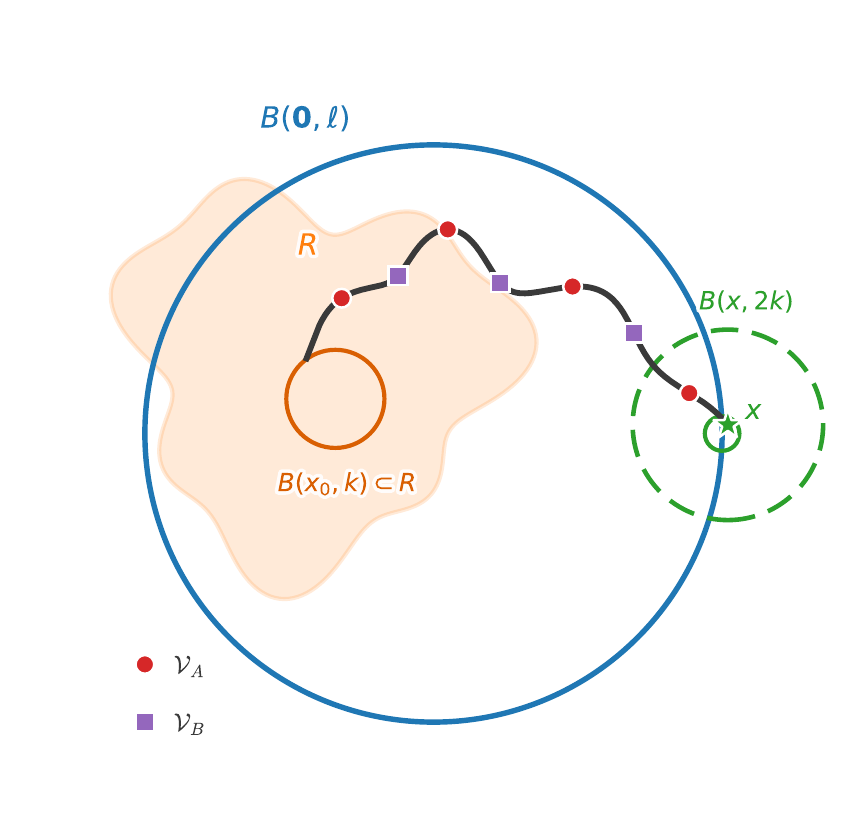}
  \caption{Schematic of the event $E_{\ell,k,R}$.}
\end{figure}

\begin{lemma}\label{ABlongconnection}
	Let $ \lambda>\lambda_c $ and $ \eps,\delta>0 $ be arbitrary with $ \lambda\geq \delta $. Then, there exists a $ k=k(\lambda,\eps,\delta) $ such that for all large enough
	$ \ell\geq\ell_0(\lambda,\eps,\delta,k) $ it holds that
	\begin{equation*}
		\Pro_{\lambda,\delta,R}(E_{\ell,k,R})>1-\eps,
	\end{equation*}
   for all admissible $ R $ with $B(x,k) \subseteq R $ for some $ x\in B(\typ,\ell) $.
\end{lemma}
The proof follows and adapts the ideas of Grimmett and Marstrand \cite{GrimmettMarstrand1990Supercritical} and we will split it into four lemmas. We define the annulus $ A(\ell)\coloneqq  B(\typ,\ell)\setminus B(\typ,\ell-1)$ and
\begin{equation*}
	X_{\ell,k}\coloneqq\{x\in \mathcal{V}\cap A(\ell)\,:\, x\stackrel{B(\typ,\ell)}{\longleftrightarrow} B(\typ,k)\}.
\end{equation*}
We let $Y_{\ell,k}$ be the minimum size of a set of centers $\{y_1,\dots,y_m\}\subset X_{\ell,k}$
such that $X_{\ell,k}\subset \bigcup_{i=1}^m B(y_i,100k)$ and $\|y_i-y_j\|>20k$ for all $i\ne j$.
\begin{lemma}\label{l1}
	For all $ \lambda>\lambda_c $ and $ \eps>0 $ there is a $ k=k(\lambda,\eps) $ such that the following holds. For all $ N\in \NN $ we have for all $\ell\geq \ell_0(\lambda,\eps,N,k) $
	\begin{equation*}
		\Pro(Y_{\ell,k}\geq N)>1-\eps.
	\end{equation*} 
\end{lemma}
\begin{proof}
	We let $ \eps>0 $ and since we are supercritical we can and do choose $ k=k(\lambda,\eps) $ large enough such that
	\begin{equation*}
		\Pro(B(\typ,k)\longleftrightarrow \infty)>1-\eps/2. 
	\end{equation*}
	Then, since edges have a length of less than $ 1 $ we also have for all $ \ell> k$
	\begin{equation*}
		\Pro(B(\typ,k)\longleftrightarrow A(\ell))>1-\eps/2.
	\end{equation*}
 We now reveal $ \mathcal{V} $ in two stages. First, all points inside $ B(\typ,\ell) $ and then the remaining points of $ B(\typ,\ell)^c $. We write $ S_1 $ for the event that $ B(\typ,k)\longleftrightarrow A(\ell) $ occurs after the first stage. Note that if for every center $ y $ of a ball counted by $  Y_{\ell,k}$ it holds that $ \mathcal{V}\cap B(y,100k+1)\setminus B(\typ,\ell)=\emptyset $, the event $ B(\typ,k)\longleftrightarrow \infty $ cannot hold. If after stage one it holds that $ Y_{\ell,k}\leq N $ for some $ N\in\NN $, the probability of this occurring in the second stage is lower bounded by $ \exp(-2N\lambda\vol(B(\typ,100k+1))) $, where the factor of $ 2 $ arises since we have two independent Poisson point processes with intensity $ \lambda $. Thus for all $ N $
 \begin{equation*}
 	\Pro(B(\typ,k)\nocon\infty\,\vert\,S_1)\geq \exp(-2N\lambda\vol(B(\typ,100k+1))) \Pro(Y_{\ell,k}\leq N\,\vert S_1).
 \end{equation*}
Since we have that 
\begin{equation*}
	\Pro(B(\typ,k)\nocon\infty\,\vert\,S_1)=1-\frac{\Pro(B(\typ,k)\longleftrightarrow \infty)}{\Pro(B(\typ,k)\longleftrightarrow A(\ell))},
\end{equation*}
and because the right hand side converges to $ 0 $ as $  \ell\to\infty $, it follows that $ \Pro(Y_{\ell,k}\leq N\,\vert S_1)\to 0 $ as  $  \ell\to\infty $ too. We can conclude that for large enough $ \ell $ 
\begin{equation*}
	\begin{split}
		\Pro(Y_{\ell,k}\leq N)&\leq \Pro(S_1^c)+\Pro(Y_{\ell,k}\leq N\,\vert S_1)\Pro(S_1)
		\leq \Pro(S_1^c)+\Pro(Y_{\ell,k}\leq N\,\vert S_1)<\eps.
	\end{split}
\end{equation*}
\end{proof}

\noindent
We now define the set 
\begin{equation*}
	A_{\ell,k,R}\coloneqq\{x\in \mathcal{V}\cap A(\ell)\,:\, x\stackrel{B(\typ,\ell)}{\longleftrightarrow} R\, \text{ and }x \text{ is a }2k\text{-seed} \}.
\end{equation*}
\begin{lemma}\label{l2}
	For all $ \lambda>\lambda_c $ and $ \eps,\delta>0 $ there exist $ k=k(\lambda,\eps,\delta) $ such that for all large enough $ \ell\geq \ell_0(\lambda,\eps,\delta,k) $
	\begin{equation*}
		\Pro_{\lambda,\delta,B(\typ,\ell-2k-2)^c}(|A_{\ell,k,B(\typ,k)}|> 0)>1-\eps.
	\end{equation*}
\end{lemma}
\begin{proof}
	We let $ \lambda_1=(\lambda+\lambda_c)/2>\lambda_c $ and $ N $ large enough to be determined in the course of the proof. \cref{l1} yields a $ k=k(\lambda_1,\eps/2) $ and $ \ell\geq \ell_0(\lambda_1,\eps/2,N,k) $ such that for continuum AB percolation with intensity $ \lambda_1 $ it holds
	\begin{equation*}
		\Pro_{\lambda_1}(Y_{\ell-2k-2,k}\geq N)>1-\eps/2.
	\end{equation*}
We can split $ \mathcal{V}(\lambda)\cap B(\typ,\ell-2k-2) $ into the union of $ \mathcal{V}({\lambda_1})  $ and $ \mathcal{V}({\lambda-\lambda_1}) $, where we take the union of the $A$ and $B$ label sets separately. We now reveal the points in two steps. First, we will reveal $ \mathcal{V}({\lambda_1})  $ and then in the second stage all points of $ \mathcal{V}({\lambda-\lambda_1}) $ and $ \mathcal{V}(\delta) $. Note that if after stage one $ Y_{\ell-2k-2,k}\geq N $ holds, there exist at least $ N $ disjoint balls of radius $ 10k $ that are centered at points of $ X_{\ell-2k-2,k}$. For each such point $ x $ there is a probability that is lower bounded by some $ q\coloneqq q(\lambda-\lambda_1,\delta,k) >0$ that in the second revealment $ B(x,9k)\cap B(\typ,\ell-2k-2) $ is covered by $ \mathcal{V}(\lambda-\lambda_1) $ and $ B(x,
9k)\cap B(\typ,\ell-2k-2)^c $ is covered by $ \mathcal{V}(\delta) $. Note, for each $ x $, this would yield a $ 2k $-seed w.r.t. the measure $ \Pro_{\lambda,\delta,B(\typ,\ell-2k-2)^c} $ centered at a point inside $ A(\ell)$. Using the fact that for each $ x $, whether $B(x,9k)$ is covered depends only on points inside $B(x,10k)$, we obtain by writing $ S_1 $ if $ Y_{\ell-2k-2,k}\geq N $ occurs after stage one
\begin{equation*}
	\Pro_{\lambda,\delta,B(\typ,\ell-2k-2)^c}(|A_{\ell,k,B(\typ,k)}|> 0\,\vert\,S_1)\geq \Pro(\text{Bi}(N,q)> 0 )=1-(1-q)^N.
\end{equation*}
Since $ q $ does not depend on $ \ell $, we can now choose $ N $ and $ \ell $ large enough such that the right hand side is greater than $ 1-\eps/2 $. As 
\begin{equation*}
	\begin{split}
		\Pro_{\lambda,\delta,B(\typ,\ell-2k-2)^c}(|A_{\ell,k,B(\typ,k)}|> 0)&\geq \Pro_{\lambda,\delta,B(\typ,\ell-2k-2)^c}(|A_{\ell,k,B(\typ,k)}|> 0\,\vert\,S_1)\Pro(S_1)\\&\geq (1-\eps/2)^2\geq 1-\eps,
	\end{split}
\end{equation*}
the proof is finished.
\end{proof}

\noindent
We define an event where there is a point of $ A_{\ell,k,R} $ inside a fixed ball as
\begin{equation*}
	J_{\ell,k,R}\coloneqq\{|A_{\ell,k,R}\cap B(\ell e_1,\ell/1000)|>0\}.
\end{equation*}
\begin{lemma}\label{l3}
	For all $ \lambda>\lambda_c $ and $ \eps,\delta>0 $ there exist $ k=k(\lambda,\eps,\delta) $ such that for all large enough $ \ell\geq \ell_0(\lambda,\eps,\delta,k) $
	\begin{equation*}
		\Pro_{\lambda,\delta,B(\typ,\ell-2k-2)^c}(J_{\ell,k,B(\typ,k)})>1-\eps.
	\end{equation*}
\end{lemma}
\begin{proof}
	Note that for $ \partial B(\typ,\ell) $ and some dimension dependent constant $ c_d $, we can place $ c_d $ many balls of radius $ \ell/1000 $ with centers at the boundary such that they cover $ A(\ell) $. We let $ B(x_i,\ell/1000) $, $ 1\leq i\leq c_d $ be an enumeration of it and can assume without loss of generality that $ x_1=\ell e_1 $. For $ \eps,\delta>0 $ we use \cref{l2} to obtain a $ k $ such that for all large enough $ \ell $
	\begin{equation*}
		\Pro_{\lambda,\delta,B(\typ,\ell-2k-2)^c}(|A_{\ell,k,B(\typ,k)}|> 0)>1-\eps^{c_d}.
	\end{equation*}
	If we now write $ A_i\coloneqq \{|A_{\ell,k,B(\typ,k)}\cap B(x_i,\ell/1000)|>0\} $ it holds that $ \cup A_i= \{|A_{\ell,k,B(\typ,k)}|> 0\} $ and by symmetry all $ A_i $ are identically distributed. Note that each event $A_i$ is {increasing} w.r.t.\ adding points to either
$\mathcal V_A$ or $\mathcal V_B$. Therefore the decreasing events $A_i^c$ are positively dependent, and by
induction
\begin{equation*}
    \Pro\Big(\bigcap_{i=1}^{c_d} A_i^c\Big)\;\ge\;\prod_{i=1}^{c_d}\Pro(A_i^c).
\end{equation*}
Since the $A_i$ are identically distributed, this gives
\begin{equation*}
    \Pro\Big(\bigcup_{i=1}^{c_d} A_i\Big)
=1-\Pro\Big(\bigcap_{i=1}^{c_d} A_i^c\Big)
\le 1-\big(1-\Pro(A_1)\big)^{c_d},
\end{equation*}
and hence

\begin{equation*}
\begin{array}{rcl}
\Pro_{\lambda,\delta,B(\typ,\ell-2k-2)^c}(J_{\ell,k,B(\typ,k)}) 
& = & \Pro(A_1) \\
& \geq & 1-(1-\Pro_{\lambda,\delta,B(\typ,\ell-2k-2)^c}(|A_{\ell,k,B(\typ,k)}|> 0))^{1/c_d} \\
& >& 1-\eps.
\end{array}
	\end{equation*}
	
\end{proof}

\begin{lemma}\label{l4}
	For all $ \lambda>\lambda_c $ and $ \eps,\delta>0 $ with $ \lambda\geq \delta $ there exist $ k=k(\lambda,\eps,\delta) $ such that for all large enough $ \ell\geq \ell_0(\lambda,\eps,\delta,k) $
	\begin{equation*}
		\Pro_{\lambda,\delta,R}(J_{\ell,k,R})>1-\eps,
	\end{equation*}
for all admissible $ R $ with $ B(\typ,k) \subseteq R$.
\end{lemma}
\begin{proof}
	We let $ \lambda_1=(\lambda +\lambda_c)/2 >\lambda_c$ and $ \eps_1>0 $ small to be chosen later. We can choose $ k,\ell_0 $ as given by \cref{l3} such that for all $ \ell\geq \ell_0 $
	\begin{equation*}
		\Pro_{\lambda_1,\delta/2,B(\typ,\ell-2k-2)^c}(J_{\ell,k,B(\typ,k)})>1-\eps_1.
	\end{equation*}
Further, we define the random set
\begin{equation*}
	S\coloneqq \bigcup \{B(x,k)\,:\, x\in\mathcal{V}\cap A(\ell)\cap B(\ell e_1,\ell/1000)\text{ is a }2k\text{-seed}\}.
\end{equation*}
Additionally we let $\widehat{R}\coloneqq R\cap B(\typ,\ell-2k-2)$ and define 
\begin{equation*}
\mathcal{U}\coloneqq\bigg\{y\in \mathcal{V}\cap \widehat{R}^c:\inf_{w\in \widehat{R}}\norm{y-w}<1 \text{ and } y\stackrel{\widehat{R}^c\cap B(\typ,\ell)}{\longleftrightarrow} S \bigg\}
\end{equation*}
and write $ Y$ for the minimum number of balls with radius $ 100 $ that are centered at points of $ \mathcal{U} $ in a way that their union covers all points of $ \mathcal{U} $ and such that the centers are at distance of more than $ 10 $ from another. We again reveal the Poisson points in question in two steps. First, we reveal Poisson points in  $\mathcal{V}(\lambda_1) \cap B(\typ,\ell-2k-2)\cap R^c $ and $ \mathcal{V}(\delta/2)\cap B(\typ,\ell-2k-2)^c $. In the second step we reveal the remaining points, so $\mathcal{V}(\lambda_1)\cap \widehat{R} $. Assume that after the first revealment it holds that $ Y\leq N $ for some $ N $. We let $ \ell $ be large enough such that $ \ell-2k-2\geq k $. Then, if in the second step there are no new points of $ \mathcal{V}(\lambda_{1}) $ inside $ B(y,101)\cap B(\typ,\ell-2k-2)\cap R  $ for all $ y$ counted by $ Y $, the event $ J_{\ell,k,B(\typ,k)} $ cannot hold, as $ B(\typ,k)\subseteq R\cap B(\typ,\ell-2k-2) $. The probability of this occurring in the second step is lower bounded by $ q_1\coloneqq \exp(-2N\lambda_1\vol(B(\typ,101))) $. Thus,
\begin{equation*}
	\eps_1>\Pro_{\lambda_1,\delta/2,B(\typ,\ell-2k-2)^c}(J_{\ell,k,B(\typ,k)}^c)\geq q_1\,\cdot\, \Pro_{\lambda_1,\delta/2,B(\typ,\ell-2k-2)^c}(Y\leq N).
\end{equation*}
 As $ Y $ does not depend on any particles of $ R\cap B(\typ,\ell-2k-2) $, its distribution is identical under the measure $ \Pro_{\lambda_1,\delta/2,B(\typ,\ell-2k-2)^c}  $ and $ \Pro_{\lambda_1,\delta/2,B(\typ,\ell-2k-2)^c\cup R} $. This leads us to
\begin{equation*}
	\Pro_{\lambda_1,\delta/2,B(\typ,\ell-2k-2)^c\cup R}(Y\leq N)\leq \frac{\eps_1}{q_1}.
\end{equation*}
We now do two revealments with respect to $ \Pro_{\lambda,\delta/2,B(\typ,\ell-2k-2)^c\cup R} $, so the area where we reveal $ \mathcal{V}(\delta) $ is $ B(\typ,\ell-2k-2)^c\cup R $. Also, we can split $ \mathcal{V}(\delta)=\mathcal{V}_1(\delta/2)\cup\mathcal{V}_2(\delta/2) $ into two i.i.d. models if we take the union with respect to all $A$ and $B$ labels separately. The same holds for $ \mathcal{V}(\lambda)=\mathcal{V}(\lambda_1)\cup\mathcal{V}(\lambda-\lambda_1)$.
First we reveal all points of $\mathcal{V}(\lambda_1)$ and $\mathcal{V}_1(\delta/2)$ and in the second revealment we sprinkle with $ \mathcal{V}(\lambda-\lambda_{1})$ and $\mathcal{V}_2(\delta/2)$. If in the second revealment we find at least one point of $ \mathcal{V}(\lambda-\lambda_{1}) $ inside $ R\cap B(\typ,\ell-2k-2) $ that is connected to some particle of $ \mathcal{U} $, the event $  J_{\ell,k,R} $ holds. For each center of a ball $ y $ counted by $ Y $ there is a chance that $ B(y,5)\cap (B(\typ,\ell-2k-2)\cap R^c) $ is covered with respect to $ \mathcal{V}(\lambda-\lambda_1) $ and $ B(y,5)\cap (B(\typ,\ell-2k-2)^c\cup R) $ is covered with respect to $ \mathcal{V}_2(\delta/2) $. Additionally, as $ R $ is admissible there exists a $ z\in R $ such that $ B(z,1/2)\subseteq B(y,5)\cap R $. So, if in the second revealment there is a point of $ \mathcal{V}(\lambda-\lambda_1)\cup\mathcal{V}_2(\delta/2) $ inside $ B(z,1/2) $ that is connected to $ y $, the event $ J_{\ell,k,R} $ holds. This would be the case if the just described covering is fulfilled. The probability of all this occurring for a fixed $ y\in \mathcal{U} $ in the second revealment is lower bounded by some $ q_2=q_2(\delta,\lambda)>0 $. As distinct $ y $ are at distance of more than $ 10 $ to each other we obtain
\begin{equation*}
	\begin{split}
		\Pro_{\lambda,\delta/2,B(\typ,\ell-2k-2)^c\cup R}(J_{\ell,k,R})&\geq \bigg(1-(1-q_2)^N\bigg)\cdot \Pro_{\lambda_1,\delta/2,B(\typ,\ell-2k-2)^c\cup R}(Y> N)\\&\geq \bigg(1-(1-q_2)^N\bigg)\bigg(1-\frac{\eps_1}{q_1}\bigg).
	\end{split}
\end{equation*}
If we now choose $ \eps_1=\eps_1(\delta,\lambda,\eps) $ and $ N=N(\delta,\lambda,\eps) $ appropriately, we can conclude that
\begin{equation*}
	\Pro_{\lambda,\delta,R}(J_{\ell,k,R})\geq \Pro_{\lambda,\delta/2,B(\typ,\ell-2k-2)^c\cup R}(J_{\ell,k,R}) >1-\eps,
\end{equation*}
using that $ \lambda\geq \delta $ and that $ J_{\ell,k,R} $ remains to hold when adding points.
\end{proof}

\begin{proofof}{\cref{ABlongconnection}}
	We let $ k(\lambda,\eps,\delta) $ and $\widehat{\ell_0}(\lambda,\eps,\delta,k) $ be given by \cref{l4}. Note that by symmetry and translation we have for all $ \widehat{\ell}\geq \widehat{\ell_0} $ and $ y\in \partial B(\typ,\widehat{\ell})$
	\begin{equation*}
		\Pro_{\lambda,\delta,R}(|A_{\widehat{\ell},k,R}\cap B(y,\widehat{\ell}/1000)|>0)>1-\eps,
	\end{equation*}
	for all admissible $ R $ with $ B(\typ,k)\subset R $.
	We now let $ {\ell}\geq 1000\widehat{\ell_0} $ and assume $ B(x,k)\subset R $ for some $ x\in B(\typ,\ell) $. We see that for $ y=1000\ell e_1/999  $ it holds that $ \norm{x-y}\in [\ell/999,3\ell] $, so the distance is more than $ \widehat{\ell_0} $ in any case. Additionally we have that $ B(y,\norm{x-y}/1000)\subseteq B(\ell e_1,\ell/100) $. We obtain by translation invariance 
	\begin{equation*}
		\begin{split}
			\Pro_{\lambda,\delta,R}(E_{\ell,k,R})&\geq \Pro_{\lambda,\delta,R}\bigg(R\overset{B(\typ,3\ell)}{\longleftrightarrow} z 
			\text{ for some $2k$-seed $z\in \mathcal{V}\cap B( y,\norm{x-y}/1000)$}\bigg)
			\\&>1-\eps.
		\end{split}
	\end{equation*}
\end{proofof}
\subsection{Existence of an odd cycle}
From \cref{lem:metric} we get to choose $ t $ sufficiently close to $ 1 $ such that for all $ x_1,x_2\in C_t $
\begin{equation}\label{tcloseto1}
	\norm{\pi(x_1)-\pi(x_2)}\leq \dists(x_1,x_2).
\end{equation}
We fix some $ t>t_0 $ that satisfies this, where $t_0=t_0(\eps,d)$ is from \cref{coupleab}. We are going to write $ \mathcal{V}_{\lambda,y} $ for the corresponding continuum AB percolation model given by \cref{coupleab} and will fix the $\lambda$ it yields in the following. This time the projected points lie in the $ P_y= \{x\in \RR^{d+1} \,:\, \langle x,y\rangle=0\} $ plane. We will work with different projections and to make it clearer in what plane certain $ d $-dimensional balls lie we will write $ B_y(x,r) $ for the ball with radius $ r $ inside $ P_y $ around the point $ x\in P_y $. We also let $ B(x,r)\coloneqq B_{(0,\dots,0,-1)}(x,r) $.\\

\subsubsection{Exploration algorithm}
We now define a graph $H=H(\eta)$ on the sphere that depends on $ \eta>0 $, which we will choose later. The goal is to define an exploration algorithm that successively declares edges of $H$ open or closed. We let 
\begin{equation*}
    m=\bigg\lfloor{\frac{\sqrt{3}}{\sqrt{d}\eta}}\bigg\rfloor,
\end{equation*}
and define $H=\pi^{-1}(\eta\mathbb{Z}^d\cap D)$, where $D\coloneqq [-\eta m, \eta m]^d$ and adjacencies are inherited from the nearest neighbor model. Note that this means it holds $\sqrt{d} m \eta\leq \sqrt{3}$ and therefore for each $x\in H$
\begin{equation}\label{latticecap}
    \sqrt{3}\geq \norm{\pi(x)}=\frac{\sqrt{1-x^2_{d+1}}}{1-x_{d+1}},
\end{equation}
yielding $x_{d+1}\leq 1/2$. Trivially, the south pole is an element of $H$. Additionally, we let $ A=A(d)\in \NN $ be a constant integer that we will determine later, choose $ \lambda_{c}<\lambda_1<\dots<\lambda_{A}<\lambda  $ and define
\begin{equation*}
	\delta\coloneqq\min\bigg\{\lambda_1,\lambda_2-\lambda_1,\dots,\lambda_{A}-\lambda_{A-1}\bigg\}.
\end{equation*}
For all $ x\in H $, the points $ \mathcal{V}_{\lambda,x} $ that consist of the two Poisson point processes that correspond to the two labels of the $ AB $ model with intensity $ \lambda$ can be written as the union
\begin{equation*}
	\mathcal{V}_{\lambda,x}=\mathcal{V}_{\lambda_1,x}\cup\bigg(\bigcup_{i=2}^A \mathcal{V}_{\lambda_{i}-\lambda_{i-1},x}\bigg)\cup \mathcal{V}_{\lambda-\lambda_{A},x},
\end{equation*}
where we take the union of all $ A $ labels and $ B $ labels separately. Note that by \cref{coupleab} each point in these continuum AB percolation models will be a point of the Borsuk graph by mapping them back with $ g^{-1}_x(\cdot) $.
For $ \eps_1>0 $ to be chosen later we let $ k(\lambda_1,\eps_1,\delta),\ell_0(\lambda_1,\eps_1,\delta) $ be the values from \cref{ABlongconnection}. For $ x_1=(0,\dots,0,-1)\in H $ and an arbitrary neighbor $ y_1\in H $ of it we start exploring to determine whether the edge $ x_1 y_1 $ is open or not. We remind ourselves that $ g_{x_1}(x_1)=g(x_1)=\typ $. 
We declare $ x_1 y_1 $ open  if $ E_{k,B(\typ,k)}(\typ,g(y_1))  $ from \cref{abconnectionevent} holds with respect to  $ \mathcal{V}_{\lambda_1,x_1} $. Note that we can determine whether the event occurs by not revealing all points instantaneously. This can be done in the following way: First, we will reveal all $ \mathcal{V}_{\lambda_1,x_1}\cap B(\typ,k)  $. We now write $ \mathcal{V}_{\lambda_1,x_1}(A) $ (resp. $ \mathcal{V}_{\lambda_1,x_1}(B) $) for all points with label $ A $ (resp. label $ B $). Then, for each $  a\in \mathcal{V}_{\lambda_1,x_1}(A)\cap B(\typ,k) $ we can reveal all $ \mathcal{V}_{\lambda_1,x_1}(B)\cap B(a,1) $ that are not yet revealed (so inside $ B(\typ,k)^c $). Those points will then all be connected to $ a $. In the same manner we will reveal for all $  b\in \mathcal{V}_{\lambda_1,x_1}(B)\cap B(\typ,k) $ all $ \mathcal{V}_{\lambda_1,x_1}(A)\cap B(b,1) $ that are not yet found. Iterating this procedure will reveal only $  \mathcal{V}_{\lambda_1,x_1} $ that lie inside $ B(\typ,k) $ and points that are connected to $ B(\typ,k) $. Going back to the graph $ G_t $ on the sphere this means that we only reveal part of the points inside the ball 
  \begin{equation*}
  	B_k\coloneqq\{y\in S^d:\dists(x_1,y)\leq \dists(x_1,g^{-1}(ke_1)) \}
  \end{equation*}
   and all other points found are connected to that set. We will never reveal area outside $B(\typ,3\norm{g(y_1)})$ and stop exploring once either $ E_{k,B(\typ,k)}(\typ,g(y_1))   $ is satisfied or all points we found inside $B(\typ,3\norm{g(y_1)})$ have been used for an exploration. The next steps will be defined iteratively with the same strategy in mind as long as there exists an edge $ x_{i}y_i $ in $ H $ such that $ x_i $ is connected to $ x_1 $ via a path of open edges and the status of $ x_{i}y_i $ is not determined yet. In addition, our exploration will yield a set $ R_i\subset S^d $ where we have already revealed part of the points of $ X $ and $ Y $. We can split it into disjoint $ A_{1},\dots,A_i $ such that for $ 1\leq j\leq i $ we have revealed points in  $ A_j $ for $ j $ different iterations. Similarly to the initial edge, we open $ x_iy_i $ if $ E_{k,g_{x_i}(R_i)}(\typ,g_{x_i}(y_i)) $ occurs when we reveal $ \mathcal{V}_{\lambda_1,x_i} $ in the `fresh area' $ g_{x_i}(R^c_i) $ and $ \mathcal{V}_{\lambda_{j+1}-\lambda_{j},x_i} $ inside $ g_{x_i}(A_j) $ for each $ 1\leq j\leq i $. Note that by doing this, we have not encountered these points in any previous exploration. Using the same procedure as done for the first edge, we again reveal the points successively to make sure every found point is connected to $B_{x_i}(\typ,k)$.

\subsubsection{Verifying requirements}\label{sec:requirements}

In order to use \cref{ABlongconnection} to make sure edges are open with a probability close to 1, there are multiple requirements we have to check in order to make sure it is actually applicable to our setting. Firstly, for our steps in the exploration it is important that neighbors of $ H $ are sufficiently far away after projecting, as the Lemma required a distance of at least $\ell_0$. 
\begin{lemma}\label{mindistproj}
	For
	\begin{equation*}
		\eta=  6\ell_0\alpha(n)
	\end{equation*}
 we have for all large enough $ n=n(d)  $ that it holds for all neighboring $ x,y\in H $
	\begin{equation*}
		 \ell_0 \leq \norm{g_x(y)}\leq 24\,\ell_0 .
	\end{equation*}
\end{lemma}
\begin{proof}
   For neighbors $ x,y\in H $ we have $ x_{d+1},y_{d+1}\leq 1/2 $ and so by \cref{lem:metric} (\rom{1},\rom{3}) it holds that
   \begin{equation}\label{proofeq1}
   	 \eta=\norm{\pi(x)-\pi(y)}\leq 6\dists(x,y)\leq 12\norm{\pi(x)-\pi(y)}=12\eta. 
   \end{equation}
	By \cref{lem:metric} (\rom{1}) and using the symmetry of different stereographic projections we get
	\begin{equation*}
		\norm{g_x(y)}=\frac{2}{\alpha(n)}\norm{\pi_x(y)}\geq \frac{1}{\alpha(n)}\ \dists(x,y),  
	\end{equation*}
and the lower bound follows. Using that $ \norm{\pi(x)-\pi(y)}=\eta $ we can apply \cref{lem:metric} (\rom{2}) and \cref{tcloseto1} to obtain for all large enough $ n=n(d) $ 
\begin{equation*}
	\begin{split}
		\norm{g_x(y)}&=\frac{2}{\alpha(n)}\norm{\pi_x(y)}\leq\frac{2}{\alpha(n)} \dists(x,y)\leq  \frac{4}{\alpha(n)} \norm{\pi(x)-\pi(y)}=24\ell_0 .
	\end{split}
\end{equation*}
\end{proof}

\noindent
 In the following, we always let $ \eta $ be as given by the lemma above. In order to determine whether we have a long connection in the Borsuk graph or not, we will project caps $ C_t(x) $. It is important that the relative size of the cap is enough to be able to use \cref{ABlongconnection}, as shown in the following. This means, we will not have to reveal points that are outside  $ G_t(x_i) $ in order to determine whether $ E_{k,g_{x_i}(R_i)}(\typ,g_{x_i}(y_i))$ is satisfied 
 wrt.~the points we reveal.
 \begin{lemma}\label{largeenoughcap}
 	We can choose $ n=n(t,\ell_0,d) $ large enough such that for all adjacent $ y,x\in H $ and all $ z\in \partial C_t(x) $ 
 	\begin{equation*}
 		 \norm{g_{x}(z)}\geq 3 \norm{g_{x}(y)} .
 	\end{equation*}
 \end{lemma}
\begin{proof}
	By \cref{mindistproj} it holds that $ 24\ell_0 \geq \norm{g_{x}(y)} $, provided $ n=n(d) $ is large enough. Further, we have by symmetry of the different stereographic projections and \cref{ppp} for all $ z_2\in \partial C_t $  
	\begin{equation*}
	\norm{g_{x}(z)}=\norm{g(z_2)}=\rho(t,n)=\frac{2\sqrt{1-t^2}}{(1+t)\alpha(n)}.
	\end{equation*}
	The claim now follows by choosing $ n=n(t,d,\ell_0) $ sufficiently large.
\end{proof}

\noindent
  We will choose $ n $ large enough such that the \cref{mindistproj,largeenoughcap} above are satisfied. 
  \begin{lemma}\label{maxspheredist}
		For all edges $ h_1h_2 $ and $ h_3h_4 $ of $ H $ it holds that
		\begin{equation*}
			\dists(h_1,h_2)\leq12\dists(h_3,h_4).
		\end{equation*}
	\end{lemma}
\begin{proof}
	By \cref{lem:metric} we have that
	\begin{equation*}
		\begin{split}
			\eta/6=\norm{\pi(h_1)-\pi(h_2)}/6&\leq \dists(h_1,h_2)\leq 2\norm{\pi(h_1)-\pi(h_2)}=2\eta\leq  12\dists(h_3,h_4).
		\end{split}
	\end{equation*}
\end{proof}

\noindent
Each successful exploration step that we will define after this Lemma produces a seed that gets used for the next iterations. The following makes sure that the seeds are actually large enough after the area gets distorted by the stereographic projections. The lemma implies in the beginning of each exploration step there is a $ B_{x_i}(x,k)\subseteq B_{x_i}(\typ,\norm{g_{x_i}(y_i)}) $ such that $ g^{-1}_{x_i}(B_{x_i}(x,k)) $ is connected to $ B_k $ in the Borsuk graph while staying inside $ R_i $. Additionally, $ B_{x_i}(x,k) $ is covered with respect to $ \mathcal{V}_{\lambda,x_i} $. 
		\begin{lemma}\label{restart}
		Let $ h_1h_2 $ and $ h_2h_3 $ be edges in $ H $. Assume that $ x\in P_{h_1} $ and $ k\in\NN $ are such that $B_{h_1}(x,2k)\subseteq B_{h_1}(g_{h_1}(h_2),10^{-2}\norm{g_{h_1}(h_2)}) $. Then we can choose $ n=n(\ell_0,k,d) $ sufficiently large such that there is a $ y\in P_{h_2} $ with
		\begin{equation*}
			B_{h_2}(y,k)\subseteq g_{h_2}(g^{-1}_{h_1}(B_{h_1}(x,2k))) \subseteq B_{h_2}(\typ,\norm{g_{h_2}(h_3)}).
		\end{equation*} 
	\end{lemma}
	\begin{proof}
		As $ t $ is fixed, we can and will choose  $  n=n(\ell_0,k,d)  $ large enough such that all areas on the sphere that we consider in this proof lie in $ C_{t}(h_1) $. We let $ y= g_{h_2}(g^{-1}_{h_1}(x)) $. Then, by \cref{tcloseto1} and \cref{lem:metric} it holds for all $ z_1\in B_{h_2}(y,k) $ that
		\begin{equation*}
			\begin{split}
				k\geq \norm{ g_{h_2}(g^{-1}_{h_1}(x))-z_1} &\geq \frac{1}{\alpha(n)}\dists(g^{-1}_{h_1}(x),g_{h_2}^{-1}(z_1))\geq \frac{1}{2}\norm{x-g_{h_1}(g^{-1}_{h_2}(z_1))},
			\end{split}
		\end{equation*}
		and $ z_1\in g_{h_2}(g^{-1}_{h_1}(B_{h_1}(x,2k))) $ follows. For $ z_2\in B_{h_1}(x,2k)\subseteq B_{h_1}(g_{h_1}(h_2),10^{-2}\norm{g_{h_1}(h_2)}) $ we can bound
		\begin{equation*}
			\begin{split}
				10^{-2}{\norm{g_{h_1}(h_2)}}&\geq \norm{z_2-g_{h_1}(h_2)}\geq \frac{1}{\alpha(n)}\dists(g^{-1}_{h_1}(z_2),h_2)\geq \frac{1}{2}\norm{g_{h_2}(g^{-1}_{h_1}(z_2) )},
			\end{split}
		\end{equation*}
		where we utilized \cref{lem:metric} again. We are now done, as it further holds by \cref{maxspheredist} and \cref{tcloseto1} that
		\begin{equation*}
			\norm{g_{h_1}(h_2)}\leq \frac{2}{\alpha(n)}\dists(h_1,h_2)\leq \frac{24}{\alpha(n)} \dists(h_2,h_3)\leq 24 \norm{g_{h_2}(h_3)}.
		\end{equation*} 
	\end{proof}	  

    \noindent
	It is important that for the area that is relevant for the event, we did not encounter it in too many previous iterations. This is because if we define  
	\begin{equation*}
		a_{i}\coloneqq | \{1\leq j\leq i: g_{x_i}(A_j)\cap B_{x_i}(\typ,3\norm{g_{x_i}(y_i)})\neq \emptyset \}|,
	\end{equation*}
	to be the number of revealments, $ a_i $ is uniformly bounded from above, as shown in the following two lemmas. This is needed, as we have chosen $ A(d) $ different intensities $ \lambda_i $ in the beginning. For $ x\in H $ we will write $ x^o $ for the point of $ H $ that is closest to $ -x $. If this is not a unique point we pick a point out of the possible choice by some arbitrary deterministic rule.
	\begin{lemma}\label{nointersect}
	 Let $ r_1>0 $ and for all $ x\in H $ define an event $ E_x $ that depends only on the points of $ \mathcal{V}_{\lambda,x} $ inside $ B_x(\typ, r_1\ell_0 ) $. Then, there is a constant $M=M(r_1,d)>0 $ such that each $ E_x $ depends only on $ E_y $ for sites $ y\in H $ that have a 
	 graph distance of less than $M$ to $ x $ or $ x^o $.
	\end{lemma}
\begin{proof}
By \cref{lem:metric,mindistproj} and \cref{tcloseto1} it holds for all $ z\in S^d $ with $ g_{x}(z)\in B_{x}(\typ,r_1\ell_0)  $ and for all edges $ uv $ of $ H $
\begin{equation*}
	\begin{split}
		\frac{1}{\alpha(n)}\dists(x,z)&\leq r_1\ell_0\leq r_1\norm{g_{u}(v)}\leq\frac{2r_1}{\alpha(n)}\dists(u,v).
	\end{split}
\end{equation*}
Writing $ \Delta\coloneqq \max\{\dists(u,v)\,:\, uv\text{ is an edge in } H\} $ we can thus say that $ \dists(x,z)\leq 2r_1\Delta $. Going back to the sphere, it follows that the region on which $  E_x $ depends on is a subset of
\begin{equation*}
	\begin{split}
		\{y\in S^d \, :\,\dists(x,y)\leq 2r_1\Delta \}\cup \{y\in S^d \, :\,\dists(-x,y)\leq 2r_1\Delta \}.
	\end{split}
\end{equation*}
As a square of side length $\ell$ in Euclidean space has a diameter of $\ell\sqrt{d}$, we can now further say that this is a subset of
\begin{equation*}
		\{y\in S^d \, :\,\dists(x,y)\leq 2r_1\Delta \}\cup \{y\in S^d \, :\,\dists(-x,y)\leq (2r_1+\sqrt{d})\Delta \}.
\end{equation*}
We follow that any $ E_y $, where $ y $ has graph distance of more than $ 2(2r_1+\sqrt{d}) $  from $ x $ and $ x^o $ is independent of $ E_x $.
\end{proof}
	\begin{lemma}\label{independentexpl}
		There exists a dimension dependent integer $ A=A(d)\in \NN $ such that for all iteration steps $ a_{i}<A$.
	\end{lemma}	
	\begin{proof}
		Note that by \cref{mindistproj} for all $ i $ the exploration of the status of $ x_iy_i $ only depends on points that fall inside $ B_{x_i}(\typ,3\lVert g_{x_i}(y_i)\rVert)\subset B_{x_i}(\typ,100 \ell_0) $. By \cref{nointersect} there is now a constant $ c=c(d)>0 $ such that there are only $ c $ points $ x_j\in H $ where  the sets $ R_i $ and $ R_j $ can intersect if $ x_j $ is an endpoint of an exploration for an edge.
		\end{proof}

 \noindent       
In order to use \cref{ABlongconnection} we also have to make sure that the relevant previous area is actually admissible. 
\begin{lemma}\label{admissible}
	For all iteration steps, $ g_{x_i}(R_i)\cap B_{x_i}(\typ,3\norm{g_{x_i}(y_i)}) $ is admissible, provided that $n$ is large enough.
\end{lemma}
\begin{proof}
	Assume we finished exploring for $ x_jy_j $ and added some $ y\in B_{x_j}(\typ,3\norm{g_{x_j}(y_j)}) $ to  $ g_{x_j}(R_j) $. Then, by construction, there also is a $ x\in B_{x_j}(y,1) $ with $ B_{x_j}(x,1)\subseteq g_{x_j}(R_j)  $. Note that by \cref{nointersect}, area that gets revealed for $ x_jy_j $ will only intersect $ B_{x_i}(\typ,3\norm{g_{x_i}(y_i)}) $ for $ x_i $ that have graph distance less than $ m $ from $ x_j $ or $ x_j^o $. If $ x_i $ has distance less than $ m $ from $ x_j $ we can choose $ n $ large enough such that $ x_j\in C_t(x_i) $. Then we obtain by using \cref{tcloseto1}
	\begin{equation*}
		\norm{g_{x_i}(g^{-1}_{x_j}(x))-g_{x_i}(g^{-1}_{x_j}(y))}\leq \frac{2}{\alpha(n)}\dists(g_{x_j}^{-1}(x),g_{x_j}^{-1}(y))\leq 2 \norm{x-y}<2.
	\end{equation*}
	Likewise, for all $ v\in \partial B_{x_j}(x,1) $ it holds
	\begin{equation*}
		1=\norm{v-x}\leq \frac{2}{\alpha(n)}\dists(g_{x_j}^{-1}(v),g_{x_j}^{-1}(x))\leq {2}\norm{g_{x_i}(g_{x_j}^{-1}(v))-g_{x_i}(g_{x_j}^{-1}(x))}.
	\end{equation*}
	The same inequalities go through if $x_i  $ has a graph distance of less than $m $ to $ x_j^o $. 
	It follows that $ g_{x_i}(R_i)\cap B_{x_i}(\typ,3\norm{g_{x_i}(y_i)}) $ is admissible.
\end{proof}
\subsubsection{Large open cluster in $H$}
In the following we are going to write $ |H| $ for the number of vertices of $ H $, which only depends on $ \eta=\eta(n,\ell_0,d) $ and the dimension.
\begin{lemma}\label{largecluster}
	For all $ \eps_1>0 $ we can choose $ n_0=n_0(\eps_1,\ell_0,k,d) $ large enough such that for all $ n\geq n_0 $ the following holds. After our exploration procedure terminates, with a probability of more than $ 1-\eps_1 $, the connected component of $ x_1 $ is greater than $ (1-\eps_1)|H| $.
\end{lemma}	
\begin{proof}
By the choice of $ \ell_0,k,\delta $, the definition of our algorithm, \cref{ABlongconnection} and the lemmas of \cref{sec:requirements}, each edge that gets processed in our exploration is set to be open with probability of more than $ 1-\eps_1 $, provided that $n$ is sufficiently large. This lower bound holds uniformly over the revealed set $ R_i $ and also over any points we revealed. Further, by the definition of our exploration all edges that are set to be open, are connected to $ x_1 $ via an open path and we will run the exploration for every edge that already has an adjacent edge that is connected to $ x_1 $. Thus, we can couple the exploration of the edges with the cluster of $ x_1 $ in usual bond percolation on $ H $ with parameter $ 1-\eps_1 $ such that the cluster of $ x_1 $ is a subset of the component our exploration yields. As $ \eta^{-1}\pi(H) = \ZZ^d\cap\{-m,\dots,m\}^d $ we can deduce the claim from \cref{dependentperco} by noting that $m\to \infty$ as $n\to\infty$.
\end{proof}

\noindent
We write $\pi_s\coloneqq\pi_{(0,\dots,0,1)}$ for the projection through the south pole. Similarly to $H$ we now define the mirrored version $H^2\coloneq \pi_s^{-1}(\eta \ZZ^d\cap D) $, where we remember we had $D\coloneqq [-\eta m, \eta m]^d$ with $m=\lfloor{{\sqrt{3}}/{\sqrt{d}\eta}}\rfloor$.
Additionally let $D^o \coloneq[-0.9 m\eta,0.9 m\eta ]^d \subset D$ and
\begin{equation*}
    S\coloneqq \{ x\in S^d\,:\, \pi(x)\in D^o, \pi_s(x)\in D^o \}.
\end{equation*}
Further, we let $b\coloneqq 0.9 m\eta/2$ and $a\coloneqq 1/\sqrt{2d}$. Note that for all large enough $n$ it holds that $b>a$ since $m\eta$ converges to $\sqrt{3/d}$. We define
    \begin{equation*}
        B\coloneqq \{z\in \RR^d\,:\, a\leq |z_j|\leq b \text{ for all }j=1,\dots,d\}
    \end{equation*}
    to obtain $B\subset D^o$ and to define  
    \begin{equation*}
        H_B\coloneqq\pi^{-1}(\eta\ZZ^d\cap B)\quad\text{and}\quad H_B^2\coloneqq\pi_s^{-1}(\eta\ZZ^d\cap B).
    \end{equation*}
We also define the function $I:\RR^ d\setminus\{\typ\}\to\RR^d$ with
\begin{equation*}
    I(z)\coloneq\frac{z}{\norm{z}^2}.
\end{equation*}
Note that as for $x\in S^d$ we have
\begin{equation*}
    \frac{1-x_{d+1}}{1+x_{d+1}}=\frac{1}{\norm{\pi(x)}^2}=\norm{\pi_s(x)}^2,
\end{equation*}
we can easily derive the identities $\pi_s(x)=I(\pi(x))$ and $\pi(x)=I(\pi_s(x))$. The two smaller graphs $H_B$ and $H_B^2$ are needed to deal with boundary issues in the following Lemma.

\begin{lemma}\label{halfpoints}
	If $n$ is large enough, the following hold:
	\begin{enumerate}[label=(\roman*)]
		\item  There is a constant $c=c(d)>0$ such that
		\begin{equation*}
		    |H\cap S| = |H^2\cap S|\geq  |H_B| = |H_B^2| \;\ge\;  c|H|.
		\end{equation*}
		\item 
		For every $u\in H_B$ there exists $v\in H^2\cap S$ with
		\begin{equation*}
		    \norm{g_u(v)}\;\le\; 12\sqrt{d}\ell_0,
		\end{equation*}
		and conversely for every $v\in H_B^2$ there exists $u\in H\cap S$ with the same bound.
	\end{enumerate}
\end{lemma}

\begin{proof}
    It holds that
    \begin{equation*}
         |\eta \ZZ^d\cap B|\geq \bigg(\frac{2(b-a)}{\eta}-2\bigg)^d,
    \end{equation*}
    and also $|H|\leq (2m+1)^d\leq (3m)^d$. Since $m\eta $ converges to a constant as $n\to \infty$ and $b-a$ is also a constant depending on $d$ once $n $ is large, the ratio $|\eta \ZZ^d\cap B|/|H|$ is bounded from below by some $c=c(d)>0$ for all large enough $n$. This yields $|H_B|\geq c |H|$ and by symmetry the same bound holds for $H_B^2$. For the second claim take some $u\in H_B $ and let $p\coloneq \pi_s(u)=I(\pi(u))\in D^o$. We define $q\in \eta \ZZ^d$ by rounding each coordinate of $p$ to the nearest multiple of $\eta$ and set $v\coloneq \pi_s^{-1}(q)$. Since $p\in I(B)\subset D^o$ and $I(B)$ is compact and strictly contained in $D^o$, for all sufficiently large $n$ the rounding point lies in $D^o$. Since $z=\pi(u)\in B$, every coordinate of $z$ satisfies $|z_j|\leq b$, hence
    \begin{equation*} 
        \inf_{x\in \partial D^o}\norm{z-x}_\infty\geq 0.9 m \eta -b =b. 
    \end{equation*}
    We define $r_0\coloneqq\min\{\norm{w}\,:\,w\in I(B)\}>0$ and also $M\coloneqq\max\{\norm{w}\,:\, w\in I(B)\}<\infty$. As $\norm{p-q}\leq (\sqrt{d}/2)\cdot\eta\to 0$ we have for all large enough $n$ that $\norm{p-q}<r_0/2$ and $\norm{p-q}<1$. We now have that $\norm{q}\geq \norm{p}-\norm{p-q}\geq r_0/2$ and also $\norm{q}\leq \norm{p}+\norm{p-q}\leq M+1$. Thus, both $p$ and $q$ lie in the compact set $K\coloneq\{x\,:\, r_0/2\leq \norm{x}\leq M+1\}$. Hence $I$ is uniformly continuous on $K$. Since $\norm{p-q}\leq \sqrt{d} \eta /2\to 0$, for all sufficiently large $n$ we have
    \begin{equation*}
        \norm{I(p)-I(q)}_\infty <b.
    \end{equation*}
    As $I(p)=z$, it follows that $I(q)\in D^o$. Therefore $\pi(v)=I(q)\in D^o$, so $v\in S$. As by \cref{lem:metric} 
     \begin{equation*}
        \dists(u,v) \leq 2 \norm{\pi_s(u)-\pi_s(v)}=2\norm{p-q}\leq \sqrt{d}\eta,
     \end{equation*}
     and by using \cref{tcloseto1} it holds for all large enough $n$ 
     \begin{equation*}
         \norm{g_u(v)}= \frac{2}{\alpha(n)}\norm{\pi_u(v)}\leq \frac{2}{\alpha(n)}\dists(u,v),
     \end{equation*}
     so we can conclude the proof.
\end{proof}

\begin{lemma}\label{oddcycle}
	For all $ \eps,\eps_2>0 $ and if $  \alpha(n)\geq (c_2+\eps)n^{-1/d} $, we can choose $ n_0=n_0(\eps,\eps_2) $ large enough such that for all $ n\geq n_0 $
	\begin{equation*}
		\Pro(\text{there exists }X_i\in X\text{ that is connected to }-X_i\text{ in }  G_{geo})>1-\eps_2.
	\end{equation*}
\end{lemma}

\begin{proof}
	For $ \eps_3>0$ we can use \cref{largecluster,halfpoints} to get $ \ell_0,k  $ such that for all large enough $ n $, with a probability of more than $ 1-\eps_3 $ the number of points of $ H $ that lie inside $ S $ that are connected to $ x_1 $ by an open path is more than $ |H|(c-\eps_3) $ for some $c=c(d)>0$. We will write $ F $ if this event occurs. In the definition of $ H $ we have used the stereographic projection $ \pi $ that projects from the north pole and we started exploring from the south pole $ x_1 $. By symmetry we could have also used the projection from the south pole, $ \pi_{s} $ and start our exploration at the north pole to determine whether  $ -x_1(-y_1) $ is open with respect to $H^2$. Note that again by symmetry, so changing the roles of $ X $ and $ Y $, if $ F $ occurs, then there are also more than $ |H|(c-\eps_3) $ points of $ H^2 $ that are connected to $ -x_1 $ by an open path in $ H^2 $. All $ x\in H $ that are part of an open edge $ xy $ have at least one point inside $ \{z\in S^d:\dists(z,x)\leq\dists(x,y) \} $ that is connected in the Borsuk graph to $ B_k $. In the same sense if the edge $ x^2y^2 $ is open in $ H^2 $, the set $ \{z\in S^d:\dists(z,x^2)\leq\dists(x^2,y^2) \} $ is connected to $ -B_k $. Also, each point inside $ B_k $ (resp. $ -B_k $) has a mirror image inside $ -B_k $ (resp. $ B_k $). By \cref{halfpoints} for every $x^1\in H_B$ there is a $x^2\in H^2\cap S$ with $ \norm{g_{x^1}{(x^2)}}\leq 12\sqrt{d}\ell_0 $. We call $ x^1\in H_B $ \textit{good} if in addition to this distance requirement it holds that $ x^1 $ is connected to $ x_1 $ inside $ H $ and $ x^2\in H^2\cap S $ is connected to $ -x_1 $ w.r.t. $ H^2 $. We also see that for $ y^2 $ being adjacent to $ x^2 $ and $ z$ with $ \dists(z,x^2)\leq\dists(x^2,y^2) $ we can bound 
	\begin{equation*}
		\begin{split}
			\norm{g_{x^1}(x^2)-g_{x^1}(z)}&\leq \frac{2}{\alpha(n)}\dists(x^2,z)\leq \frac{4}{\alpha(n)}\eta= 24 \ell_0.
		\end{split}
	\end{equation*}
	For each $ x\in H $ we are going to define the event 
	\begin{equation*}
		E_x\coloneqq\{B_x(\typ,36\sqrt{d}\ell_0)\text{ is covered w.r.t. }\mathcal{V}_{\lambda-\lambda_{A},x}\}.
	\end{equation*}
 Note that if in addition to $ F $ at least one $ E_{x} $ occurs for a good $ x $, $ B_k $ is connected to $ -B_k $ and by symmetry there has to exist $ X_i\in B_k $ that is connected to $ -X_i\in (-B_k) $. Also by \cref{independentexpl}, we have never revealed $ \mathcal{V}_{\lambda-\lambda_{A},x} $ for any $ x\in H $ during our exploration process. By \cref{nointersect,halfpoints} we can conclude that there is a constant $ c_2=c_2(d) >0$ such that we can find $ J\subset H_B $ with $ |J|\geq c_2|H| $ such that for any two $ x,y\in J $ it holds that $ E_x $ and $ E_y $ are independent. Using this independence and that under $ F$ there are less $ \eps_3 |H|$ points of $ H $ (resp. $ H^2 $) that are not connected to $ x_1 $ (resp. $ -x_1 $) we obtain 
	\begin{equation*}
		\begin{split}
			\Pro(E_x\text{ occurs for at least one good }x\,\vert\,F)\geq1-\Pro(E_x^c\,\vert\, F)^{|H|(c_2-2\eps_3)}.
		\end{split}
	\end{equation*}
As $ \Pro(E_x\,\vert\, F)>0 $ depends only on $ (\ell_0,\lambda,d) $,
	the lemma follows by letting $ n $ be large enough and choosing $ \eps_3 $ appropriately, together with the fact that $ |H|\to\infty $ with $ n\to\infty $.
\end{proof}

\begin{lemma}\label{finallemmachrom3}
		For all $ \eps,\eps_2>0 $ and if $  \alpha(n)\geq (c_2+\eps)n^{-1/d} $, we can choose $ n_0=n_0(\eps,\eps_2) $ large enough such that for all $ n\geq n_0 $
		\begin{equation*}
		\Pro(\chi(G(\mathcal{X},\alpha))\geq 3)\geq 1-\eps_2.
		\end{equation*}
\end{lemma}

\begin{proof}
	Assume that there exists $  X_i\in X $ that is connected to $ -X_i $ in   $ G_{geo} $. Then, there is a path $ x_0,y_0,x_1,y_1,\dots,x_{k},y_{k} $ with $ x_0=X_i, y_k=-X_{i}  $ and $ x_j\sim y_{j}\sim x_{j+1} $ for all $ 0\leq j\leq k-1 $ and also $ x_k\sim y_k $. Note that in the Borsuk graph this implies the existence of the path $ x_0,-y_0,x_1,-y_1,\dots,x_{k},-y_{k} $. Thus, $ X_i $ is connected to itself via an odd cycle. Since no odd cycle can be colored with just two alternating colors, the statement we are proving follows directly from \cref{oddcycle}.
\end{proof}

\subsection{Bipartite regime}

We will write $ \mathcal{W}_{\lambda,x} $ for the AB model as given by \cref{coupleab2} if we 
use  $g_x(\cdot)$ to project the points from $ C_t(x) $ (by symmetry the identical claim also holds in these cases).

\begin{lemma}\label{chrom2}
	For all $ \eps>0 $ and if $  \alpha(n)\leq (c_2-\eps)n^{-1/d} $ 
	\begin{equation*}
		\Pro(\chi(G(\mathcal{X},\alpha))\leq  2)= 1-o(1).
	\end{equation*}
\end{lemma}
\begin{proof}
    We let $t=t(\eps)$ be as provided by \cref{coupleab2}. For each $X\in \mathcal{X}$ define
    \begin{equation*}
        E_{X}\coloneqq\{ \text{the component of }X \text{ in } G_{geo} \text{ is not contained inside } C_t(X)\} 
    \end{equation*}
    and also for each $x\in S^d$ 
    \begin{equation*}
        J_x\coloneqq\{\exists y\in \mathcal{W}_{\lambda,x}\cap B_x(\typ,\rho(t,n)/4)\,:\, y\longleftrightarrow  B_x(\typ,\rho(t,n)/2)^c\}.
    \end{equation*}
    We note that by \cref{coupleab2} it holds that $E_{X}\subseteq J_{{X}}$.
    We express the total number of occurrences of $E_{X}$ as 
    \begin{equation*}
        N\coloneqq\sum_{X\in \mathcal{X}} \mathbbm{1}_{E_{X}}.
    \end{equation*}
    Since $\mathcal{X}$ is a Poisson point process on $S^d$ with intensity $n/((d+1)\kappa_{d+1})$ we can write by using the Mecke formula (see e.g. Theorem 3.2.5 of \cite{schneider2008stochastic}) and \cref{abdecaycor}
    \begin{equation*}
    \begin{split}
        \Ex N &=\frac{n}{(d+1)\kappa_{d+1}}  \int_{S^d}\Pro(E_u\text{ occurs w.r.t. }\mathcal{X}\cup\{u\} )\,\text{d}\sigma(u),
    \end{split}
    \end{equation*}
    where $d\sigma$ is the surface measure on $S^d$. Note that since $\rho(t,n)/4>2$ for all large enough $n$, the probability of $E_u$ occurring is unchanged if we add a deterministic point of either label $A$ or $B$ to the center of $B_u(\typ,\rho(t,n)/4)$. Therefore, by \cref{abdecaycor} we obtain a constant $c=c(\eps)>0$ such that for all $u\in S^d$
    \begin{equation*}
        \Pro(E_u\text{ occurs w.r.t. }X\cup\{u\} )\leq\Pro(J_u) \leq e^{-c\rho(t,n)}.
    \end{equation*}
     We conclude that $\Ex N =o(1)$ as $n\to\infty$. If $\Ex N=0$, almost surely for all $ X\in \mathcal{X} $ it holds that the component of $ X $ in $ G_{geo}$ is contained in $ C_t(X) $. If this is the case, for each $  G_t(X)  $ we can assign a color to the points of $ \mathcal{X} $ and another color to the points of $ \mathcal{M} $, since on  $ C_t(X) $ these sets are disjoint, and all particles are now connected to a different color than themselves. Since any two coloring of $G_{geo}$ implies a two coloring of $G(\mathcal{X},\alpha)$, we can conclude the proof.
\end{proof}

\begin{proofof}{\cref{bt1}}
    \cref{bt1} follows directly by combining  \cref{finallemmachrom3} and \ref{chrom2}.
\end{proofof}

\section{Proof of Theorem~\ref{bt2}.\label{sec:bt2}}

For the Borsuk graph $G(n,\alpha)$ on $ n $ vertices we write $ I_{ij}=\indicator{X_i\text{ is connected to }X_j} $. Then, we can also express the total number of edges as
\begin{equation*}
	W_n\coloneqq\sum_{i<j}I_{ij}
\end{equation*}
and the connection probability as 
\begin{equation*}
	p_n\coloneqq\Pro(I_{ij}=1)=\Pee( U \in \text{cap}((0,\dots,0,1),\alpha), 
%
%
\end{equation*}

\noindent
where $U \isd \text{Unif}(S^d)$.

\begin{lemma}\label{probconv}
If $\alpha = \alpha(n)\to 0$ as $ n\to \infty $ then 

$$  p_n=(1+o_n(1)) \cdot c_d 
 \cdot \alpha^d,
$$
	where
	\begin{equation*}
		c_d\coloneqq\frac{\Gamma((d+3)/2)}{(d+1)\cdot\sqrt{\pi}\cdot\Gamma((d+2)/2)}
	\end{equation*}
\end{lemma}

\begin{proof}
Writing $Z := \pi(U)$, parts~\ref{itm:metriccc1} and~\ref{itm:metriccc2} of Lemma~\ref{lem:metric} tell us that

$$ \Pee( Z \in B(\orig,\alpha/2 )) \leq p_n \leq \Pee( Z \in B(\orig,(1+\eps)\cdot\alpha/2 )), $$

\noindent 
for $\eps>0$ an arbitrary fixed small constant and $n$ sufficiently large (using $\alpha(n)\to 0$). 
Invoking Lemma~\ref{lem:piunif}, this gives

$$ \frac{\left(\frac{2}{1+(\alpha/2)^2}\right)^d}{(d+1)\kappa_{d+1}}  \cdot \kappa_d (\alpha/2)^d 
\leq p_n \leq \frac{2^d}{(d+1)\kappa_{d+1}} \cdot \kappa_d((1+\eps)\alpha/2)^d. $$

\noindent 
Sending $\eps\searrow 0$ and using that $\alpha \to 0$, we see that indeed 

$$ p_n = (1+o_n(1)) \cdot \frac{\kappa_d}{(d+1)\kappa_{d+1}} \cdot \alpha^d 
= (1+o_n(1)) \cdot 
\frac{\Gamma((d+3)/2)}{(d+1)\cdot \sqrt{\pi} \cdot \Gamma((d+2)/2)} \cdot \alpha(n)^d, $$

\noindent 
where we have used~\eqref{eq:volball} for the final identity.
\end{proof}

\begin{lemma}\label{empty2}
	If $ n^2\alpha(n)^d\to \nu\in (0,\infty) $  it holds that
	\begin{equation*}
		W_n\xrightarrow{\mathcal L} \emph{\text{Poi}}( (c_d/2) \cdot \nu ).
	\end{equation*}
\end{lemma}

\begin{proof}
	We are going to show that for all $ m\in \NN $
	\begin{equation*}
		\lim_{n \to \infty}\Ex (W_n)_m = (c_d\nu/2)^m.
	\end{equation*}
	The proof is then done by \cref{fallingmethod}. We can write
	\begin{equation}\label{factorialsum}
		\Ex (W_n)_m =\sum_{\substack{e_1,\dots,e_m\\ \text{distinct}}}\Pro(I_{e_1}=1,\dots,I_{e_m}=1),
	\end{equation}
	where the sum is over all ordered $ m $-tuples of distinct candidate edges. We write $ 2\leq s\leq 2m $ for the number of unique vertices that occur in $ (e_1,\dots,e_m)  $ in the sum above. For $ s=2m $ there are $ (n)_{2m}/2^m $ possible choices of $ (e_1,\dots,e_m) $. This holds since $ (n)_{2m} $ gives the number of ordered lists of length $ 2m $ out of $ n $ vertices, but an edge is an unordered pair we count each of the $ m $ edges twice. If all vertices are different it holds by independence that $ \Pro(I_{e_1}=1,\dots,I_{e_m}=1)=p_n^m $. Thus, the $ s=2m $ part of the sum is equal to 
	\begin{equation*}
		\frac{(n)_{2m}}{2^m} p_n^m= n^{2m}\prod_{i=0}^{2m-1}\bigg(1-\frac{i}{n}\bigg)\cdot\bigg(\frac{p_n^m}{2^m}\bigg)=\bigg(\frac{n^2p_n}{2}\bigg)^m(1+o(1)),
	\end{equation*}
	and the right hand side converges to $ (c\nu/2)^m $ by assumption and \cref{probconv}. We will now fix $ e=(e_1,\dots,e_m)  $ that corresponds to some $ s $ with $ 2\leq s\leq 2m-1 $. For that, we will write $ c=c(e) $ for the number of connected components w.r.t. the graph $ G=G(e) $ with structure $ e $. We will write $ F=F(e) $ for an arbitrary spanning forest of $ G $. For each tree we can start at one arbitrary vertex as the root and reveal a point that is a candidate for an edge. We can iterate this and in each revealment we choose a point that is already part of the tree and check whether an unrevealed point that was a candidate for a connection falls into the necessary cap. Then, in each step the probability of an edge being present equals $ p_n $. Note that there are $ e(F)=s-c$ edges in total. Therefore by independence 
	\begin{equation*}
		\Pro(\text{all edges are present in }G)\leq \Pro(\text{all edges are present in  }F)=  p_n^{e(F)}.
	\end{equation*}
	We further claim to have $ s\geq 2c +1 $. This can be seen as follows. First, $ s\geq 2c $, as every component has at least two vertices. If $ s=2c $ would be true every component would have exactly two vertices and one edge, leading to $ s=2m $. But we assumed $ s<2m $ here, leading to the contradiction. We conclude that $ p_n^{e(F)}\leq p_n^{s/2+1/2} $. We now want to compute an upper bound on the possible number of edges $(e_1,\dots,e_m)  $ that correspond to a fixed $ s $. First, we choose an ordered list of $ s $ distinct vertex labels. There are less than $ (n)_s\leq n^s $ possible ways of doing that. We can now form $ \binom{s}{2} $ many unordered pairs out of them. There are $\binom{s}{2}^m  $ many ways of forming an ordered $ m $ tuple made out of these pairs. We note that this construction over-counts (the list of $ s $ vertex labels is ordered and we allowed for repetitions in the $ m $ tuple), but the bound is good enough for our needs. Finally, we sum over all possible  $ c\leq s/2 $ to conclude with our derivation above that each fixed $ 2\leq s\leq 2m-1 $ has a contribution for \cref{factorialsum} of less than
	\begin{equation*}
		n^s\binom{s}{2}^m\cdot \frac{s}{2}\cdot p_n^{s/2+1/2} =\frac{s}{2}\binom{s}{2}^m (n^2p_n)^{s/2+1/2} \cdot n^{-1}.
	\end{equation*}
	As $ s $ and $ m $ are fixed and $ n^2p_n\to c_d\nu $ by our assumption and \cref{probconv}, the parts that do not correspond to $ s=2m $ vanish with $ n\to \infty $.
\end{proof}

\begin{proofof}{\cref{bt2}}
    \cref{bt2} follows from \cref{empty2} by noting that for $X\sim\text{Po}(\mu)$ it holds 
    that $\Pro(X=0)=e^{-\mu}$ and that $W_n=0$ implies $\chi(G(n,\alpha))=1$. 
\end{proofof}

\section{Discussion and suggestions for further work.\label{sec:discuss}}

Theorem~\ref{thm:main} showed that the property of having chromatic number $> k$ has a sharp threshold 
``for almost all $n$''.
Naturally we conjecture that this last addition is not needed:

\begin{conjecture}
For every $d\geq 2$ and $3 \leq k \leq d+1$, the property of having chromatic number $>k$ 
has a sharp threshold.
\end{conjecture}

\noindent
The proof of Theorem~\ref{thm:main} leveraged the semi-sharpness of the threshold to show that the value of 
$\alpha$ where the probability of chromatic number $> k$ is exactly 1/2
as a function of $n$ 
can only ``jump'' relatively rarely. If one could argue that instead of rarely, it does not jump at all (for 
large $n$) then one would have proved the conjecture. 

We offer the following more detailed conjectures stating also the form of the threshold that 
most researchers in the field will probably expect:

\begin{conjecture}\label{conj:d+2}
For every $d\geq 2$ there exists a constant $c=c(d)$ such that, for every fixed $\eps>0$
and sequence $\alpha=\alpha(n)$:

$$ \lim_{n\to\infty}\Pee(  \chi(G(n,\alpha)) > d+1 ) = 
\begin{cases}
1 & \text{ if $\alpha > (1+\eps)\cdot c\cdot (\ln n /n)^{1/d}$, } \\
0 & \text{ if $\alpha < (1-\eps)\cdot c\cdot (\ln n /n)^{1/d}$. }
\end{cases}. $$

\end{conjecture}

\begin{conjecture}\label{conj:conjk}
For each $d\geq 2$, there exist constants $c_3=c_3(d),\dots,c_{d}=c_{d}(d)$ such that 

$$ c_2 < c_3 < \dots < c_{d}, $$ 

\noindent 
where $c_2$ is as provided by Theorem~\ref{bt1} and, for every fixed $\eps>0$
and $3 \leq k \leq d$, and every sequence $\alpha=\alpha(n)$:

$$ 
\lim_{n\to\infty}\Pee(  \chi(G(n,\alpha)) > k ) = 
\begin{cases}
1 & \text{ if $\alpha > (1+\eps)\cdot c_k\cdot n^{-1/d}$, } \\
0 & \text{ if $\alpha < (1-\eps)\cdot c_k\cdot n^{-1/d}$. }
\end{cases}. $$

\end{conjecture}

\noindent
The latter conjecture would in particular answer Question 2 of Kahle and Martinez-Figueroa (\cite{KahleFig}, Section 4).

The proof of Theorem~\ref{bt1} (establishing the existence of $c_2$) hinged on the 
fact that chromatic number $\geq 3$ is the same as having an odd cycle and 
a comparison with continuum AB percolation.
One might hope that, for $3 \leq k\leq d$, the constant $c_k$ will correspond to the 
critical intensity for some sort of ``higher order percolation event'' in the continuum AB percolation model. 
It is however not so clear (at least to the authors) what might be a useful property or characterization of graphs with chromatic 
number $> k$ that we could exploit in place of existence of odd cycles.

For Conjecture~\ref{conj:d+2} the relevant regime is where the
average degree is of logarithmic order. Here we can imagine that a more careful analysis along the lines 
of the arguments in~\cite{KahleFig} might work. In particular, perhaps there is an (approximate) correspondence 
between chromatic number $\geq d+2$ and some 
kind of ``coverage of the sphere by random caps'' event.

Lov\'asz' pioneering paper~\cite{Lovasz78} gives a general lower bound for the chromatic number : 
if the neighbourhood complex of $G$ is $k$-connected then $\chi(G) \geq k+3$.
(See e.g.~\cite{MatousekBorsukUlam} for the relevant definitions and much more background.)
It can be shown that the neighbourhood complex of the 
subgraph of the random Borsuk graph that we construct in the proof of
Theorem~\ref{thm:maind+1} is in fact homotopy equivalent to $S^{d-1}$ and hence $(d-2)$-connected
(see e.g.~Theorem 4.3.2 in~\cite{MatousekBorsukUlam}).
We have chosen not to follow this route in our proof and instead invoke Lemma~\ref{lem:KFchi}, which is quicker.
It does however suggest the following conjecture. 
For $n$ fixed, we can imagine increasing $\alpha$ from zero to $\pi$ while keeping the
random points fixed.
This results in what might be called the ``random Borsuk graph {\em process}'' : edges are included one-by-one in 
order of decreasing geodesic distance between their endpoints.
Analogously to the {\em hitting time} for Erd\H{o}s-R\'enyi random graph and {\em hitting radius} for random geometric 
graphs, we can define the {\em hitting angle} $\alpha(n,\Pcal)$ for a given monotone property $\Pcal$ (preserved under the addition of edges but not 
necessarily vertices) as the (random) value of $\alpha$ where the random Borsuk graph process first 
attains the property $\Pcal$. In other words it equals $\pi$ minus the geodesic length of the edge 
being added the moment $\Pcal$ starts holding.
We offer the following conjecture.

\begin{conjecture}
For $k=3,\dots,d+1$ let us write $\Pcal$ for the 
property of having chromatic number $>k$ and $\Qcal$ for the property of 
containing a subgraph whose neighbourhood complex is homotopy equivalent to $S^{k-1}$.
The we have 

$$ \Pee( \alpha(n,\Pcal) = \alpha(n,\Qcal) ) = 1-o_n(1). $$

\end{conjecture}

We point out that for $k=2$ the assertion in the above conjecture in fact holds.
(Having chromatic number at least three is equivalent to containing an odd cycle.
The neighbourhood complex of an odd cycle is homotopy equivalent 
to $S^1$, while the neighbourhood complex of any bipartite graph with more than one vertex is disconnected, so in particular
not homotopy equivalent to $S^1$.)

Conjecture~\ref{conj:conjk} would also establish the following weaker conjecture, that might be provable without 
first establishing Conjecture~\ref{conj:conjk}.

\begin{conjecture}\label{conj:twopoint}
For all $d \geq 1$ and all sequences $\alpha = \alpha(n)$ with $\alpha < \alpha_0$, where $\alpha_0$ is 
the constant provided By the Erd\H{o}s-Hajnal result mentioned in the introduction, there is a 
sequence $k=k(n)$ such that 

$$ \Pee\left( \chi(G(n,\alpha)) \in \{k,k+1\} \right) = 1-o_n(1). $$

\end{conjecture}

This would provide an analogue of the two-point concentration results for the Erd\H{o}s-R\'enyi random 
graph~\cite{Luczak91,AlonKrivelevich} and the random geometric graph~\cite{Mullertwopoint}.
We note that by the results in this paper and the results of Kahle and Martinez-Figueroa~\cite{KahleFig}, it suffices
to consider only sequences of $\alpha$ in the thermodynamic regime.

\subsection*{Acknowledgements}

We thank R\'eka Szabo for helpful discussions and in particular for pointing us to the reference~\cite{BalisterEtal}.
We thank Lyuben Lichev for helpful discussions in the early stages of the project.

\bibliographystyle{plain}
\bibliography{ref}

@article{BalisterEtal,
    AUTHOR = {Balister, P. and Bollob\'as, B. and Morris, R. and Smith, P.},
    TITLE = {Subcritical monotone cellular automata},
    JOURNAL = {Rand. Struct. Alg.},
    VOLUME = {64},
    NUMBER = {1},
    YEAR = {2024},
    PAGES = {38--61},
    DOI = {10.1002/rsa.21174},
    URL = {https://doi.org/10.1002/rsa.21174},
    ISSN = {1042-9832,1098-2418},
    MRCLASS = {60K35 (60C05)},
    MRNUMBER = {4672996},
    MRREVIEWER = {Rinaldo\ Schinazi}
}

@book{Bollobas2001RandomGraphs,
    AUTHOR = {Bollob{\'a}s, B.},
    TITLE = {Random Graphs},
    VOLUME = {73},
    YEAR = {2001},
    DOI = {10.1017/CBO9780511814068},
    SERIES = {Cambridge Studies in Advanced Mathematics},
    EDITION = {2},
    PUBLISHER = {Cambridge University Press},
    ADDRESS = {Cambridge},
    ISBN = {9780521809207}
}

@book{bollobas2006percolation,
    AUTHOR = {Bollob{\'a}s, B. and Riordan, O.},
    TITLE = {Percolation},
    YEAR = {2006},
    DOI = {10.1017/CBO9781139167383},
    PUBLISHER = {Cambridge University Press},
    ADDRESS = {Cambridge},
    ISBN = {9780521872324}
}

@article{ErdosHajnal67,
    AUTHOR = {Erd\H{o}s, P. and Hajnal, A.},
    TITLE = {On chromatic graphs},
    JOURNAL = {Mat. Lapok},
    FJOURNAL = {Matematikai Lapok. Bolyai J\'anos Matematikai T\'arsulat},
    VOLUME = {18},
    YEAR = {1967},
    PAGES = {1--4},
    ISSN = {0025-519X},
    MRCLASS = {05.55},
    MRNUMBER = {227050},
    MRREVIEWER = {P.\ Ungar}
}

@article{Friedgut99,
    AUTHOR = {Friedgut, E.},
    TITLE = {Sharp thresholds of graph properties, and the {$k$}-sat
                          problem},
    JOURNAL = {J. Amer. Math. Soc.},
    FJOURNAL = {Journal of the American Mathematical Society},
    VOLUME = {12},
    NUMBER = {4},
    YEAR = {1999},
    PAGES = {1017--1054},
    NOTE = {With an appendix by Jean Bourgain}
}

@article{GrimmettMarstrand1990Supercritical,
    AUTHOR = {Grimmett, G. R. and Marstrand, J. M.},
    TITLE = {The supercritical phase of percolation is well behaved},
    JOURNAL = {Proc. Roy. Soc. Lond. Ser. A},
    VOLUME = {430},
    NUMBER = {1879},
    YEAR = {1990},
    MONTH = aug,
    PAGES = {439--457},
    DOI = {10.1098/rspa.1990.0100}
}

@article{KahleFig,
    AUTHOR = {Kahle, M. and Martinez-Figueroa, F.},
    TITLE = {The chromatic number of random {B}orsuk graphs},
    JOURNAL = {Rand. Struct. Alg.},
    FJOURNAL = {Rand. Struct. Alg.},
    VOLUME = {56},
    NUMBER = {3},
    YEAR = {2020},
    PAGES = {838--850},
    DOI = {10.1002/rsa.20897},
    URL = {https://doi.org/10.1002/rsa.20897},
    ISSN = {1042-9832,1098-2418},
    MRCLASS = {05C80 (05C62 60G55)},
    MRNUMBER = {4084191},
    MRREVIEWER = {Andr\'as\ S\'andor\ Pluh\'ar}
}

@book{Kingmanboek,
    AUTHOR = {Kingman, J. F. C.},
    TITLE = {{Poisson} Processes},
    VOLUME = {3},
    YEAR = {1993},
    PAGES = {viii+104},
    SERIES = {Oxford Studies in Probability},
    PUBLISHER = {The Clarendon Press, Oxford University Press, New York},
    ADDRESS = {Oxford},
    ISBN = {0-19-853693-3},
    NOTE = {Oxford Science Publications},
    MRCLASS = {60G05 (60G55 60K99)},
    MRNUMBER = {1207584},
    MRREVIEWER = {Dietrich Stoyan}
}

@book{Last_Penrose_2017,
    AUTHOR = {Last, G. and Penrose, M. D.},
    TITLE = {Lectures on the {Poisson} Process},
    VOLUME = {7},
    YEAR = {2017},
    DOI = {10.1017/9781316104477},
    SERIES = {Institute of Mathematical Statistics Textbooks},
    PUBLISHER = {Cambridge University Press},
    ADDRESS = {Cambridge},
    ISBN = {9781107088016}
}

@book{lee2006riemannian,
    AUTHOR = {Lee, J. M.},
    TITLE = {{Riemannian} Manifolds: An Introduction to Curvature},
    VOLUME = {176},
    YEAR = {1997},
    DOI = {10.1007/b98852},
    SERIES = {Graduate Texts in Mathematics},
    PUBLISHER = {Springer},
    ADDRESS = {New York},
    ISBN = {978-0-387-98271-7}
}

@article{Lovasz78,
    AUTHOR = {Lov\'{a}sz, L.},
    TITLE = {Kneser's conjecture, chromatic number, and homotopy},
    JOURNAL = {J. Combin. Theory Ser. A},
    FJOURNAL = {Journal of Combinatorial Theory. Series A},
    VOLUME = {25},
    NUMBER = {3},
    YEAR = {1978},
    PAGES = {319--324},
    DOI = {10.1016/0097-3165(78)90022-5},
    URL = {https://doi.org/10.1016/0097-3165(78)90022-5},
    ISSN = {0097-3165,1096-0899},
    MRCLASS = {05C15 (57M15)},
    MRNUMBER = {514625}
}

@book{MatousekBorsukUlam,
    AUTHOR = {Matou\v{s}ek, J.},
    TITLE = {Using the {B}orsuk-{U}lam theorem},
    YEAR = {2003},
    PAGES = {xii+196},
    SERIES = {Universitext},
    PUBLISHER = {Springer-Verlag, Berlin},
    ISBN = {3-540-00362-2},
    NOTE = {Lectures on topological methods in combinatorics and geometry,
                          Written in cooperation with Anders Bj\"orner and G\"unter M.
                          Ziegler},
    MRCLASS = {55-01 (05-01 52-01 52A35 55M20)},
    MRNUMBER = {1988723},
    MRREVIEWER = {Zdzis\l aw\ Dzedzej}
}

@book{meester1996continuum,
    AUTHOR = {Meester, R. and Roy, R.},
    TITLE = {Continuum Percolation},
    VOLUME = {119},
    YEAR = {1996},
    DOI = {10.1017/CBO9780511895357},
    SERIES = {Cambridge Tracts in Mathematics},
    PUBLISHER = {Cambridge University Press},
    ADDRESS = {Cambridge},
    ISBN = {9780521475044}
}

@article{MullerConfetti,
    AUTHOR = {M\"uller, T.},
    TITLE = {The critical probability for confetti percolation equals 1/2},
    JOURNAL = {Rand. Struct. Alg.},
    FJOURNAL = {Rand. Struct. Alg.},
    VOLUME = {50},
    NUMBER = {4},
    YEAR = {2017},
    PAGES = {679--697},
    DOI = {10.1002/rsa.20675},
    URL = {https://doi.org/10.1002/rsa.20675},
    ISSN = {1042-9832,1098-2418},
    MRCLASS = {60K35 (82B43)},
    MRNUMBER = {3660524}
}

@book{penroseboek,
    AUTHOR = {Penrose, M. D.},
    TITLE = {Random geometric graphs},
    VOLUME = {5},
    YEAR = {2003},
    PAGES = {xiv+330},
    DOI = {10.1093/acprof:oso/9780198506263.001.0001},
    URL = {http://dx.doi.org/10.1093/acprof:oso/9780198506263.001.0001},
    SERIES = {Oxford Studies in Probability},
    PUBLISHER = {Oxford University Press},
    ADDRESS = {Oxford},
    ISBN = {0-19-850626-0},
    MRCLASS = {60-02 (05C80 60D05)},
    MRNUMBER = {1986198 (2005j:60003)},
    MRREVIEWER = {Ilya S. Molchanov}
}

@book{Schilling,
    AUTHOR = {Schilling, R. L.},
    TITLE = {Measures, integrals and martingales},
    YEAR = {2005},
    PAGES = {xii+381},
    DOI = {10.1017/CBO9780511810886},
    URL = {https://doi.org/10.1017/CBO9780511810886},
    PUBLISHER = {Cambridge University Press, New York},
    ISBN = {978-0-521-61525-9},
    MRCLASS = {28-01 (28A12 28A20 28A25 60G42 60G46)},
    MRNUMBER = {2200059},
    MRREVIEWER = {Peter Eichelsbacher}
}

@book{schneider2008stochastic,
    AUTHOR = {Schneider, R. and Weil, W.},
    TITLE = {Stochastic and Integral Geometry},
    YEAR = {2008},
    PAGES = {xii+693},
    DOI = {10.1007/978-3-540-78859-1},
    URL = {https://doi.org/10.1007/978-3-540-78859-1},
    SERIES = {Probability and its Applications (New York)},
    PUBLISHER = {Springer-Verlag, Berlin},
    ADDRESS = {Berlin, Heidelberg},
    ISBN = {978-3-540-78858-4},
    MRCLASS = {60-02 (52A22 60D05 60G55 62M30)},
    MRNUMBER = {2455326},
    MRREVIEWER = {V. K. Ohanyan}
}

@unpublished{MaEtalArXiv,
    AUTHOR = {Ma, J. and Shen, W. and Xie, S.},
    TITLE = {An exponential improvement for {R}amsey lower bounds},
    NOTE = {Preprint, 2025. Available from {\tt{arxiv.org/abs/2507.12926}}}
}

@article{IyerYogesh,
    AUTHOR = {Iyer, S. K. and Yogeshwaran, D.},
    TITLE = {Percolation and connectivity in {$AB$} random geometric
                          graphs},
    JOURNAL = {Adv. in Appl. Probab.},
    FJOURNAL = {Advances in Applied Probability},
    VOLUME = {44},
    NUMBER = {1},
    YEAR = {2012},
    PAGES = {21--41},
    DOI = {10.1239/aap/1331216643},
    URL = {https://doi.org/10.1239/aap/1331216643},
    ISSN = {0001-8678,1475-6064},
    MRCLASS = {60D05 (05C40 05C80 60K35)},
    MRNUMBER = {2951545}
}

@article{PenroseAB,
    AUTHOR = {Penrose, M. D.},
    TITLE = {Continuum {AB} percolation and {AB} random geometric graphs},
    JOURNAL = {J. Appl. Probab.},
    FJOURNAL = {Journal of Applied Probability},
    VOLUME = {51A},
    YEAR = {2014},
    PAGES = {333--344},
    DOI = {10.1239/jap/1417528484},
    URL = {https://doi.org/10.1239/jap/1417528484},
    ISSN = {0021-9002,1475-6072},
    MRCLASS = {05C40 (05C80 60K35 82B43)},
    MRNUMBER = {3317367},
    MRREVIEWER = {Emmanuel\ Jacob}
}

@article{DereudrePenrose,
    AUTHOR = {Dereudre, D. and Penrose, M.},
    TITLE = {On the critical threshold for continuum {AB} percolation},
    JOURNAL = {J. Appl. Probab.},
    FJOURNAL = {Journal of Applied Probability},
    VOLUME = {55},
    NUMBER = {4},
    YEAR = {2018},
    PAGES = {1228--1237},
    DOI = {10.1017/jpr.2018.81},
    URL = {https://doi.org/10.1017/jpr.2018.81},
    ISSN = {0021-9002,1475-6072},
    MRCLASS = {60D05 (60K35 82B43)},
    MRNUMBER = {3899938},
    MRREVIEWER = {Andrew\ R.\ Wade}
}

@unpublished{AdamsEtal,
    AUTHOR = {Adams, H. and Elchesen, A. and Mallick, S. and Moy, M.},
    TITLE = {Anti-{V}ietoris--{R}ips metric thickenings and {B}orsuk graphs},
    NOTE = {Preprint, 2025. Available from {\tt{arxiv.org/abs/2503.08862}}}
}

@article{Perkins_survey,
    AUTHOR = {Perkins, W.},
    TITLE = {Searching for (sharp) thresholds in random structures: where are we now?},
    JOURNAL = {Bull. Amer. Math. Soc. (N.S.)},
    FJOURNAL = {American Mathematical Society. Bulletin. New Series},
    VOLUME = {62},
    NUMBER = {1},
    YEAR = {2025},
    PAGES = {113--143},
    DOI = {10.1090/bull/1857},
    URL = {https://doi.org/10.1090/bull/1857},
    ISSN = {0273-0979,1088-9485},
    MRCLASS = {05C80 (68Q87 82D30)},
    MRNUMBER = {4845928},
    MRREVIEWER = {David\ B.\ Penman}
}

@article{Friedgut_hunting,
    AUTHOR = {Friedgut, E.},
    TITLE = {Hunting for sharp thresholds},
    JOURNAL = {Rand. Struct. Alg.},
    FJOURNAL = {Random Structures {\&} Algorithms},
    VOLUME = {26},
    NUMBER = {1-2},
    YEAR = {2005},
    PAGES = {37--51},
    DOI = {10.1002/rsa.20042},
    URL = {https://doi.org/10.1002/rsa.20042},
    ISSN = {1042-9832,1098-2418},
    MRCLASS = {05C80 (60C05)},
    MRNUMBER = {2116574},
    MRREVIEWER = {Michael\ Krivelevich}
}

@article{FriedgutKalai96,
    AUTHOR = {Friedgut, E. and Kalai, G.},
    TITLE = {Every monotone graph property has a sharp threshold},
    JOURNAL = {Proc. Amer. Math. Soc.},
    FJOURNAL = {Proceedings of the American Mathematical Society},
    VOLUME = {124},
    NUMBER = {10},
    YEAR = {1996},
    PAGES = {2993--3002},
    DOI = {10.1090/S0002-9939-96-03732-X},
    URL = {https://doi.org/10.1090/S0002-9939-96-03732-X},
    ISSN = {0002-9939,1088-6826},
    MRCLASS = {05C80},
    MRNUMBER = {1371123},
    MRREVIEWER = {Andrzej\ Ruci\'nski}
}

@article{McColm_threshold,
    AUTHOR = {McColm, G. L.},
    TITLE = {Threshold functions for random graphs on a line segment},
    JOURNAL = {Combin. Probab. Comput.},
    FJOURNAL = {Combinatorics, Probability and Computing},
    VOLUME = {13},
    NUMBER = {3},
    YEAR = {2004},
    PAGES = {373--387},
    DOI = {10.1017/S0963548304006121},
    URL = {https://doi.org/10.1017/S0963548304006121},
    ISSN = {0963-5483,1469-2163},
    MRCLASS = {05C80},
    MRNUMBER = {2056406},
    MRREVIEWER = {David\ B.\ Penman}
}

@article{GoelEtal,
    AUTHOR = {Goel, A. and Rai, S. and Krishnamachari, B.},
    TITLE = {Monotone properties of random geometric graphs have sharp
                          thresholds},
    JOURNAL = {Ann. Appl. Probab.},
    FJOURNAL = {The Annals of Applied Probability},
    VOLUME = {15},
    NUMBER = {4},
    YEAR = {2005},
    PAGES = {2535--2552},
    DOI = {10.1214/105051605000000575},
    URL = {https://doi.org/10.1214/105051605000000575},
    ISSN = {1050-5164,2168-8737},
    MRCLASS = {05C80 (60D05)},
    MRNUMBER = {2187303},
    MRREVIEWER = {Bert\ Fristedt}
}

@article{BollobasThomason87,
    AUTHOR = {Bollob\'as, B. and Thomason, A.},
    TITLE = {Threshold functions},
    JOURNAL = {Combinatorica},
    FJOURNAL = {Combinatorica. An International Journal of the J\'anos Bolyai
                          Mathematical Society},
    VOLUME = {7},
    NUMBER = {1},
    YEAR = {1987},
    PAGES = {35--38},
    DOI = {10.1007/BF02579198},
    URL = {https://doi.org/10.1007/BF02579198},
    ISSN = {0209-9683},
    MRCLASS = {05C80 (04A20 05A05)},
    MRNUMBER = {905149},
    MRREVIEWER = {Zbigniew\ Palka}
}

@article{ErdosRenyi60,
    AUTHOR = {Erd{\H{o}}s, P. and R\'enyi, A.},
    TITLE = {On the evolution of random graphs},
    JOURNAL = {Magyar Tud. Akad. Mat. Kutat\'o{} Int. K\"ozl.},
    FJOURNAL = {A Magyar Tudom\'anyos Akad\'emia. Matematikai Kutat\'o{}
                          Int\'ezet\'enek K\"ozlem\'enyei},
    VOLUME = {5},
    YEAR = {1960},
    PAGES = {17--61},
    ISSN = {0541-9514},
    MRCLASS = {05.40},
    MRNUMBER = {125031},
    MRREVIEWER = {John\ Riordan}
}

@article{GrimmettMcDiarmid75,
    AUTHOR = {Grimmett, G. R. and McDiarmid, C. J. H.},
    TITLE = {On colouring random graphs},
    JOURNAL = {Math. Proc. Cambridge Philos. Soc.},
    FJOURNAL = {Mathematical Proceedings of the Cambridge Philosophical
                          Society},
    VOLUME = {77},
    YEAR = {1975},
    PAGES = {313--324},
    DOI = {10.1017/S0305004100051124},
    URL = {https://doi.org/10.1017/S0305004100051124},
    ISSN = {0305-0041,1469-8064},
    MRCLASS = {05C15},
    MRNUMBER = {369129},
    MRREVIEWER = {D.\ Cvetkovi\'c}
}

@article{Bollobas88,
    AUTHOR = {Bollob\'as, B.},
    TITLE = {The chromatic number of random graphs},
    JOURNAL = {Combinatorica},
    FJOURNAL = {Combinatorica. An International Journal of the J\'anos Bolyai
                          Mathematical Society},
    VOLUME = {8},
    NUMBER = {1},
    YEAR = {1988},
    PAGES = {49--55},
    DOI = {10.1007/BF02122551},
    URL = {https://doi.org/10.1007/BF02122551},
    ISSN = {0209-9683},
    MRCLASS = {05C80 (05C15)},
    MRNUMBER = {951992},
    MRREVIEWER = {Zbigniew\ Palka}
}

@article{Heckel21,
    AUTHOR = {Heckel, A.},
    TITLE = {Non-concentration of the chromatic number of a random graph},
    JOURNAL = {J. Amer. Math. Soc.},
    FJOURNAL = {Journal of the American Mathematical Society},
    VOLUME = {34},
    NUMBER = {1},
    YEAR = {2021},
    PAGES = {245--260},
    DOI = {10.1090/jams/957},
    URL = {https://doi.org/10.1090/jams/957},
    ISSN = {0894-0347,1088-6834},
    MRCLASS = {05C15 (05C80 60C05)},
    MRNUMBER = {4188818},
    MRREVIEWER = {Mark\ R.\ Jerrum}
}

@article{Mullertwopoint,
    AUTHOR = {M\"uller, T.},
    TITLE = {Two-point concentration in random geometric graphs},
    JOURNAL = {Combinatorica},
    FJOURNAL = {Combinatorica. An International Journal on Combinatorics and
                          the Theory of Computing},
    VOLUME = {28},
    NUMBER = {5},
    YEAR = {2008},
    PAGES = {529--545},
    DOI = {10.1007/s00493-008-2283-3},
    URL = {https://doi.org/10.1007/s00493-008-2283-3},
    ISSN = {0209-9683,1439-6912},
    MRCLASS = {05C80 (05C15 60C05 60D05)},
    MRNUMBER = {2501248}
}

@article{McDiarmidMuller,
    AUTHOR = {McDiarmid, C. and M\"uller, T.},
    TITLE = {On the chromatic number of random geometric graphs},
    JOURNAL = {Combinatorica},
    FJOURNAL = {Combinatorica. An International Journal on Combinatorics and
                          the Theory of Computing},
    VOLUME = {31},
    NUMBER = {4},
    YEAR = {2011},
    PAGES = {423--488},
    DOI = {10.1007/s00493-011-2403-3},
    URL = {https://doi.org/10.1007/s00493-011-2403-3},
    ISSN = {0209-9683,1439-6912},
    MRCLASS = {05C80 (05C15)},
    MRNUMBER = {2861238}
}

@article{AlonKrivelevich,
    AUTHOR = {Alon, N. and Krivelevich, M.},
    TITLE = {The concentration of the chromatic number of random graphs},
    JOURNAL = {Combinatorica},
    FJOURNAL = {Combinatorica. An International Journal on Combinatorics and
                          the Theory of Computing},
    VOLUME = {17},
    NUMBER = {3},
    YEAR = {1997},
    PAGES = {303--313},
    DOI = {10.1007/BF01215914},
    URL = {https://doi.org/10.1007/BF01215914},
    ISSN = {0209-9683,1439-6912},
    MRCLASS = {05C80 (05C15)},
    MRNUMBER = {1606020},
    MRREVIEWER = {J.\ Spencer}
}

@article{AchlioptasFriedgut,
    AUTHOR = {Achlioptas, D. and Friedgut, E.},
    TITLE = {A sharp threshold for {$k$}-colorability},
    JOURNAL = {Rand. Struct. Alg.},
    FJOURNAL = {Rand. Struct. Alg.},
    VOLUME = {14},
    NUMBER = {1},
    YEAR = {1999},
    PAGES = {63--70},
    DOI = {10.1002/(SICI)1098-2418(1999010)14:1<63::AID-RSA3>3.0.CO;2-7},
    URL = {https://doi.org/10.1002/(SICI)1098-2418(1999010)14:1<63::AID-RSA3>3.0.CO;2-7},
    ISSN = {1042-9832,1098-2418},
    MRCLASS = {05C80 (05C15 60C05)},
    MRNUMBER = {1662274},
    MRREVIEWER = {Tomasz\ J.\ \L uczak}
}

@article{Luczak91,
    AUTHOR = {{\L}uczak, T.},
    TITLE = {A note on the sharp concentration of the chromatic number of
                          random graphs},
    JOURNAL = {Combinatorica},
    FJOURNAL = {Combinatorica. An International Journal on Combinatorics and
                          the Theory of Computing},
    VOLUME = {11},
    NUMBER = {3},
    YEAR = {1991},
    PAGES = {295--297},
    DOI = {10.1007/BF01205080},
    URL = {https://doi.org/10.1007/BF01205080},
    ISSN = {0209-9683},
    MRCLASS = {05C80 (05C15)},
    MRNUMBER = {1122014},
    MRREVIEWER = {Alan\ M.\ Frieze}
}

@article{AchlioptasNaor,
    AUTHOR = {Achlioptas, D. and Naor, A.},
    TITLE = {The two possible values of the chromatic number of a random
                          graph},
    JOURNAL = {Ann. of Math. (2)},
    FJOURNAL = {Annals of Mathematics. Second Series},
    VOLUME = {162},
    NUMBER = {3},
    YEAR = {2005},
    PAGES = {1335--1351},
    DOI = {10.4007/annals.2005.162.1335},
    URL = {https://doi.org/10.4007/annals.2005.162.1335},
    ISSN = {0003-486X,1939-8980},
    MRCLASS = {05C80 (60C05)},
    MRNUMBER = {2179732},
    MRREVIEWER = {Michael\ Krivelevich}
}

\appendix

\section{Proof of Lemma~\ref{lem:metric}.\label{sec:metric}}

The proof of Lemma~\ref{lem:metric} makes use of the following observation, which we separate out 
as a lemma.

\begin{lemma}
Let $f : [0,1] \to S^d \setminus \{(0,\dots,0,1)\}$ be a continuously differentiable curve, and set 
$g := \pi \circ f$.
Then 

$$ \norm{g'(t)} = \frac{1}{1-f_{d+1}(t)}\cdot\norm{f'(t)}, $$

\noindent
for all $0\leq t \leq 1$.
\end{lemma}

\begin{proof}
Writing $f(t)=\left(f_1(t),\dots,f_{d+1}(t)\right), g(t)=\left(g_1(t),\dots,g_d(t)\right)$, the 
chain rule tells us that

$$ g_i'(t) = \sum_{j=1}^{d+1} \left(\partial \pi_i/\partial x_j\right)_{x=f(t)} \cdot f_j'(t), $$

\noindent
The partial derivatives of $\pi$ are given by:

$$ \left(\partial \pi_i/\partial x_j\right) = \begin{cases}
                                               \displaystyle \frac{1}{1-x_{d+1}} & \text{ if $i=j$, } \\[2ex]
                                               \displaystyle \frac{x_i}{(1-x_{d+1})^2} & \text{ if $j=d+1$,} \\[2ex]
                                               0 & \text{ otherwise. }
                                            \end{cases} $$
                                     
\noindent 
(Note $1\leq i \leq d$ while $1\leq j \leq d+1$.)
This translates to:

$$ g_i'(t) = \frac{f_i'(t)}{1-f_{d+1}(t)} + 
\frac{f_i(t) \cdot f_{d+1}'(t)}{\left(1-f_{d+1}(t)\right)^2} 
= 
\frac{f_i'(t) \cdot (1-f_{d+1}(t)) + f_i(t) \cdot f_{d+1}'(t)}{\left(1-f_{d+1}(t)\right)^2},
$$

\noindent
for all $0\leq t \leq 1$ and $1\leq i \leq d$.
We can now compute (suppressing the parameter $t$ everywhere for conciseness):

$$ \begin{array}{rcl}
    \norm{g'}^2 
    & = & 
    \frac{1}{\left(1-f_{d+1}\right)^4} \cdot \left( \sum_{i=1}^d (f_i')^2(1-f_{d+1})^2
    + \sum_{i=1}^d 2 f_i' f_i f_{d+1}' (1-f_{d+1}) 
    + \sum_{i=1}^d f_i^2 (f_{d+1}')^2 \right) \\
    & = & 
    \frac{1}{\left(1-f_{d+1}\right)^4} \cdot \left( \sum_{i=1}^d (f_i')^2(1-f_{d+1})^2
    + \left[ \sum_{i=1}^d f_i^2 \right]' f_{d+1}' (1-f_{d+1}) 
    + \sum_{i=1}^d f_i^2 (f_{d+1}')^2 \right) \\
    & = &
    \frac{1}{\left(1-f_{d+1}\right)^4} \cdot \left( \sum_{i=1}^d (f_i')^2(1-f_{d+1})^2
    + \left[ 1-f_{d+1}^2 \right]' f_{d+1}' (1-f_{d+1}) 
    + (1-f_{d+1}^2) (f_{d+1}')^2 \right) \\
    & = & 
    \frac{1}{\left(1-f_{d+1}\right)^4} \cdot \left( \sum_{i=1}^d (f_i')^2(1-f_{d+1})^2
    - 2 f_{d+1} (f_{d+1}')^2 (1-f_{d+1}) 
    + (1-f_{d+1}^2) (f_{d+1}')^2 \right) \\
    & = & 
    \frac{1}{\left(1-f_{d+1}\right)^4} \cdot \left( \sum_{i=1}^d (f_i')^2(1-f_{d+1})^2
    + (f_{d+1}')^2 (1-2f_{d+1} + f_{d+1}^2) \right) \\
    & = & \frac{1}{\left(1-f_{d+1}\right)^4} \cdot \left( \sum_{i=1}^{d+1} (f_i')^2(1-f_{d+1})^2
    \right) \\
    & = & \frac{1}{\left(1-f_{d+1}\right)^2} \cdot \norm{f'}^2.
   \end{array}
$$
\end{proof}

\begin{proofof}{Lemma~\ref{lem:metric}}
We use the previous lemma.
In order to prove part~\ref{itm:metricgenlb}, suppose that 
we chose $f$ in such a way that $g$ traces the line-segment between $\pi(a)$ and 
$\pi(b)$. (That is, $f(t) := \pi^{-1}\left( (1-t)\pi(a) + t\pi(b) \right)$ so that $g(t) = (1-t)\pi(a) + t\pi(b)$.)
We have 

$$ \norm{\pi(a)-\pi(b)} 
= \text{length}(g) = \int_0^1 \norm{g'(t)}\dd{t}  
\geq \frac12 \cdot \int_0^1 \norm{f'(t)}\dd{t} = \frac12 \cdot \text{length}(f) 
\geq \frac12 \cdot \dist(a,b), $$

\noindent
where in the first inequality we use that $-1 \leq f_{d+1}(t) < 1$ for all $t$, and 
in the last inequality we use that $f$ traces some curve connecting $a,b$ that 
stays inside $S^d$ -- hence is at least as long as the geodesic.
This establishes part~\ref{itm:metricgenlb}.

Now suppose that $a,b \in \kap( (0,\dots,0,-1),\delta)$ for some small $\delta>0$, to be determined.
We consider the situation where $f : [0,1] \to S^d \setminus \{(0,\dots,0,1)\}$ traces 
the geodesic from $a$ to $b$ (and again $g := \pi \circ f$).
We have $\dist(a,b) < 2\delta$, so that $f$ must stay inside $\kap((0,\dots,0,-1),3\delta)$.
In particular $f_{d+1}(t) < -\cos(3\delta)$ for all $0\leq t\leq 1$.
Arguing similarly to before, we find

$$ \begin{array}{rcl} 
\norm{\pi(a)-\pi(b)} 
& \leq & \displaystyle 
\text{length}(g)  = \int_0^1 \norm{g'(t)}\dd{t}  \leq 
\left(\frac{1}{1+\cos(3\delta)}\right) \cdot \int_0^1 \norm{f'(t)}\dd{t} \\[2ex]
& = & \displaystyle 
\left(\frac{1}{1+\cos(3\delta)}\right) \cdot \text{length}(f) =
\left(\frac{1}{1+\cos(3\delta)}\right) \cdot \dist(a,b). 
\end{array} $$

\noindent 
Since $\lim_{\delta\searrow 0} \cos(3\delta) = 1$, we can and do assume that we chose 
$\delta>0$ small enough in the beginning to ensure $1/(1+\cos(3\delta)) < \frac12+\eps$.
This concludes the proof of~\ref{itm:metricub12}.
\noindent
For~\ref{itm:metricgenub} we are going to write $ \widehat{x}=(x_1,\dots,x_d) $ for the first $ d $ coordinates to get 
	\begin{equation*}
		\pi(x)-\pi(y)=\frac{\widehat{x}-\widehat{y}}{1-x_{d+1}}+\widehat{y}\bigg(\frac{1}{1-x_{d+1}}-\frac{1}{1-y_{d+1}}\bigg)
	\end{equation*}
	and then use the triangle inequality to obtain
	\begin{equation*}
		\norm{\pi(x)-\pi(y)}\leq \frac{\norm{\widehat{x}-\widehat{y}}}{1-x_{d+1}}+\norm{\widehat{y}}\frac{|x_{d+1}-y_{d+1}|}{(1-x_{d+1})(1-y_{d+1})}.
	\end{equation*}
Now, if $  x_{d+1},y_{d+1}\leq t $ and $ \norm{\widehat{y}}\leq \norm{y}\leq 1 $ we can further estimate
\begin{equation*}
	\begin{split}
		\norm{\pi(x)-\pi(y)}\leq \frac{\norm{\widehat{x}-\widehat{y}}}{1-t}+\frac{\norm{x-y}}{(1-t)^2}\leq \frac{2-t}{(1-t)^2}\norm{x-y}\leq \frac{2-t}{(1-t)^2}\dists(x,y),
	\end{split}
\end{equation*}
as the chordal distance between $ x $ and $ y $ is less than the spherical distance.
\end{proofof}
\section{Proof of \cref{lem:piunif}.}\label{proof:piunif}
\begin{proofof}{\cref{lem:piunif}}
First, we are going to show that the density of $ H\coloneqq U_{d+1} $ satisfies
	\begin{equation*}
		f_{H}(h)=\frac{d\kappa_d}{(d+1)\kappa_{d+1}}(1-h^2)^{(d-2)/2}\indicator{-1\leq h\leq 1}.
	\end{equation*}
	For $ Y $ being distributed uniformly on the $ d+1 $ dimensional unit ball $ B^{d+1} $ it holds that $ U\stackrel{d}{=}Y/\norm{Y} $. We then have for all $ h\geq 0 $
	\begin{equation*}
		\begin{split}
			\Pro(H\leq h)&=\Pro(Y\in B^{d+1}\cap\{x_{d+1}\leq h \}\setminus \{\lambda x :x\in B^{d+1}\cap \{x_{d+1}= h\},0\leq \lambda\leq 1 \})
			\\&=\frac{1}{\kappa_{d+1}}\bigg(\int_{-1}^{h}(1-s^2)^{d/2}\kappa_d\,\text{d}s-\int_{0}^{h}\kappa_d(s\sqrt{1-h^2}/h)^d\,\text{d}s\bigg)\\&=\frac{\kappa_d}{\kappa_{d+1}}\bigg(\int_{-1}^{h}(1-s^2)^{d/2}\,\text{d}s-\frac{1}{d+1}h(1-h^2)^{d/2}\bigg).
		\end{split}
	\end{equation*}
Using this we can obtain the density by differentiating 
\begin{equation*}
    \begin{split}
        f_H(h)&=\frac{\text{d}}{\text{d} h}\Pro(H\leq h)=\frac{\kappa_d}{\kappa_{d+1}}\bigg((1-h^2)^{d/2}-\frac{1}{d+1}(1-h^2)^{d/2}+\frac{d}{d+1}h^2(1-h^2)^{(d-2)/2}\bigg)
			\\&=\frac{d\kappa_d}{(d+1)\kappa_{d+1}}(1-h^2)^{(d-2)/2}.
    \end{split}
\end{equation*}
Our claim for $ f_{H}(h)$ for arbitrary $ h $ follows by symmetry, as $ f_{H}(h)=f_{H}(-h) $. We can now relate the height $ H $ of $ U $ to the distance from the origin after projecting, namely $ R\coloneqq \norm{Z} $. 
From \cref{fig:stereorel}

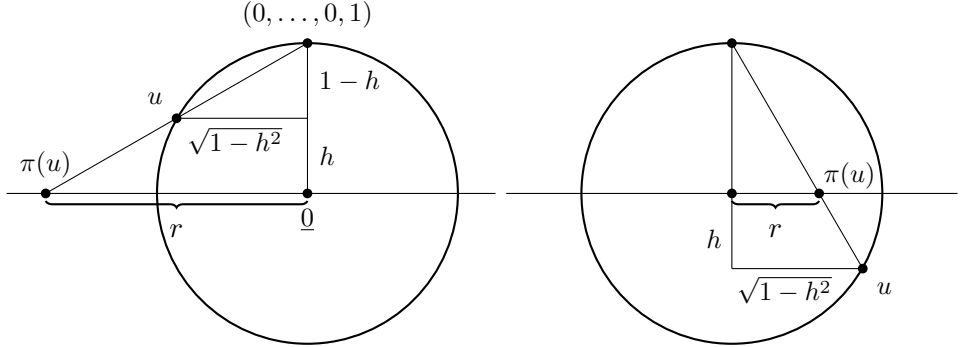
\begin{figure}[h!]
\centering
\begin{tikzpicture}[scale=2,
	axis/.style={very thin, gray,->},
	hidden/.style={gray, dashed},
	point/.style={fill=black,circle,inner sep=1.3pt},
	projline/.style={->,red!70},
	every label/.style={font=\small}
	]

	\draw[thick] (0,0) circle (1cm);
	
	\node[point,label=above:{$(0,\dots,0,1)$}] at (0,1) {};
	\node[point,label=below:{$\typ$}] at (0,0) {};
	\node[point,label=north west:{$u$}] at (-0.87,0.5) {};
	\node[point,label=above:{$\pi(u)$}] at (-1.74,0) {};
	\draw  (-2,0)--(1.25,0);
	\draw  (0,0)--(0,1);
	\draw (-1.74,0)--(0,1);
	\draw (-0.87,0.5)--(0,0.5);
	
	\node[anchor=east] at (0.55,0.75) {\small$1-h$};
	\node[anchor=east] at (0.25,0.25) {\small$h$};
	\node[anchor=east] at (-0.1,0.35) {\small$\sqrt{1-h^2}$};
	 \draw [decorate,decoration={brace,mirror}, thick]
	(-1.74,-0.05) -- (0,-0.05) node[midway,below=5pt] {$r$};
\end{tikzpicture}
\begin{tikzpicture}[scale=2,
	axis/.style={very thin, gray,->},
	hidden/.style={gray, dashed},
	point/.style={fill=black,circle,inner sep=1.3pt},
	projline/.style={->,red!70},
	every label/.style={font=\small}
	]
	
	\draw[thick] (0,0) circle (1cm);
	
	\node[point] at (0,1) {};
	\node[point] at (0,0) {};
	\node[point,label=south east:{$u$}] at (0.87,-0.5) {};
	\node[point,label={[label distance=-5pt]above right:{$\pi(u)$}}] at (0.58,0) {};
	\draw  (-1.5,0)--(1.5,0);
	\draw  (0,1)--(0,-0.5);
	\draw (0.87,-0.5)--(0,1);
	\draw (0.87,-0.5)--(0,-0.5);
	
	\node[anchor=east] at (-0.0,-0.3) {\small$h$};
	\node[anchor=south] at (0.35,-0.8) {\small$\sqrt{1-h^2}$};
	\draw [decorate,decoration={brace,mirror}, thick]
	(0,-0.05) -- (0.58,-0.05) node[midway,below=5pt] {$r$};
\end{tikzpicture}
\caption[Stereographic projection on $S^2$]{%
	Relation between $ H $ and $ R $ for $ r>1 $ (left) and $ r\leq 1 $ (right).%
}
\label{fig:stereorel}
\end{figure}

 we see that $ 1/r=(1-h)/\sqrt{1-h^2}=\sqrt{(1-h)/(1+h)} $, as the triangle with base $ r $ and height $ 1 $ is similar to the one with base $ \sqrt{1-h^2} $ and height $ 1-h $. Rearranging terms yields the relation
\begin{equation*}
	h(r)=\frac{r^2-1}{r^2+1}.
\end{equation*} 
By symmetry of the sphere it holds that $ f_{Z} $ only depends on $ R $, up to measure zero, meaning we can abuse the notation and write $f_Z(s)$ for $s\geq 0$. Thus, we can compute by making the change to $ d $-dimensional polar coordinates
\begin{equation*}
	\Pro(R\leq r)=\int_{0}^{r}f_{Z}(s)s^{d-1}d\kappa_d\,\text{d}s=\Pro(H\leq h(r))=\int_{-1}^{h(r)}\frac{d\kappa_d}{(d+1)\kappa_{d+1}}(1-s^2)^{(d-2)/2}\,\text{d}s.
\end{equation*}
After differentiating we obtain
\begin{equation*}
	\begin{split}
		f_{Z}(r)r^{d-1}d\kappa_d&=f_H(h(r))h'(r)=\frac{d \kappa_d}{(d+1)\kappa_{d+1}}(1-h(r)^2)^{(d-2)/2}\bigg[\frac{r^2-1}{r^2+1}\bigg]'
		\\&=\frac{d \kappa_d}{(d+1)\kappa_{d+1}}(1-h(r)^2)^{(d-2)/2}\frac{4r}{(1+r^2)^2}.
	\end{split}
\end{equation*}
Since additionally it holds that 
\begin{equation*}
	1-h(r)^2=1-\bigg(\frac{r^2-1}{r^2+1}\bigg)^2=\frac{4r^2}{(1+r^2)^2},
\end{equation*}
we can rearrange terms to find
\begin{equation*}
	f_{Z}(r)=\frac{1}{(d+1)\kappa_{d+1}}\bigg(\frac{2}{1+r^2}\bigg)^d,
\end{equation*}
as required.
\end{proofof}

\section{Proof of Lemma~\ref{dependentperco}\label{sec:depperco}}

\begin{proofof}{Lemma~\ref{dependentperco}}
    We let $\rho\coloneqq \eps^2/(4d)$ and consider bond percolation on $\ZZ^2$ with the usual dual graph definition. We choose $p$ so close to 1 that with $q\coloneqq 1-p$ we have 
    \begin{equation*}
        \alpha (q) \coloneqq\Pro( \text{there is a dual open circuit surrounding }(0,0))\leq \rho.
    \end{equation*}
    This is possible because of the famous Peierls contour argument that gives $\alpha(q)\leq 4\sum_{\ell\geq 4} (3q)^\ell$ by counting self avoiding paths of length $\ell$. Fix some $\beta>0$ such that $1+\beta \log(3q)<0$. For each $m$ set $r\coloneqq \ceil {\beta\log m}$ and 
    \begin{equation*}
        \Delta_m\coloneqq\{x\in \Lambda_m\,:\, \inf_{y\in \partial\Lambda_m} \norm{x-y}<r\},
    \end{equation*}
    which yields $\Delta_m=O(r m^{d-1})=o(|\Lambda_m|)$, so for all large $m$ it holds $|\Delta_m|\leq \rho|\Lambda_m|$. Now fix any $2$ coordinate plane $\Pi\subset \Lambda_m$ and 
    consider connectivity inside $\Pi$ using only edges of $\Pi$. If $u,v\in \Pi\setminus\Delta_m $ and $u$ is not connected to $v$ inside $\Pi$, then one of the following three states must occur: there is a dual open circuit surrounding $v$, or there is a dual open circuit surrounding $u$, or there is a dual open path attached to the dual boundary of $\Pi$ that reaches depth at least $r$ away from the boundary. The first two possibilities have probability $\leq \alpha(q)\leq \rho$ each. For the third we can again bound the self avoiding paths that yields constants $C_1,C_2>0$ with
    \begin{equation*}
    \begin{split}
        \Pro(\text{there is dual boundary path reaching depth}\geq r)&\leq  C_1m\sum_{\ell\geq r}(3q)^\ell \leq C_2 m(3q)^r \\&\leq C_2m^{1+\beta\log(3q)}.
    \end{split}
    \end{equation*}
    Thus, as the right hand side is less than $\rho$ for all large enough $m$, we have
    \begin{equation*}
        \Pro(u\text{ is not connected to }v \text{ in }\Pi)\leq 3\rho.
    \end{equation*}
     Now fix some $z\in \Lambda_m\setminus\Delta_m$. Note that by fixing two coordinates in each step, we can move from $0$ to $z$ by traversing $d-1$ planes. By symmetry of the planes, the error probability that inside a plane the connection to the specific coordinates of $z$ does not occur is upper bounded by $3\rho$. This yields
    $\Pro(0\nocon z)\leq 3\rho (d-1)\leq 3d\rho$, where the connection is meant to be inside $\Lambda_m$. Defining  $D_m\coloneqq |\Lambda_m\setminus C_m(\typ)|$ we can compute 
    \begin{equation*}
    \begin{split}
         \Ex D_m=\sum_{z\in \Lambda_m}\Pro(0\nocon z)\leq |\Delta_m|+\sum_{z\in \Lambda_m\setminus\Delta_m}3d\rho \leq \rho |\Lambda_m|+3d\rho |\Lambda_m|\leq \eps^2|\Lambda_m|.
    \end{split}
    \end{equation*}
    Therefore by Markov
    \begin{equation*}
        \Pro(D_m\geq \eps |\Lambda_m|)\leq \frac{\Ex D_m}{\eps |\Lambda_m|}\leq \eps,
    \end{equation*}
    which is equivalent to our statement.
\end{proofof}

\end{document}